\documentclass[11pt,letterpaper,reqno]{amsart}
\usepackage[foot]{amsaddr}
\usepackage[margin=1in]{geometry}
\usepackage{amsmath,amsthm,amsfonts,amssymb,mathrsfs}
\usepackage{hyperref,color,enumerate}
\usepackage{graphicx}

\newtheorem{thm}{Theorem}[section]

\newtheorem{cor}[thm]{Corollary}
\newtheorem{lem}[thm]{Lemma}
\newtheorem{prop}[thm]{Proposition}

\theoremstyle{definition}
\newtheorem{asp}[thm]{Assumption}

\newtheorem{defi}[thm]{Definition}
\newtheorem{rek}[thm]{Remark}

\def\R{\mathbb{R}}
\def\C{\mathbb{C}}
\def\D{\mathbb{D}}

\def\E{\mathbb{E}}
\def\P{\mathbb{P}}
\def\vec{\operatorname{vec}}
\def\diag{\operatorname{diag}}
\def\poly{\operatorname{poly}}
\def\cA{\mathcal{A}}
\def\cB{\mathcal{B}}
\def\cD{\mathcal{D}}
\def\cE{\mathcal{E}}
\def\cF{\mathcal{F}}
\def\cG{\mathcal{G}}
\def\cH{\mathcal{H}}
\def\cK{\mathcal{K}}
\def\cL{\mathcal{L}}
\def\cM{\mathcal{M}}
\def\cP{\mathcal{P}}
\def\cQ{\mathcal{Q}}

\def\cW{\mathcal{W}}
\def\cX{\mathcal{X}}
\def\cY{\mathcal{Y}}
\def\cZ{\mathcal{Z}}
\def\cR{\mathcal{R}}
\def\op{\mathrm{op}}
\def\a{\mathsf{a}}
\def\b{\mathsf{b}}
\def\d{\mathsf{d}}
\def\g{\mathsf{g}}
\def\h{\mathsf{h}}
\def\m{\mathsf{m}}
\def\p{\mathsf{p}}
\def\q{\mathsf{q}}
\def\r{\mathsf{r}}
\def\s{\mathsf{s}}
\def\st{\mathsf{t}}
\def\su{\mathsf{u}}
\def\sv{\mathsf{v}}
\def\w{\mathsf{w}}
\def\sU{\mathsf{U}}
\def\sV{\mathsf{V}}
\def\x{\mathsf{x}}
\def\y{\mathsf{y}}
\def\z{\mathsf{z}}
\def\A{\mathsf{A}}
\def\B{\mathsf{B}}
\def\Y{\mathsf{Y}}

\def\e{\mathbf{e}}
\def\i{\mathbf{i}}
\def\j{\mathbf{j}}
\def\u{\mathbf{u}}
\def\v{\mathbf{v}}
\def\bh{\mathbf{h}}
\def\bx{\mathbf{x}}

\def\btheta{\boldsymbol{\theta}}
\def\bxi{\boldsymbol{\xi}}

\def\1{\mathbf{1}}
\def\0{\mathbf{0}}
\def\Tr{\operatorname{Tr}}
\def\NC{\operatorname{NC}}
\def\Oprec{O_{\prec}}

\def\spec{\mathrm{spec}}
\def\supp{\mathrm{supp}}
\def\val{\mathrm{val}}
\def\Id{\mathrm{Id}}

\newcommand{\pnorm}[2]{\left\| #1\right\|_{#2}}
\DeclareMathOperator*{\argmin}{arg\,min}

\title{Kronecker-product random matrices and a matrix least squares problem}

\author{Zhou Fan}
\author{Renyuan Ma}
\email{zhou.fan@yale.edu, jack.ma.rm2545@yale.edu}
\address{Department of Statistics and Data Science, Yale University}

\begin{document}

\begin{abstract}
We study the eigenvalue distribution and resolvent of a Kronecker-product random matrix model $A \otimes I_{n \times n}+I_{n \times n} \otimes B+\Theta \otimes \Xi \in \mathbb{C}^{n^2 \times n^2}$, where $A,B$ are independent Wigner matrices and $\Theta,\Xi$ are deterministic and diagonal. For fixed spectral arguments, we establish a quantitative approximation for the Stieltjes transform by that of an approximating free operator, and a diagonal deterministic equivalent approximation for the resolvent. We further obtain sharp estimates in operator norm for the $n \times n$ resolvent blocks, and show that off-diagonal resolvent entries fall on two differing scales of $n^{-1/2}$ and $n^{-1}$ depending on their locations in the Kronecker structure.

Our study is motivated by consideration of a matrix-valued least-squares optimization problem $\min_{X \in \mathbb{R}^{n \times n}} \frac{1}{2}\|XA+BX\|_F^2+\frac{1}{2}\sum_{ij} \xi_i\theta_j x_{ij}^2$ subject to a linear constraint. For random instances of this problem defined by Wigner inputs $A,B$, our analyses imply an asymptotic characterization of the minimizer $X$ and its associated minimum objective value as $n \to \infty$.
\end{abstract}

\maketitle


\section{Introduction}\label{Introduction}

In recent years, high-dimensional probabilistic analyses have yielded important
insights into the exact asymptotic behavior of many
optimization problems with random data. We mention, as several examples,
analyses of ridge regression
\cite{karoui2013asymptotic,dicker2016ridge,dobriban2018high} with possibly non-linear random
features \cite{d2020double,hastie2022surprises,mei2022generalization} 
and/or in kernelized domains
\cite{mei2022generalizationkernel,xiao2022precise}
using asymptotic random matrix theory, and analyses of optimization problems
arising in contexts of non-linear regression
\cite{bayati2011lasso,el2013robust,thrampoulidis2015regularized,donoho2016high,el2018impact,thrampoulidis2018precise},
classification
\cite{sur2019modern,montanari2019generalization,liang2022precise,deng2022model}, and
variational Bayesian inference \cite{fan2021tap,qiu2023tap,celentano2023mean}
using Approximate Message Passing algorithms, Gaussian comparison and
interpolation arguments, and cavity-method techniques.
In most such examples, the behavior of the optimizer $\hat \bx \in \R^n$ in the
limit $n \to \infty$ is characterized by a system of scalar fixed-point
equations, derived via mean-field approximations over an interaction matrix
having a number of random elements much larger than the
dimension $n$ of the optimization variable.

Our current work is motivated by the study of large-$n$ asymptotics for a
different type of matrix-valued optimization problem, taking the form
\begin{equation}\label{eq:QPintro}
\min_{X \in \R^{n \times n}}
\frac{1}{2}\|XA+BX\|_F^2+\frac{1}{2}\sum_{i,j=1}^n \xi_i\theta_j x_{ij}^2
\qquad \text{ subject to } \frac{1}{n}\v^* X\u=1.
\end{equation}
Here, $\btheta,\bxi,\u,\v \in \R^n$ are deterministic ridge-regularization and
linear constraint parameters, and we will
study a setting of random inputs given by
independent Wigner matrices $A,B \in \R^{n \times n}$.
Notably, the optimization variable $X=(x_{ij})_{i,j=1}^n$ (with
$x_{ij}$ denoting the entries of $X$)
has dimension comparable to
$A$ and $B$. This problem (\ref{eq:QPintro}) may be written
equivalently in terms of the vectorization $\bx=\vec(X) \in \R^{n^2}$ and
diagonal matrices $\Theta=\diag(\btheta)$ and $\Xi=\diag(\bxi)$, as
\[\min_{\bx \in \R^{n^2}}
\frac{1}{2}\|(A \otimes I+I \otimes B)\bx\|_2^2
+\frac{1}{2}\,\bx^* (\Theta \otimes \Xi)\bx
\quad \text{ subject to } \frac{1}{n}(\u \otimes \v)^* \bx=1.\]
Such a problem is paradigmatic of a broader class of nonlinear problems/models
having a Kronecker-product structure: For motivation, let us mention
the matrix spin glass
model\footnote{We would like to thank Justin Ko for bringing a model similar to
(\ref{eq:matrixspinglass}) to our attention.} 
\begin{align}
p(X)&=\frac{1}{\mathcal{Z}}
\exp(\Tr AX^*X+\Tr BXX^*)\label{eq:matrixspinglass}\\
&=\frac{1}{\mathcal{Z}}\exp(\vec(X)^* (A \otimes I+I \otimes B) \vec(X)),
\qquad X \in \{\pm 1\}^{n \times n}\notag
\end{align}
defined by independent GOE coupling matrices $A,B \in \R^{n \times n}$,
and the optimization problem
\begin{equation}\label{eq:QPgraphmatching}
\min_{X \in \R^{n \times n}}
\|XA+BX\| \quad \text{ subject to }
\sum_{i=1}^n x_{ij}=\sum_{j=1}^n x_{ij}=1,\; x_{ij} \geq 0 \text{ for all }
i,j=1,\ldots,n
\end{equation}
defined by (possibly entrywise correlated) Wigner matrices $A,B \in
\R^{n \times n}$. The former model (\ref{eq:matrixspinglass}) describes a
disordered spin system on the lattice, with $A,B$ representing couplings for the
row and column inner-products of $X$. The latter
problem (\ref{eq:QPgraphmatching}) for various choices of matrix norm
$\|XA+BX\|$ corresponds to popular convex relaxations of combinatorial
optimization problems over permutation matrices $X$ that arise in random graph
matching \cite{almohamad1993linear,zaslavskiy2008path,aflalo2015convex}.

As a step towards developing techniques and insight for asymptotic analyses
of these types of Kronecker-structured models, in this work we carry out an
analysis of the simpler linear problem (\ref{eq:QPintro}) using
random matrix theory methods. We will establish a deterministic approximation
with $\Oprec(n^{-1/2})$ error for the value of the objective (\ref{eq:QPintro})
at its minimizer $\widehat{X}$, as well as for the value of
$n^{-1}{\v'}^*\widehat{X}\u'$ for arbitrary deterministic test vectors $\u',\v' \in \R^n$.
These results are closely related to a deterministic equivalent
approximation for the resolvent of a Kronecker-product random matrix
\begin{equation}\label{eq:matrixmodelintro}
Q=A \otimes I_{n \times n}+I_{n \times n} \otimes B+\Theta \otimes \Xi \in
\C^{n^2 \times n^2},
\end{equation}
which we will refer to as the ``Kronecker deformed Wigner model'' (in analogy
with the deformed Wigner model $A+\Theta$
studied classically in random matrix theory
\cite{pastur1972spectrum,capitaine2011free,LeeSchnelli2013Local,lee2015edge,lee2016bulk,knowles2017anisotropic}).
A second main focus of our work is to establish sharp quantitative
estimates for the resolvent $G=(Q-zI)^{-1}$ of this model at global
spectral scales, i.e.\ for fixed spectral parameters $z \in \C^+$. Writing
$G_{ij}=(\e_i \otimes I)^*G(\e_j \otimes I) \in \C^{n \times n}$
and $G_{ij,\alpha\beta}=(\e_i \otimes \e_\alpha)^*G(\e_j \otimes \e_\beta) \in
\C$ for the blocks and entries of this resolvent, we will show the
operator-norm estimates
\begin{equation}\label{eq:introopnorm}
\|G_{ii}^{-1}-G_{jj}^{-1}-(\theta_i-\theta_j)\Xi\|_\op \prec n^{-1/2},
\qquad \|G_{ij}\|_\op \prec n^{-1/2} \text{ for } i \neq j,
\end{equation}
the entrywise estimates
\begin{equation}\label{eq:introentrywise}
G_{ii,\alpha\alpha}-[G_0]_{ii,\alpha\alpha} \prec n^{-1/2},
\qquad G_{ii,\alpha\beta} \prec n^{-1/2} \text{ for } \alpha \neq \beta,
\qquad G_{ij,\alpha\beta} \prec n^{-1} \text{ for } \alpha \neq \beta, i \neq
j,
\end{equation}
and the bilinear-form estimates
\begin{equation}\label{eq:introbilinear}
(\u \otimes \v)^*[G-G_0](\u' \otimes \v') \prec n^{-1/2}
\text{ for deterministic unit vectors } \u,\v,\u',\v' \in \C^n,
\end{equation}
where $G_0 \in \C^{n^2 \times n^2}$ is a diagonal deterministic-equivalent
matrix. We remark that although this model (\ref{eq:QPintro})
is a ``toy'' setting that simplifies many of the additional complexities of
nonlinear models such as (\ref{eq:matrixspinglass}) and
(\ref{eq:QPgraphmatching}), some high-level aspects of this model may be
suggestive of properties to be expected also in these nonlinear models,
including a characterization of the large-$n$ limit
by a pair of fixed-point equations in an operator algebra rather than over the
scalar field, and a ``two-tiered'' mean-field structure as reflected in
(\ref{eq:introentrywise}), which arises from separate mean-field
approximations over $A$ and $B$.

A special case of the optimization problem (\ref{eq:QPintro}) with
$\Theta=\Xi=\eta I_{n \times n}$ (and correlated Wigner inputs $A$ and $-B$) was
previously analyzed in \cite{fan2023spectral}, in the context of the graph
matching application. We review a central idea of
this previous analysis in Appendix \ref{appendix:contour}, which,
however, is special to a commutative setting and does not extend more generally.
Here, we instead follow an alternative approach of a two-stage Schur-complement
analysis of the resolvent, in each stage applying, in an operator-algebra
setting, ideas around Dyson fixed-point equations and
fluctuation averaging techniques developed for Wigner-type models in
\cite{ErdosYauYin2012BulkUniversality,ErdosYauYin2012Rigidity,ErdosEtAl2013Local,LeeSchnelli2013Local}
and for $\C^{k \times k} \otimes \C^{n \times n}$-valued Kronecker
matrix models with fixed dimension $k$ in
\cite{AltErdosKrugerNemish2019Location,ErdosKrugerNemish2020Polynomials}. The
application of these methods in our setting of $k(n)=n$ seems to require new
ideas, even to obtain optimal quantitative estimates for fixed
spectral parameters on the global scale, and we discuss this further
in Section \ref{subsec:proofideas} below.

Recently, a breakthrough line of work has obtained sharp operator norm estimates
for polynomial matrix models in Kronecker-product spaces with growing first
dimension $k \equiv k(n)$
\cite{CollinsGuionnetParraud2022,Parraud2022OperatorNormHaar,Parraud2023Asymptotic,bandeira2023matrix,BelinschiCapitaine2022Preprint,Parraud2023Preprint,bordenave2023norm}.
The authors of \cite{CollinsGuionnetParraud2022} developed a new approach for
estimating expected traces of smooth functions of polynomials in
\begin{equation}\label{eq:polynomialmodel}
(A_1 \otimes I_{n \times n},\ldots,A_r \otimes I_{n \times n},
I_{k \times k} \otimes B_1,\ldots,I_{k \times k} \otimes B_s)
\end{equation}
for deterministic matrices $A_1,\ldots,A_r \in \C^{k \times k}$ and 
independent GUE matrices $B_1,\ldots,B_s \in \C^{n \times n}$,
via an interpolation between the GUE and free semicircular variables and a
differential calculus using Gaussian integration-by-parts and the
semicircular Schwinger-Dyson equation. These results implied
an estimate with $O(\frac{k^2}{n^2\poly(\Im z)})$-error for the expectation of
the Stieltjes transform, and a strong convergence result (i.e.\ convergence in
operator norm) when $k \ll n^{1/3}$.
This was extended into a full asymptotic expansion in
\cite{Parraud2023Asymptotic} and from GUE to Haar unitary matrices in
\cite{Parraud2022OperatorNormHaar,Parraud2023Preprint}, with
\cite{Parraud2023Preprint} deducing strong convergence for Kronecker-product
polynomials in Haar
unitary and deterministic matrices having first dimension
$k \leq n/\poly(\log n)$. An analogous
strong convergence statement in the Gaussian setting was shown as part of the
general results of \cite{bandeira2023matrix}, using a different
interpolation idea. The authors of
\cite{BelinschiCapitaine2022Preprint} showed strong convergence for
(\ref{eq:polynomialmodel}) when $k=n$ and $A_1,\ldots,A_r,B_1,\ldots,B_s$ are
all independent GUE matrices, via a mapping to unitary
matrices, an extension of the asymptotic expansion in
\cite{Parraud2023Asymptotic}, and a precise analysis of the
large-$|z|$ expansion of the expected Stieltjes transform. In
\cite{bordenave2023norm}, strong convergence for Kronecker-product
polynomials of Haar-unitary matrices up to dimensions $k \leq \exp(n^\alpha)$
was shown using a different non-backtracking high trace method.

The focus of our work is a bit different from the above,
as we will not address this question of strong convergence for our model, but
instead study the detailed structure of its resolvent in addition to its
spectral measure. We carry out a more classical analysis using the resolvent
calculus, and for our current purposes, we will also not separate the analysis
of the expectation of the resolvent from its fluctuations. (In particular, we
will not establish an estimate for the expected Stieltjes transform that is more
precise than the $\Oprec(n^{-1})$ scale of its fluctuations, as would be
needed to show strong convergence.) However, we highlight that our analyses
apply directly to non-Gaussian Wigner matrices, and also yield sharp
operator-norm estimates of the form (\ref{eq:introopnorm}) for the $n \times n$ 
resolvent blocks, which may be more difficult to obtain via arguments based
solely on large-$|z|$ expansions and analytic continuation ideas.

\subsection{Proof ideas}\label{subsec:proofideas}

We present here the high-level ideas of the arguments, deferring to Sections
\ref{sec:model} and \ref{sec:preliminaries} a more detailed description of the
model and notations. Consider first the model
\[Q=A \otimes I+I \otimes B+\Theta \otimes \Xi \in \C^{n^2 \times n^2}\]
and denote its resolvent $G=G(z)=(Q-zI)^{-1}$. The strategy will be
to perform a two-stage Schur-complement analysis of this resolvent, first over
the randomness of $A$, and then over $B$.

In the first stage,
conditioning on $B$, the independence structure of $A \otimes I$
suggests a Schur-complement analysis at the level of its $n \times n$ blocks:
Denoting $G_{ij}=(\e_i \otimes I)^* G(\e_j \otimes I) \in \C^{n \times n}$ and
applying standard resolvent identities, we have
\begin{equation}\label{eq:introGiiGij}
G_{ii}=\Big(a_{ii}I+B+\theta_i\Xi-zI-\sum_{r,s}^{(i)}
a_{ir}G_{rs}^{(i)}a_{si}\Big)^{-1},
\qquad G_{ij}={-}G_{ii}\sum_r^{(i)} a_{ir}G_{rj}^{(i)},
\end{equation}
where $G_{rs}^{(i)}$ is the version of $G_{rs}$ setting to 0
the $i^\text{th}$ row and column of $A$, and $\sum_{r}^{(i)}$ is the summation 
over indices $r \neq i$.
Applying concentration of the quadratic form in 
the first expression and
averaging over $i=1,\ldots,n$, we will eventually obtain an approximate
fixed-point relation for the partial trace
\[(n^{-1}\Tr \otimes I)G=\frac{1}{n}\sum_{i=1}^n G_{ii}
=\sum_{i=1}^n \Big(B+\theta_i\Xi-zI-(n^{-1}\Tr
\otimes I)G \Big)^{-1}+\Oprec(n^{-1}).\]
Stability of this equation will imply
\begin{equation}\label{eq:introMB}
(n^{-1}\Tr \otimes I)G=M_B+\Oprec(n^{-1}), \quad
\text{ for the fixed point }
M_B=\sum_{i=1}^n \Big(B+\theta_i\Xi-zI-M_B\Big)^{-1}.
\end{equation}

In the second stage, we then carry out the analysis over $B$. As
the implicit dependence of $M_B$ on $B$ in (\ref{eq:introMB}) is rather complex
and makes a direct Schur-complement analysis difficult, we introduce
an explicit operator algebra representation and write
\begin{equation}\label{eq:introMBrepr}
M_B=(\tau \otimes I)\Big[\underbrace{(\a \otimes I+1 \otimes B+\Theta \otimes
\Xi-z\,1 \otimes I)^{-1}}_{=\tilde \g}\Big]
\end{equation}
for a ($n$-dependent) von Neumann algebra $\cA$ with trace $\tau$,
containing a semicircular variable $\a$ and the subalgebra $\C^{n \times n}$
free of $\a$ (Lemma \ref{lem:uniquefixedpoint}).
We study the Stieltjes transform $n^{-2}\Tr G \approx n^{-1}\Tr M_B$ by first taking
$1 \otimes n^{-1}\Tr$ inside $\tau \otimes I$ in (\ref{eq:introMBrepr}),
and applying resolvent identities parallel to (\ref{eq:introGiiGij}),
\begin{equation}\label{eq:introGiiGij2}
\tilde\g_{\alpha\alpha}=\Big(b_{\alpha\alpha}1
+\a+\xi_\alpha\Theta-z1-\sum_{\gamma,\delta}^{(\alpha)}
b_{\alpha\gamma}\tilde\g_{\gamma\delta}^{(\alpha)}b_{\delta\alpha}\Big)^{-1},
\qquad \tilde\g_{\alpha\beta}={-}\tilde\g_{\alpha\alpha}
\sum_\gamma^{(\alpha)} b_{\alpha\gamma}\tilde\g_{\gamma\beta}^{(\alpha)},
\end{equation}
to analyze $(1 \otimes n^{-1}\Tr)\tilde\g$. This yields 
\begin{equation}\label{eq:introMA}
(1 \otimes n^{-1}\Tr)\tilde \g=\m_a+\Oprec(n^{-1}),
\quad \text{ for the fixed point }
\m_a=\sum_{\alpha=1}^n \Big(\a+\xi_\alpha \Theta-z1-\m_a\Big)^{-1},
\end{equation}
and applying this back to (\ref{eq:introMBrepr}) gives an approximation for
$n^{-2}\Tr G$. These arguments also yield, as direct
consequences, estimates in operator norm for the resolvent
blocks $G_{ii},G_{ij} \in \C^{n \times n}$.

In deducing (\ref{eq:introMB}) from (\ref{eq:introGiiGij}) and
(\ref{eq:introMA}) from (\ref{eq:introGiiGij2}), two difficulties arise that do
not occur in usual scalar random matrix models:
\begin{enumerate}
\item Applying non-commutative concentration inequalities to analyze
(\ref{eq:introGiiGij}) and (\ref{eq:introGiiGij2}), several of the
terms controlling the scale of fluctuations cannot be bounded
spectrally by the resolvent. For example, a non-commutative Khintchine-type
inequality gives for $G_{ij}$ in (\ref{eq:introGiiGij})
\begin{equation}\label{eq:introkhintchine}
\E\|G_{ij}\|_p^p \prec \frac{1}{\sqrt{n}}
\max\bigg\{\E\bigg\|\bigg(\sum_r
G_{rj}^{(i)}G_{rj}^{(i)*}\bigg)^{1/2}\bigg\|_p^p,\,
\E\bigg\|\bigg(\sum_r G_{rj}^{(i)*}G_{rj}^{(i)}\bigg)^{1/2}\bigg\|_p^p\bigg\}.
\end{equation}
The second term may be directly controlled by $\|G^*G\|_\op$ and a
spectral argument, but the first term is related instead to the spectrum of the
partial transpose ${G^\st}=\sum_{i,j} E_{ji} \otimes G_{ij}$, for which a naive
bound gives $\|G^\st\|_\op \leq n\|G\|_\op$ \cite{tomiyama1983transpose}.
Applying this naive bound produces a trivial estimate $G_{ij} \prec 1$.
A similar issue arises for the quadratic forms in the expressions of $G_{ii}$
and $\tilde\g_{\alpha\alpha}$ in
(\ref{eq:introGiiGij}) and (\ref{eq:introGiiGij2}).
\item In the infinite-dimensional context of (\ref{eq:introGiiGij2}),
Khintchine-type inequalities only yield estimates in the $L^p$-norms
$\|\cdot\|_p$ for $p<\infty$, rather than a dimension-free estimate in the
operator norm, whereas stability of the fixed-point equation (\ref{eq:introMA})
is most readily established under perturbations that are
bounded in operator norm.
\end{enumerate}

We address these difficulties by first carrying out the analysis for $|z|$
sufficiently large for which the resolvent admits a convergent series
expansion in $z^{-1}$, and furthermore the concentration of errors in
(\ref{eq:introGiiGij}) and (\ref{eq:introGiiGij2}) may be established by
expanding into elementary tensors and
applying scalar concentration term-by-term. Then, using a quantitative version
of the maximum modulus principle, we obtain a weak high-probability
estimate in the operator norm
\[\big\|G_{ii}-\big(B+\theta_i\Xi-zI-(n^{-1}\Tr \otimes I) G\big)^{-1}\big\|_\op
<n^{-\alpha}\]
for any fixed $z \in \C^+$ and a small (sub-optimal) constant $\alpha>0$
depending on $\Im z$
(Lemma \ref{lemma:weakopnormestimates}). Such an estimate and the stability of
the fixed-point equation (\ref{eq:introMB}) is sufficient to deduce
\begin{equation}\label{eq:introweakestimates}
G_{ii}-\underbrace{(B+\theta_i\Xi-zI-M_B)^{-1}}_{=M_i}
\prec n^{-\alpha}, \qquad G_{ij} \prec n^{-\alpha}.
\end{equation}

These estimates (\ref{eq:introweakestimates}) now enable the application of
non-commutative analogues of fluctuation averaging
techniques \cite{erdHos2013averaging}, which we use in conjunction with
an iterative bootstrapping argument to obtain the
optimal estimates for $G_{ii},G_{ij}$ as follows: Writing the
first non-spectral term of (\ref{eq:introkhintchine}) as
\[\frac{1}{n}\sum_r G_{rj}^{(i)}G_{rj}^{(i)*}
=\frac{1}{n}\sum_r \E_r[G_{rj}^{(i)}G_{rj}^{(i)*}]+\frac{1}{n}\sum_r \cQ_r[G_{rj}^{(i)}G_{rj}^{(i)*}]\]
where $\E_r$ is the partial expectation over row/column $r$ of $A$ and
$\cQ_r=1-\E_r$, fluctuation averaging with (\ref{eq:introweakestimates}) as
input shows for the second term
$n^{-1}\sum_r \cQ_r[G_{rj}^{(i)}G_{rj}^{(i)*}] \prec n^{-3\alpha}$.
Applying a resolvent expansion of the first term
$n^{-1}\sum_r \E_r[G_{rj}^{(i)}G_{rj}^{(i)*}]$ then yields
\[\cL_2\bigg(\frac{1}{n}\sum_r G_{rj}^{(i)}G_{rj}^{(i)*}\bigg) \prec
n^{-1}+n^{-3\alpha}\]
for the linear operator
$\cL_2(X)=X-\frac{1}{n}\sum_i M_iXM_i^*$, with $M_i$ defined in
(\ref{eq:introweakestimates}). This type of operator associated to Dyson
fixed-point equations has been studied previously in 
\cite{ajanki2019stability,AltErdosKrugerNemish2019Location}, and we adapt a
Perron-Frobenius argument of \cite{ajanki2019stability} to show quantitative
invertibility of $\cL_2$ for any fixed $z \in \C^+$. In our context, we require
invertibility of $\cL_2$ in the $L^p$-norm for each $p \in [1,\infty)$, and in
infinite-dimensional settings where $\cL_2$ may not have an exact
Perron-Frobenius eigenvector --- we thus carry out this analysis in the
$L^1$-$L^\infty$ duality rather than the Hilbert-space setting of
\cite{ajanki2019stability}, and appeal to the Riesz-Thorin interpolation
to obtain invertibility for all $p$ (Lemma
\ref{lemma:linearmap2}).
This shows $n^{-1}\sum_r G_{rj}^{(i)}G_{rj}^{(i)*} \prec n^{-1}+n^{-3\alpha}$, which
applied to (\ref{eq:introkhintchine}) shows
the implication
\[G_{ii}-M_i \text{ and } G_{ij} \prec n^{-\alpha}
\Longrightarrow G_{ii}-M_i \text{ and }
G_{ij} \prec \max(n^{-1/2},n^{-3\alpha/2}).\]
Iterating this bound gives finally the optimal errors
$G_{ii}-M_i \prec n^{-1/2}$ and $G_{ij} \prec n^{-1/2}$,
and an additional fluctuation averaging
step shows $n^{-1}\sum_i G_{ii}-M_B \prec n^{-1}$ for the partial trace,
as claimed in (\ref{eq:introMB}).

These arguments extend to show quantitative approximations for bilinear forms
of the resolvent $(\u \otimes \v)^* G(\u' \otimes \v')$, with
$\Oprec(n^{-1/2})$ error. To show the estimate
$G_{ij,\alpha\beta} \prec n^{-1}$ for off-diagonal entries, we apply a similar
idea as above, deriving from the resolvent identities and a fluctuation averaging
argument, for each fixed $i \neq j$ and $\alpha \neq \beta$,
\[\cL_1\bigg(\sum_k G_{kj}\e_\beta\e_\alpha^* G_{ik}\bigg) \prec n^{-1/2}\]
where $\cL_1(X)=X-\frac{1}{n}\sum_{i=1}^n M_iXM_i$. Quantitative invertibility
of $\cL_1$ follows from differentiation of the fixed-point equation
(\ref{eq:introMB}), yielding $\sum_k G_{kj}\e_\beta\e_\alpha^* G_{ik} \prec
n^{-1/2}$, and we will deduce from this the estimate
$G_{ij,\alpha\beta} \prec n^{-1}$ (Section \ref{sec:resolvent}).
Finally, results for the least-squares
problem (\ref{eq:QPintro}) are obtained via a similar analysis of a linearized
model (Section \ref{sec:optimization})
\[\begin{bmatrix}
            0 & 1\\
            1 & 0
        \end{bmatrix} \otimes (A \otimes I + I \otimes B) -i
        \begin{bmatrix}
            1 & 0\\
            0 & 0
        \end{bmatrix} \otimes \Theta \otimes \Xi - i\begin{bmatrix}
            0 & 0\\
            0 & 1
        \end{bmatrix}\otimes I \otimes I \in \C^{2 \times 2} \otimes \C^{n
\times n} \otimes \C^{n \times n}.\]

%

\subsection{Notation and conventions}
$(\e_1,\ldots,\e_n) \in \C^n$ are the standard basis vectors,
$E_{ij}=\e_i\e_j^*$ are the coordinate basis elements of $\C^{n \times n}$,
and $I_{n \times n}$ is the identity matrix (omitting the
subscript when the meaning is clear).
For $M \in \C^{n \times n}$, $\|M\|_F=(\sum_{i,j}
|m_{ij}|^2)^{1/2}$ and $\Tr M=\sum_i m_{ii}$ are the Frobenius norm and
(unnormalized) trace. We write $U=\diag(\u)$ for the diagonal matrix with
$\u \in \C^n$ on its diagonal, and $\u=\diag(U)$ for the vector
$(U_{11},\ldots,U_{nn})$ on the diagonal of $U \in \C^{n \times n}$.

For a von Neumann algebra $\cX$, we denote its unit as $1_\cX$ (omitting the
subscript when the meaning is clear).
Scalars $z \in \C$ are identified implicitly as elements of $\cX$ via
$z \mapsto z1_\cX$. We write $\x \geq 0$ if $\x \in \cX$ is self-adjoint and has
nonnegative spectrum.
For $\x,\y \in \cX$ self-adjoint, $\x \geq \y$ means $\x-\y \geq 0$.
We write
\[\Re \x=\tfrac{1}{2}(\x+\x^*), \qquad \Im \x=\tfrac{1}{2i}(\x-\x^*),
\qquad |\x|=(\x^*\x)^{1/2}\]
for the operator real and imaginary parts and absolute value. We denote
\[\cX^+=\{\x \in \cX:\Im \x \geq \epsilon \text{ for some } \epsilon>0\}\]
as the elements with strictly positive imaginary part (not to be confused with
the positive cone $\{\x \in \cX:\x \geq 0\}$).
$\|\x\|_\op$ is the operator norm, and $\|\x\|_p=\phi(|\x|^p)^{1/p}$
is the $L^p$-norm for a given trace $\phi$.

%

\section{Model and main results}\label{sec:model}

\subsection{Kronecker deformed Wigner model}

\begin{asp}\label{assump:Wigner}
\begin{enumerate}[(a)]
\item $A,B \in \C^{n \times n}$ are independent \emph{Wigner matrices} (defined
on an underlying probability space $(\Omega,\mathscr{F},\P)$)
satisfying $A=A^*$ and $B=B^*$, and having independent entries
$(a_{ij},b_{ij})_{i \leq j}$ such that, for all $1 \leq i<j \leq n$,
\[\E a_{ii}=\E b_{ii}=\E a_{ij}=\E b_{ij}=0, \qquad
\E |a_{ij}|^2=\E |b_{ij}|^2=1/n\]
and for all $1 \leq i,j \leq n$, each $p \geq 2$ and a constant $C_p>0$,
\begin{equation}\label{eq:momentassump}
\E |a_{ii}|^p,\E|a_{ij}|^p,\E |b_{ii}|^p,\E|b_{ij}|^p \leq C_p n^{-p/2}.
\end{equation}
(This includes the special case where $A,B$ are real and symmetric.)
\item $\Theta=\diag(\theta_1,\ldots,\theta_n) \in \R^{n \times n}$
and $\Xi=\diag(\xi_1,\ldots,\xi_n) \in \R^{n \times n}$
are deterministic diagonal matrices satisfying
$\|\Theta\|_\op,\|\Xi\|_\op \leq \upsilon$ for a constant $\upsilon>0$.
\end{enumerate}
\end{asp}
\vspace{\baselineskip}

We study the Kronecker deformed Wigner model
\begin{equation}\label{eq:Q}
Q=A \otimes I_{n \times n}+I_{n \times n} \otimes B+\Theta \otimes \Xi
\in \C^{n^2 \times n^2}.
\end{equation}
For spectral arguments $z \in \C^+$, define the resolvent and Stieltjes
transform of $Q$ by
\begin{align*}
G(z)&=(Q-z\,I_{n \times n}
\otimes I_{n \times n})^{-1},\\
m(z)&=n^{-2}\Tr G(z)
=(n^{-1}\Tr\,\otimes\,n^{-1}\Tr)\Big[(Q-z\,I_{n \times n}
\otimes I_{n \times n})^{-1}\Big].
\end{align*}
We will use Roman indices for the first tensor factor, Greek indices for the
second, and write the blocks and entries of $G(z)$ as
\[G_{ij}=(\e_i \otimes I)^*G(\e_j \otimes I) \in \C^{n \times n},
\qquad G_{\alpha\beta}=(I \otimes \e_\alpha)^*G(I \otimes \e_\beta)
\in \C^{n \times n},\]
\[G_{ij,\alpha\beta}=(\e_i \otimes \e_\alpha)^* G(\e_j \otimes \e_\beta) \in
\C.\]

Our main result will establish deterministic equivalent approximations for
$G(z)$ and $m(z)$. These may be defined
through a free probability construction, in the following setting:
\begin{asp}\label{assump:free}
\begin{enumerate}[(a)]
\item $\cA$ is a von Neumann algebra with unit $1_\cA$, operator norm
$\|\cdot\|_\op$, and faithful, normal, tracial state $\tau:\cA \to \C$. 
(We review this definition at the start of
Section \ref{sec:preliminaries}.)
\item $\cA$ contains as a von Neumann subalgebra
\[\cM \equiv \C^{n \times n}\]
where $I_{n \times n} \in \cM \subset \cA$ coincides with $1_\cA$,
and $\tau|_{\cM}$ and $\|\cdot\|_\op|_{\cM}$ restrict to the
normalized matrix trace $\frac{1}{n}\Tr$ and matrix operator norm on $\cM$.
\item $\cA$ has two semicircular elements $\a,\b$
(i.e.\ satisfying $\tau(\a^k)=\tau(\b^k)=\int_{-2}^2
x^k\,\frac{1}{2\pi}\sqrt{4-x^2}\,dx$ for each $k \geq 1$) which are free of
$\cM$ with respect to $\tau$.
\end{enumerate}
\end{asp}

Let $\cD \subset \cM$ be the subalgebra of diagonal matrices, 
generated by the diagonal basis elements
$\{E_{ii}\}_{i=1}^n$, and denote by $\tau^\cD:\cA \to \cD$ the diagonal
projection
\begin{equation}\label{eq:tauD}
\tau^\cD(\x)=\sum_{i=1}^n \frac{\tau(\x E_{ii})}{\tau(E_{ii})}E_{ii}
=\sum_{i=1}^n n\,\tau(\x E_{ii})E_{ii}.
\end{equation}
(This is the unique $\tau$-preserving conditional expectation onto $\cD$,
in the sense of Lemma \ref{conditional expectation}.)
We implicitly identify $\cA$ with its representation on an
underlying Hilbert space $\cH$, and denote by $\cA \otimes \cA$ the von
Neumann tensor product acting on $\cH \otimes \cH$. We denote by
$\tau \otimes \tau:\cA \otimes \cA \to \C$,
$\tau \otimes 1:\cA \otimes \cA \to \cA$,
and $1 \otimes \tau:\cA \otimes \cA \to \cA$
the unique bounded linear maps satisfying
\[(\tau \otimes \tau)(\x \otimes \y)
=\tau(\x)\tau(\y), \qquad
(\tau \otimes 1)(\x \otimes \y)=\tau(\x)\y,
\qquad (1 \otimes \tau)(\x \otimes \y)=\tau(\y)\x,\]
and similarly for
$\tau^\cD \otimes \tau^\cD$, $\tau^\cD \otimes 1$, and $1 \otimes \tau^\cD$.

Identifying $\Theta,\Xi \in \cM$ as elements of $\cA$ (which are
free of $\a,\b$ by assumption), in the limit $n \to \infty$, an
approximation of our matrix of interest $Q$ in the tracial sense is given by
\[\q=\a \otimes 1+1 \otimes \b+\Theta \otimes \Xi \in \cA \otimes \cA.\]
For spectral arguments $z \in \C^+$, we define its (deterministic,
$n$-dependent) resolvent and Stieltjes transform by
\begin{align*}
\g(z)&=(\q-z\,1 \otimes 1)^{-1},\\
m_0(z)&=\tau \otimes \tau[\g(z)] \in \C^+.
\end{align*}
The deterministic equivalent matrix $G_0(z)$ for $G(z)$ is then given by
\[G_0(z)=(\tau^\cD \otimes \tau^\cD)[\g(z)] \in \cD \otimes \cD.\]
We remark that $G_0(z)$ is a deterministic diagonal matrix
in $\C^{n^2 \times n^2}$. (We refer to \cite[Theorem 4.4]{ZhouFanSunWang2021Principal} for a previous example of this type of construction, in a different model.)

In contrast to more classical random matrix models, the above Stieltjes
transform $m_0(z)$ in general does not seem to admit a simple characterization
in terms of scalar-valued fixed-point equations.
In some previously studied models with variance profiles or
correlated entries, deterministic equivalents for the Stietljes transform
and resolvent may instead be defined via fixed-point equations in a vector or matrix space
\cite{AltErdosKrugerNemish2019Location,ErdosKrugerNemish2020Polynomials,AltErdosKruger2017}.
For our model of interest, the most direct characterization of its asymptotic
Stieltjes transform seems to be
via the following fixed-point equation in the operator algebra $\cA$, from
which a deterministic equivalent $G_0(z)$ for its resolvent may also be
constructed.

\begin{prop}\label{rmt:fixed point characterization}
In the setting of Assumption \ref{assump:free}, set
$\cA^+=\{\x \in \cA:\Im \x \geq \epsilon \text{ for some }
\epsilon>0\}$.
For any $z \in \C^+$, there exists a unique element $\m_a(z)
\in \cA^+$ satisfying the fixed-point equation
\begin{align}
\m_a(z)&=\frac{1}{n}\sum_{\alpha=1}^n (\a+\xi_\alpha\Theta-z-\m_a(z))^{-1}.
\label{eq:ma}
\end{align}
Similarly, there exists a unique element $\m_b(z) \in \cA^+$ satisfying
\begin{align}
\m_b(z)&=\frac{1}{n}\sum_{i=1}^n (\b+\theta_i \Xi-z-\m_b(z))^{-1}.
\label{eq:mb}
\end{align}
We have
\begin{align*}
m_0(z)&=\tau[\m_a(z)]=\tau[\m_b(z)],\\
G_0(z)&=\sum_{\alpha=1}^n
\tau^\cD\big[(\a+\xi_\alpha\Theta-\z-\m_a(z))^{-1}\big]\otimes E_{\alpha\alpha}
=\sum_{i=1}^n E_{ii} \otimes \tau^\cD\big[(\b+\theta_i\Xi-z-\m_b(z))^{-1}\big].
\end{align*}
\end{prop}

Qualitative approximations of the resolvent $G(z)$ and
Stieltjes transform $m(z)$ by the above constructions $G_0(z)$ and $m_0(z)$ may
be obtained using standard techniques of free probability theory, such as a
moment expansion of $G(z)$ and $m(z)$ for large $\Im z$ and analytic
continuation to the full upper half-plane. A primary purpose of our
work is to obtain more quantitative estimates that are not readily deduced from
these types of methods. The following theorem states these estimates, which
represent optimal approximations for $G(z)$ and $m(z)$ 
on the global spectral scale. Here, the stochastic domination notation 
$\xi \prec \zeta$ (c.f.\ \cite{erdHos2013averaging})
means $\P[|\xi|>n^\epsilon \zeta] \leq n^{-D}$ for any fixed $\epsilon,D>0$ 
and all $n \geq n_0(\epsilon,D)$; we review this definition
in Section \ref{subsec:domination}. 

\begin{thm}\label{rmt:main theorem}
Fix any $\upsilon,\delta>0$. Under Assumptions \ref{assump:Wigner} and
\ref{assump:free}, uniformly over $z \in \C^+$ with 
$|z| \leq \upsilon$ and $\Im z \geq \delta$, and over $i,j,\alpha,\beta \in
\{1,\ldots,n\}$ with $i \neq j$ and $\alpha \neq \beta$, we have:
\begin{enumerate}[(a)]
\item (Stieltjes transform)
\begin{equation}\label{rmt:Stieltjes transform}
|m(z)-m_0(z)| \prec n^{-1}.
\end{equation}
\item (Resolvent blocks)
\begin{align}
\Big\|G_{ii}^{-1}-G_{jj}^{-1}-(\theta_i-\theta_j)\Xi\Big\|_\op &\prec
n^{-1/2},\label{rmt:diagonalblocksa}\\
\Big\|G_{\alpha\alpha}^{-1}-G_{\beta\beta}^{-1}
-(\xi_\alpha-\xi_\beta)\Theta\Big\|_\op &\prec n^{-1/2},
\label{rmt:diagonalblocksb}\\
\|G_{ij}\|_\op,\|G_{\alpha\beta}\|_\op &\prec n^{-1/2}.
\label{rmt:offdiagonalblocks}
\end{align}
\item (Resolvent entries)
\begin{align}
G_{ii,\alpha\alpha}-(G_0)_{ii,\alpha\alpha} &\prec n^{-1/2},
\label{rmt:diagonal entries}\\
G_{ii,\alpha\beta}, G_{ij,\alpha\alpha}
&\prec n^{-1/2},\label{rmt:root n entries}\\
G_{ij,\alpha\beta} &\prec n^{-1}.\label{rmt:1/n entries}
\end{align}
\item (Bilinear forms)
Uniformly over vectors $\u,\v,\u',\v' \in \C^n$ with
$\|\u\|_2,\|\v\|_2,\|\u'\|_2,\|\v'\|_2 \leq \upsilon$,
\begin{equation}\label{rmt:isotropic entries}
(\u \otimes \v)^* \big[G-G_0\big] (\u' \otimes \v')
\prec n^{-1/2}.
\end{equation}
\end{enumerate}
\end{thm}

The estimates (\ref{rmt:diagonal entries}), (\ref{rmt:root n entries}), and
(\ref{rmt:1/n entries}) indicate that the entries of the resolvent fall on three
scales of orders 1, $n^{-1/2}$, and $n^{-1}$, as depicted in the numerical
simulation of Figure \ref{fig:resolvent}. In a setting where
$\Theta \otimes \Xi$ commutes with $A \otimes I$ and $I
\otimes B$, the origins and optimality of these estimates may be understood via
a contour integral representation of $G(z)$, as we discuss in
Appendix \ref{appendix:contour}.
The above theorem illustrates that various properties of $G(z)$ 
for general $\Theta \otimes \Xi$ are in fact similar to this commutative case.

\begin{figure}[t]
    \centering
    \includegraphics[height=7cm]{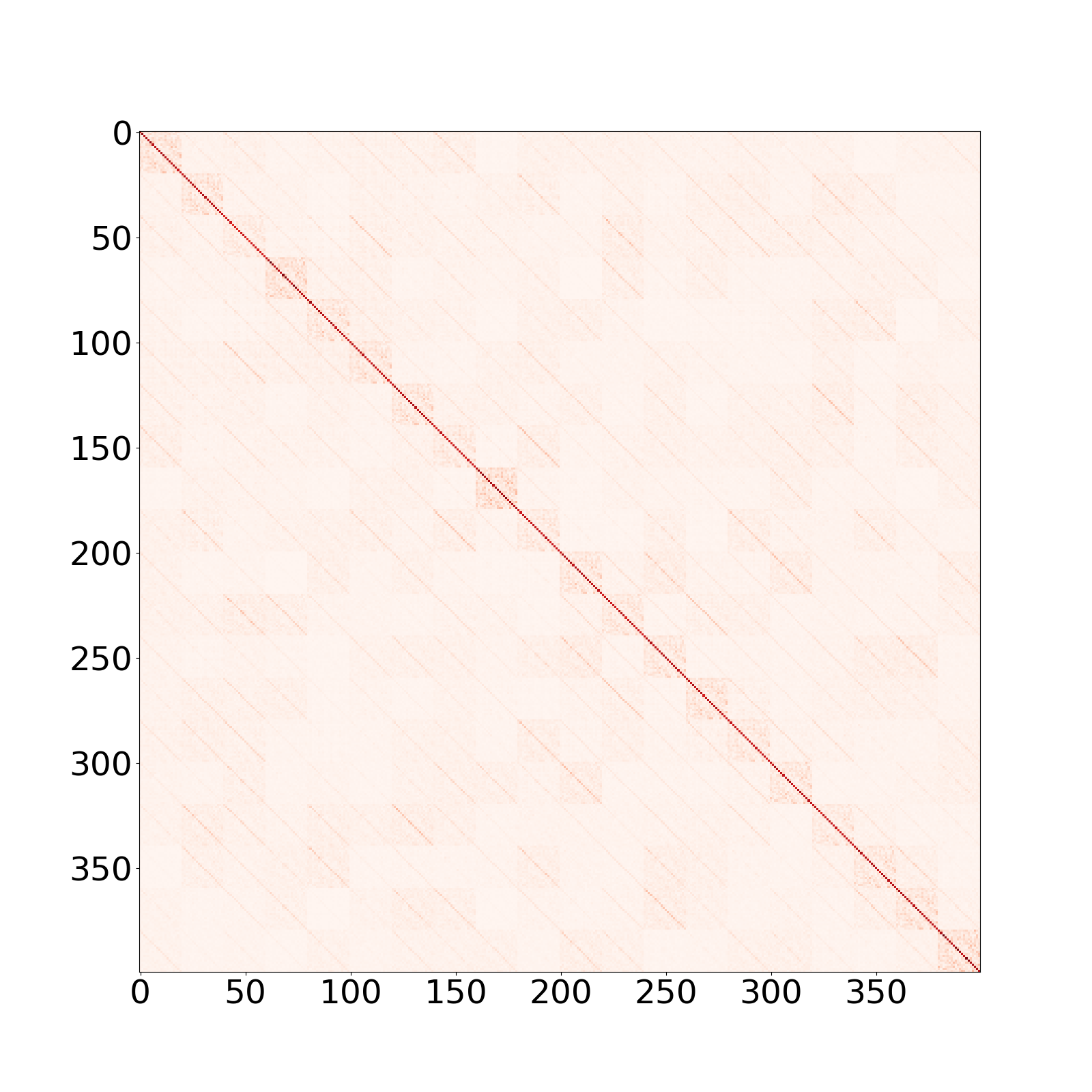}
    \caption{Entrywise modulus of the resolvent $G(z)=(A \otimes I+I
\otimes B+\Theta \otimes \Xi-z)^{-1}$
    at $z=i$, where $n=20$, $A,B$ are independent GOE matrices of size $n$,
    and $\Theta,\Xi$ have independent $\text{Uniform}(-1,1)$ diagonal entries.}
    \label{fig:resolvent}
\end{figure}

\subsection{Application to least-squares problem}

In the same setting of Assumption \ref{assump:Wigner} with real-valued Wigner
matrices $A,B \in \R^{n \times n}$ and
$\theta_i,\xi_j>0$, consider now the optimization objective
\begin{equation}\label{eq:objective}
f(X)=\frac{1}{2}\|XA+BX\|_F^2+\frac{1}{2}\sum_{i,j=1}^n \xi_i\theta_j
x_{ij}^2.
\end{equation}
Given vectors $\u,\v \in \R^n$, define its minimizer under a linear constraint
\begin{equation}\label{eq:opt}
\widehat{X}=\argmin_{X \in \R^{n \times n}} f(X)
\qquad \text{ subject to } \qquad \frac{1}{n}\v^* X\u=1.
\end{equation}
Our main results for this optimization problem (\ref{eq:opt}) are the following
asymptotic characterization of the minimum objective value $f(\widehat X)$
and the values of linear forms $n^{-1}{\v'}^*\widehat X\u'$ for deterministic
test vectors $\u',\v' \in \R^n$.

\begin{thm}\label{optimization: main theorem}
Fix any $\upsilon,\delta>0$. Suppose Assumptions \ref{assump:Wigner} and
\ref{assump:free} hold, where $A,B \in \R^{n \times n}$ are real-valued.
Associated to $\u,\v,\u',\v' \in \R^n$, denote
$\su=\diag(\u)$, $\sv=\diag(\v)$,
$\su'=\diag(\u')$, and $\sv'=\diag(\v')$ as diagonal matrices in $\cD$.
Then, uniformly over $\Theta,\Xi,\u,\v$ defining (\ref{eq:opt}) such that
$\Theta,\Xi \geq \delta I$ and
$\upsilon^{-1}\sqrt{n} \leq \|\u\|_2,\|\v\|_2 \leq \upsilon\sqrt{n}$:
\begin{enumerate}[(a)]
\item (Objective value)
\[f(\widehat X)=\frac{1}{2} \cdot \frac{1}{(\tau \otimes \tau)[(\su \otimes \sv)[(\a
\otimes 1+1 \otimes \b)^2+\Theta \otimes \Xi]^{-1}(\su \otimes \sv)]}
+\Oprec(n^{-1/2})\]
\item (Linear projection) Uniformly over $\u',\v' \in \R^n$ with
$\|\u'\|_2,\|\v'\|_2 \leq \upsilon\sqrt{n}$,
\[\frac{1}{n}\,{\u'}^*\widehat X{\v'}=\frac{(\tau \otimes \tau)[(\su' \otimes
\sv')[(\a \otimes 1+1 \otimes \b)^2+\Theta \otimes \Xi]^{-1}(\su \otimes \sv)]}
{(\tau \otimes \tau)[(\su \otimes \sv)[(\a \otimes 1+1 \otimes
\b)^2+\Theta \otimes \Xi]^{-1}(\su \otimes \sv)]}
+\Oprec(n^{-1/2}).\]
\end{enumerate}
\end{thm}

We remark that the above value
\begin{equation}\label{eq:value1}
(\tau \otimes \tau)[(\su' \otimes \sv')[(\a
\otimes 1+1 \otimes \b)^2+\Theta \otimes \Xi]^{-1}(\su \otimes \sv)]
\end{equation}
may be understood as
\begin{equation}\label{eq:value2}
n^{-2}(\u' \otimes \v')^*
(\tau^\cD \otimes \tau^\cD)[(\a
\otimes 1+1 \otimes \b)^2+\Theta \otimes \Xi]^{-1}(\u \otimes \v),
\end{equation}
where $(\tau^\cD \otimes \tau^\cD)[(\a
\otimes 1+1 \otimes \b)^2+\Theta \otimes \Xi]^{-1} \in \R^{n^2 \times n^2}$ is
a deterministic-equivalent approximation for the matrix
$[(A \otimes I+I \otimes B)^2+\Theta \otimes \Xi]^{-1}$ arising in the
vectorization of the objective (\ref{eq:objective}), and these expressions
(\ref{eq:value1}) and (\ref{eq:value2})
coincide because this deterministic-equivalent matrix is diagonal.

In an asymptotic setting, if the empirical distributions of coordinates of
$\Theta,\Xi,\u,\v,\u',\v'$ converge to deterministic limits, then a qualitative
implication of Theorem \ref{optimization: main theorem} is a characterization
of the almost-sure limit values of $f(\widehat{X})$ and
$n^{-1}{\v'}^*\widehat{X}\u'$, as summarized in the following corollary.

\begin{cor}\label{cor:optimization} 
Asymptotically as $n \to \infty$, if the empirical distributions of coordinates
of $\btheta=\diag(\Theta)$, $\bxi=\diag(\Xi)$, and $\u,\v,\u',\v'$ satisfy
\[\frac{1}{n}\sum_{i=1}^n \delta_{\theta_i,u_i,u_i'} \Rightarrow \cP,
\qquad \frac{1}{n}\sum_{\alpha=1}^n \delta_{\xi_\alpha,v_\alpha,v_\alpha'}
\Rightarrow \cQ\]
weakly for two joint laws $\cP,\cQ$ on $\R^3$,
then there exist almost-sure limit values
\begin{align*}
T(\cP,\cQ)&=\lim_{n \to \infty}
(\tau \otimes \tau)[(\su \otimes \sv)[(\a \otimes 1+1 \otimes \b)^2
+\Theta \otimes \Xi]^{-1}(\su \otimes \sv)],\\
T'(\cP,\cQ)&=\lim_{n \to \infty}
(\tau \otimes \tau)[(\su' \otimes \sv')[(\a \otimes 1+1 \otimes \b)^2
+\Theta \otimes \Xi]^{-1}(\su \otimes \sv)]
\end{align*}
depending only on $\cP,\cQ$, and almost surely as $n \to \infty$,
\[f(\widehat{X}) \to \frac{1}{2T(\cP,\cQ)},
\qquad \frac{1}{n}{\u'}^*\widehat{X}\v' \to \frac{T'(\cP,\cQ)}{T(\cP,\cQ)}.\]
\end{cor}

\subsubsection{Numerical computation}\label{sec:computation}

Since operator fixed-point equations of the form in Proposition \ref{rmt:fixed
point characterization} are not directly amenable to numerical computation,
we provide in this section a procedure for numerically approximating
the values of $T(\cP,\cQ)$ and $T'(\cP,\cQ)$ in Corollary \ref{cor:optimization}
using a moment expansion. (This procedure applies equally to
approximate the finite-$n$ value of (\ref{eq:value1}) rather than its limit
$T'(\cP,\cQ)$, upon replacing the expectations in
(\ref{eq:val}--\ref{eq:emptyval}) below by averages under the empirical
measures of coordinates $\frac{1}{n}\sum_{i=1}^n \delta_{\theta_i,u_i,u_i'}$
and $\frac{1}{n}\sum_{\alpha=1}^n \delta_{\xi_\alpha,v_\alpha,v_\alpha'}$.)

For each even integer $m \geq 2$, let $\NC_{2,2}(m)$ denote the set of ordered
pairs $(\rho_a,\rho_b)$ where $\rho_a,\rho_b$ are non-crossing pairings of two
disjoint even-cardinality subsets of $\{1,2,\ldots,m\}$ whose union is all of
$\{1,2,\ldots,m\}$. (The pairings $\rho_a$ and $\rho_b$ may cross each other,
and either pairing may be empty.) We associate to each $(\rho_a,\rho_b) \in
\NC_{2,2}(m)$ a value $\val(\rho_a,\rho_b)$ in the following way:
\begin{enumerate}
\item Let $w_a$ be the word in the letters $\{\a,\d\}$ that is obtained by
traversing the elements of $\{1,\ldots,m\}$ in sequential order, and writing
$\d\a$ for odd elements in $\rho_a$, $\a\d$ for even elements in $\rho_a$,
and $\d$ for all elements in $\rho_b$.

Similarly, let $w_b$ be the word obtained by writing
$\d\b$ for odd elements in $\rho_b$, $\b\d$ for even elements in $\rho_b$,
and $\d$ for all elements in $\rho_a$.
\item Define the complement $K(\rho_a)$ of $\rho_a$ in $w_a$
as the coarsest non-crossing partition of its letters $\d$,
for which $\rho_a \cup K(\rho_a)$ forms a non-crossing partition of all letters
of $w_a$. Define similarly the complement $K(\rho_b)$ of $\rho_b$ in $w_b$.
\item Finally, let $S_1$ be the block of $K(\rho_a)$ that
contains the first letter of $w_a$, let $T_1$ be the block of
$K(\rho_b)$ that contains the first letter of $w_b$, and set
\begin{align}
\val(\rho_a,\rho_b)
&=\E\Big[\sU\sU'\theta^{-(|S_1|+2)/2}\Big]
\prod_{S \in K(\rho_a) \setminus S_1}\E\Big[\theta^{-|S|/2}\Big] \times
\nonumber\\
&\hspace{1in}\E\Big[\sV\sV'\xi^{-(|T_1|+2)/2}\Big] 
\prod_{T \in K(\rho_b) \setminus T_1}
\E\Big[\xi^{-|T|/2}\Big]\label{eq:val}
\end{align}
where the expectations are over
$(\theta,\sU,\sU') \sim \cP$ and $(\xi,\sV,\sV') \sim \cQ$.
\end{enumerate}
For $m=0$, we take $\NC_{2,2}(0)$ to consist of the single pair
$(\rho_a,\rho_b)$ with $\rho_a=\rho_b=\emptyset$ both as the empty pairing,
and set
\begin{equation}\label{eq:emptyval}
\val(\emptyset,\emptyset)
=\E[\sU\sU'\theta^{-1}]\E[\sV\sV'\xi^{-1}].
\end{equation}

To illustrate this in an example, suppose $m=6$, $\rho_a=\{(1,3)\}$,
and $\rho_b=\{(2,6),(4,5)\}$. Then
\begin{align*}
w_a&=\underbrace{\d\a}_{1}\underbrace{\d}_{2}\underbrace{\d\a}_{3}\underbrace{\d}_{4}\underbrace{\d}_{5}\underbrace{\d}_{6}
=\d\a\d\d\a\d\d\d\\
w_b&=\underbrace{\d}_{1}\underbrace{\b\d}_{2}\underbrace{\d}_{3}\underbrace{\b\d}_{4}\underbrace{\d\b}_{5}\underbrace{\b\d}_{6}
=\d\b\d\d\b\d\d\b\b\d
\end{align*}
Re-numbering the letters of $w_a$ as $\{1,\ldots,8\}$, $\rho_a$ now corresponds
to the pairing $\{(2,5)\}$ of the letters $\a$, and $K_a$ is the partition
$\{(1,6,7,8),(3,4)\}$ of the letters $\d$. Similarly, re-numbering the letters
of $w_b$ as $\{1,\ldots,10\}$, $\rho_b$ is now the
pairing $\{(2,9),(5,8)\}$ of the letters $\b$, and $K_b$ is the partition
$\{(1,10),(3,4),(6,7)\}$ of the letters $\d$. Then $K_a$ has blocks of sizes
$4,2$, $K_b$ has blocks of sizes $2,2,2$, so
\[\val(\rho_a,\rho_b)
=\E[\sU\sU'\theta^{-3}]\E[\theta^{-1}]\E[\sV\sV'\xi^{-2}]\E[\xi^{-1}]\E[\xi^{-1}].\]

The values $T(\cP,\cQ)$ and $T'(\cP,\cQ)$ then admit the following series
approximations.

\begin{prop}\label{prop:computation}
Let $(\theta,\sU,\sU') \sim \cP$ and $(\xi,\sV,\sV') \sim \cQ$ where $\cP,\cQ$
are the joint laws on $\R^3$ of Corollary \ref{cor:optimization}.
Denote $\|\sU\|_\infty=\max\{|x|:x \in \supp(\sU)\}$ where $\supp(\sU)$ is the
support of $\sU$.
Set $\eta=\min\{\sqrt{x_ax_b}:x_a \in \supp(\theta),\,x_b \in \supp(\xi)\}$.
Then for any $M \geq 1$ and $z>1$,
\begin{align*}
T'(\cP,\cQ)&=\mathop{\sum_{m=0}^{M-1}}_{m \text{ is even}}
\left(\frac{(-1)^{m/2}}{z(z-1)^m}
\sum_{k=m}^{M-1} \binom{k}{m}\left(\frac{z-1}{z}\right)^k\right)
\sum_{(\rho_a,\rho_b) \in \NC_{2,2}(m)}
\val(\rho_a,\rho_b)+r_M
\end{align*}
where
\begin{equation}\label{eq:computationremainder}
|r_M| \leq \|\sU\|_\infty\|\sV\|_\infty\|\sU'\|_\infty\|\sV'\|_\infty
\cdot \eta^{-2} \left(\frac{\sqrt{(z-1)^2+16\eta^{-2}}}{z}\right)^M.
\end{equation}
The same result holds for $T(\cP,\cQ)$ upon replacing $\sU',\sV'$ in
(\ref{eq:val}), (\ref{eq:emptyval}), and (\ref{eq:computationremainder}) by
$\sU,\sV$.
\end{prop}

In particular, choosing $z=1+16\eta^{-2}$, the remainder is bounded as
\[|r_M| \leq \|\sU\|_\infty\|\sV\|_\infty\|\sU'\|_\infty\|\sV'\|_\infty
\cdot \eta^{-2}\Big(\frac{z-1}{z}\Big)^{M/2},\]
so the series in $m$ converges geometrically (with faster convergence for larger
values of the regularization $\eta$). Then $T(\cP,\cQ)$ and $T'(\cP,\cQ)$ 
may be approximated to high numerical accuracy by computing
$\sum_{(\rho_a,\rho_b) \in \NC_{2,2}(m)} \val(\rho_a,\rho_b)$ for a small number
of terms $m=0,2,\ldots,M$. A numerical illustration is provided in Figure
\ref{fig:simulation}.

\begin{figure}[t]
    \centering
    \includegraphics[height=7cm]{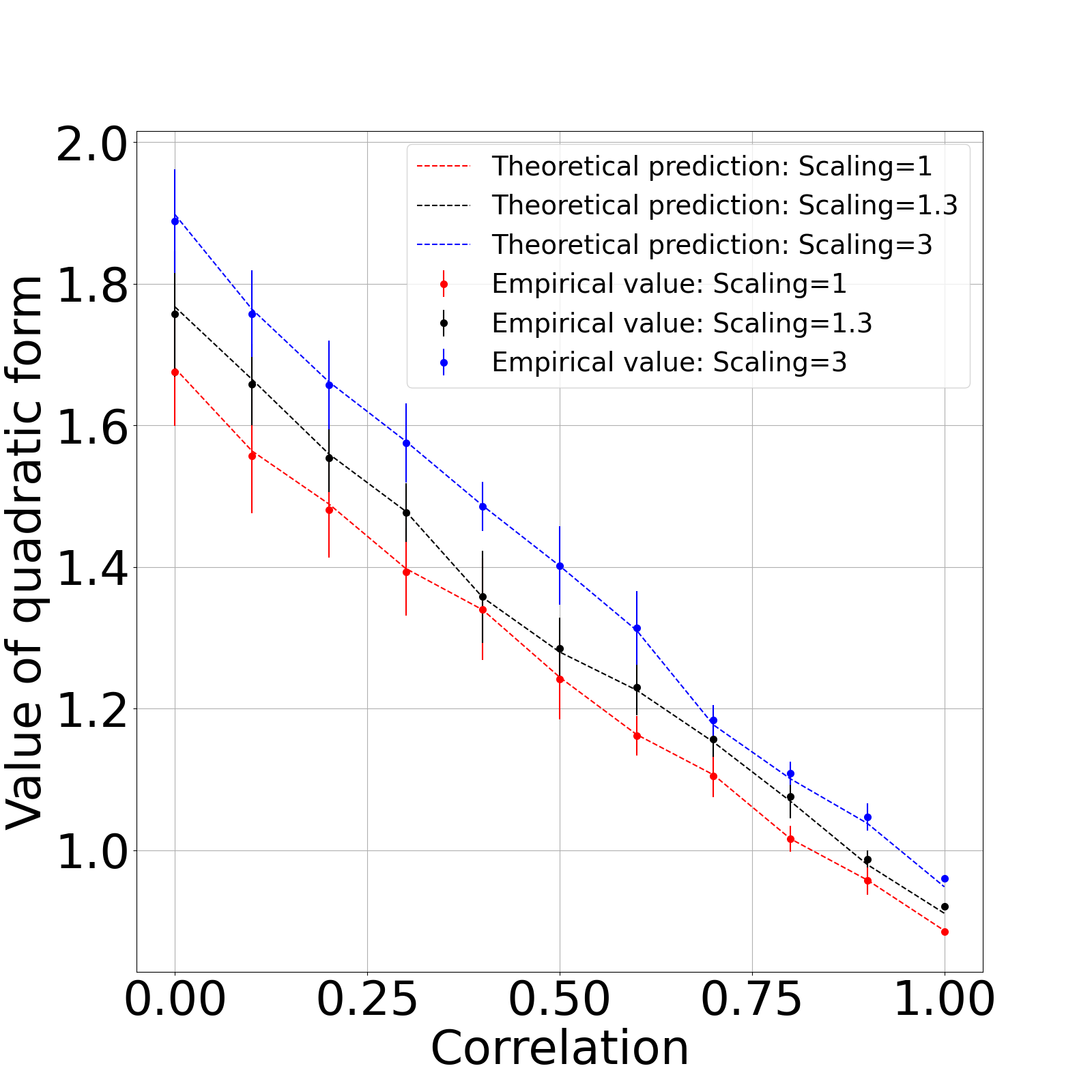}
    \caption{Values of $1/(2f(\widehat{X}))$ obtained from solving
(\ref{eq:opt}) across 10 independent realizations (solid dots, with vertical
lines indicating 1 standard deviation) versus the theoretical prediction
$T(\cP,\cQ)$ computed from Proposition \ref{prop:computation} with $M=12$ (dashed
lines). Here $A,B$ are GOE matrices of size $n=1000$, we take
$\theta_i^{-1},\xi_\alpha^{-1},u_i^2,v_\alpha^2 \sim \text{Uniform}(0.05,0.5)/k$
with $k \in \{1,1.3,3\}$, and the horizontal axis indicates the correlation
between coordinate pairs $(\theta_i^{-1},u_i^2)$ and between
$(\xi_\alpha^{-1},v_\alpha^2)$.}\label{fig:simulation}
\end{figure}

\section{Preliminaries}\label{sec:preliminaries}

In the remainder of the paper, we prove the preceding results. 
We note that the bulk of our analysis is in fact carried out
for a generalized resolvent operator
\[R=\bigg(H \otimes \x-\sum_{k=1}^K D_k \otimes \x_k\bigg)^{-1} \in \C^{n \times n}
\otimes \cX\]
in the setting of an abstract von Neumann algebra $\cX$, which we define
more precisely in Section \ref{sec:general}. This analysis may be of independent
interest to models beyond the ones we consider in our work.

We summarize in
this section the technical tools needed for our analyses.
Section
\ref{sec:general} contains the core of our main argument for analyzing
the above generalized resolvent
in an operator-algebra setting.
Section \ref{sec:resolvent} completes the analyses for the Kronecker deformed
Wigner model (\ref{eq:Q}), and Section \ref{sec:optimization} completes the
analyses for the least-squares problem (\ref{eq:opt}).

Throughout this section and Section \ref{sec:general},
$\cX$ is a von Neumann algebra acting on a (complex)
Hilbert space $\cH$, having unit $1_\cX$, operator norm $\|\cdot\|_\op$,
and a faithful, normal, tracial state $\phi:\cX \to \C$. We recall that this
means $\phi$ is a linear functional satisfying
\[|\phi(\x)| \leq \|\x\|_\op, \qquad
\phi(1_\cX)=1, \qquad \phi(\x\y)=\phi(\y\x) \text{ for all } \x,\y \in \cX,\]
\[\phi(\x) \geq 0 \text{ for all } \x \geq 0, \qquad \phi(\x)=0 \text{ only if }
\x=0,\]
and $\phi(\sup \x_i)=\lim \phi(\x_i)$ for any bounded increasing net $\{\x_i\}$
of elements $\x_i \geq 0$ in $\cX$. We will be working with $\cX$-valued random
variables, understood as strongly (i.e.\ Bochner) measurable functions from
the underlying probability space $(\Omega,\mathscr{F},\P)$ to the Banach space
$(\cX,\|\cdot\|_\op)$. For a $\cX$-valued
random variable $\x$ satisfying $\E\|\x\|_\op<\infty$, we denote by $\E\x \in
\cX$ its expectation and $\E[\x \mid \mathscr{G}]$ its conditional
expectation with respect to a sub-sigma-field $\mathscr{G} \subset \mathscr{F}$
(c.f.\ \cite[Chapter 1]{pisier2016martingales}).

\subsection{Stochastic domination}\label{subsec:domination}

For $p \in [1,\infty)$, we denote the $L^p$-norms with respect to $\phi$ as
\[\|\x\|_p=\phi\big(|\x|^p\big)^{1/p}, \qquad |\x|=(\x^*\x)^{1/2}.\]
Relevant properties of these norms and
their associated non-commutative $L^p$-spaces are reviewed in Appendix
\ref{appendix:background}. Throughout, we will write $\x \prec \zeta$
to mean stochastic domination of the $L^p$-norm for each fixed $p \in
[1,\infty)$, in the following sense.

\begin{defi}[Stochastic domination]\label{def:domination}
Let $\x=\{\x(u):u\in U\}$ be a $n$-dependent family of $\cX$-valued random
variables, and $\zeta=\{\zeta(u):u \in U\}$ a corresponding family of
(positive) scalar-valued random variables, where $U$ is a
$n$-dependent parameter set. We say that
\[\x \prec \zeta, \qquad \x=\Oprec(\zeta)\]
uniformly over $u \in U_n$ if, for any fixed $p \in [1,\infty)$
and $\epsilon,D>0$, there exists $n_0(p,\epsilon,D)>0$ such that for all
$n \geq n_0$,
    \begin{align*}
        \sup_{u\in U_n}\mathbb{P}\left[\|\x(u)\|_p \geq
n^\epsilon \zeta(u) \right] \leq n^{-D}.
    \end{align*}
In the scalar setting of a $\C$-valued random variable $\x$, this means
        $\sup_{u\in U_n}\mathbb{P}\left[|\x(u)| \geq n^\epsilon
\zeta(u) \right] \leq n^{-D}$
for some $n_0(\epsilon,D)>0$ and all $n \geq n_0(\epsilon,D)$.
\end{defi}

We will often use implicitly the following basic properties of $\prec$.

\begin{lem}\label{lemma:domination}
\phantom{ }
    \begin{enumerate}[(a)]
        \item If $\x(u,v)\prec \zeta(u,v)$ uniformly over $u\in U$ and $v\in V$,
and $|V|\leq n^C$ for some constant $C>0$, then uniformly over $u \in U$,
        \begin{align*}
            \sum_{v\in V} \x(u,v) \prec \sum_{v\in V}\zeta(u,v)
        \end{align*}
        \item If $\x_1 \prec \zeta_1$ and $\x_2 \prec \zeta_2$
uniformly over $u \in U$, then also $\x_1\x_2\prec \zeta_1\zeta_2$
uniformly over $u \in U$.
        \item 
Suppose $\x \prec \zeta$ uniformly over $u \in U$, where $\zeta$ is
deterministic, $\zeta>n^{-C}$,
and $\E[\|\x\|_p^k] \leq n^{C_{p,k}}$ for all $p,k \in [1,\infty)$ and some constants
$C,C_{p,k}>0$. Then $\E[\x \mid \mathscr{G}] \prec \zeta$ uniformly over
$u \in U$ and over all sub-sigma-fields $\mathscr{G} \subseteq \mathscr{F}$
of the underlying probability space $(\Omega,\mathscr{F},\P)$.
    \end{enumerate}
\end{lem}
\begin{proof}
The argument is similar to the scalar setting
(c.f.\ \cite[Lemma D.2]{fan2022tracy}):
For each fixed $p \in [1,\infty)$, by the triangle inequality and H\"older's
inequality for the $L^p$-norm (Lemma \ref{Holder's inequality}),
\[\bigg\|\sum_{v \in V} \x(u,v)\bigg\|_p \leq
\sum_{v \in V} \|\x(u,v)\|_p, \qquad
\|\x_1\x_2\|_p \leq \|\x_1\|_{2p}\|\x_2\|_{2p}.\]
Statements (a) and (b) then follow from a union bound. For (c), by the given
assumptions, for any $p,k \in [1,\infty)$, $\epsilon>0$, and
all $n \geq n_0(\epsilon,k,p)$,
\begin{align}
\E[\|\x\|_p^k] &\leq \E[\|\x\|_p^k\1\{\|\x\|_p^k \leq n^{\epsilon/2}\zeta^k\}]+
\E[\|\x\|_p^k\1\{\|\x\|_p^k>n^{\epsilon/2}\zeta^k\}]\notag\\
&\leq
n^{\epsilon/2}\zeta^k+\E[\|\x\|_p^{2k}]^{1/2}\P[\|\x\|_p^k>n^{\epsilon/2}\zeta^k]^{1/2}<n^\epsilon\zeta^k.\label{eq:Expk}
\end{align}
The triangle inequality for $\|\cdot\|_p$ implies for $\lambda \in [0,1]$ that
$\|\lambda\x+(1-\lambda)\y\|_p \leq \lambda\|\x\|_p+(1-\lambda)\|\y\|_p$,
so $\x \mapsto \|\x\|_p$ is continuous and convex. Then
(c.f.\ \cite[Proposition 1.12]{pisier2016martingales})
\begin{equation}\label{eq:jensencondexp}
\|\E[\x \mid \mathscr{G}]\|_p \leq \E[\|\x\|_p \mid \mathscr{G}].
\end{equation}
So for any $\mathscr{G} \subseteq \mathscr{F}$, $p \in [1,\infty)$, and
$\epsilon,D>0$, fixing $k \geq 1$ such that $(k-1)\epsilon>D$
and choosing $n \geq n_0(p,\epsilon,D)$ large enough so that (\ref{eq:Expk})
holds,
\[\P[\|\E[\x \mid \mathscr{G}]\|_p>n^\epsilon \zeta]
\leq \frac{\E[\|\E[\x \mid \mathscr{G}]\|_p^k]}{n^{k\epsilon}\zeta^k}
\leq \frac{\E[\E[\|\x\|_p \mid \mathscr{G}]^k]}{n^{k\epsilon}\zeta^k}
\leq \frac{\E[\|\x\|_p^k]}{n^{k\epsilon}\zeta^k}<n^{-(k-1)\epsilon}<n^{-D}.\]
\end{proof}

\begin{rek}\label{rek:op holds in finite dimension}
For the finite-dimensional matrix algebra
$(\C^{n \times n},\|\cdot\|_\op,\frac{1}{n}\Tr)$, we have
$\|X\|_\op \leq (\Tr (X^*X)^{p/2})^{1/p}=n^{1/p}\|X\|_p$.
Thus if $X \prec \zeta$, then for any
$\epsilon,D>0$ and all $n \geq n_0(\epsilon,D)$, choosing $p=\max(1,2/\epsilon)$,
\[\P[\|X\|_\op \geq n^\epsilon \zeta]
\leq \P[n^{1/p}\|X\|_p \geq n^\epsilon\zeta]
\leq \P[\|X\|_p \geq n^{\epsilon/2}\zeta]<n^{-D}\]
so this implies the operator norm bound $\|X\|_\op \prec \zeta$.
However, we caution that this implication does not hold in
infinite-dimensional settings, where our notation $\x \prec \zeta$ has a weaker
meaning than $\|\x\|_\op \prec \zeta$.
\end{rek}

\subsection{Khintchine-type inequalities}

The following statements may be deduced from the non-commutative Rosenthal
inequalities of \cite{JungeXu2008}. In the setting of Rademacher variables,
similar Khintchine inequalities have been shown in
\cite[Eq.\ (8.4.11)]{pisier1998non} and
\cite[Theorem 6.22]{rauhut2010compressive}.

\begin{lem}\label{lemma:concentration}
Let $(\x_i,\y_{ij}:i,j=1,\ldots,n)$ be (deterministic) elements in
$\cX$, and denote
\begin{equation}\label{eq:partialtranspose}
\Y=\sum_{i,j=1}^n E_{ij} \otimes \y_{ij} \in
\C^{n \times n} \otimes \cX, \qquad
\Y^\st=\sum_{i,j=1}^n E_{ji} \otimes \y_{ij} \in
\C^{n \times n} \otimes \cX
\end{equation}
where $\Y^\st$ is the partial transpose of $\Y$ in its first tensor factor.
We equip $\C^{n \times n} \otimes \cX$ with the trace
$n^{-1}\Tr \otimes \phi$ and its $L^p$-norm
$\|\x\|_p=((n^{-1}\Tr \otimes \phi)[|\x|^p])^{1/p}$.

Let $(\alpha_i,\beta_i:i=1,\ldots,n)$ be independent $\C$-valued random
variables, satisfying $\E \alpha_i=\E \beta_i=0$
and $\E[|\alpha_i|^p],\E[|\beta_i|^p] \leq C_p$ for every $p \geq 2$ and some
constant $C_p>0$. Then for all $p \in [2,\infty)$,
there are constants $C_p',C_p'', C_p'''>0$ such that
\begin{enumerate}[(a)]
\item
\[\E\bigg[\bigg\|\sum_{i=1}^n \alpha_i\x_i\bigg\|_p^p\bigg] \leq C_p'
\max\bigg\{\bigg\|\bigg(\sum_{i=1}^n \x_i\x_i^*\bigg)^{1/2}\bigg\|_p^p,
\bigg\|\bigg(\sum_{i=1}^n \x_i^*\x_i\bigg)^{1/2}\bigg\|_p^p\bigg\}\]
\item
       \begin{align*}
        \E\bigg[\bigg\|\sum_{i,j=1}^n \alpha_i\beta_j\y_{ij}\bigg\|_p^p\bigg] &\leq C_p'' \max\bigg\{\bigg\|\bigg(\sum_{i,j=1}^n
\y_{ij}\y_{ij}^*\bigg)^{1/2}\bigg\|_p^p, 
\bigg\|\bigg(\sum_{i,j=1}^n\y_{ij}^*\y_{ij}\bigg)^{1/2}\bigg\|_p^p,
n\|\Y\|_p^p,n\|\Y^\st\|_p^p\bigg\}
       \end{align*}
\item Suppose $\y_{ii}=0$ for each $i=1,\ldots,n$. Then
\begin{align*}
    \E\bigg[\bigg\|\sum_{1 \leq i \neq j \leq n}
\alpha_i\alpha_j\y_{ij}\bigg\|_p^p\bigg] &\leq C_p'''
\max\bigg\{\bigg\|\bigg(\sum_{1 \leq i \neq j \leq n}
\y_{ij}\y_{ij}^*\bigg)^{1/2}\bigg\|_p^p, 
\bigg\|\bigg(\sum_{1 \leq i \neq j \leq n}
\y_{ij}^*\y_{ij}\bigg)^{1/2}\bigg\|_p^p,n\|\Y\|_p^p,n\|\Y^\st\|_p^p\bigg\}
\end{align*}
\end{enumerate}
\end{lem}
\begin{proof}
See Appendix \ref{appendix:concentration}.
\end{proof}

\subsection{Fluctuation averaging}\label{sec:fluctuationavg}

Let $\{\mathscr{G}_i:i=1,\ldots,n\}$ be a collection of sub-sigma-fields in the
underlying probability space $(\Omega,\mathscr{F},\P)$.
For a $\cX$-valued random variable $\x$ with $\E\|\x\|_\op<\infty$,
define the projections
\[\E_i[\x]=\E[\x \mid \mathscr{G}_i],
\qquad \cQ_i[\x]=(1-\E_i)[\x]=\x-\E_i[\x].\]
Supposing that $\{\E_i,\cQ_i:i=1,\ldots,n\}$ all commute, set
\[\E_S = \prod_{i \in S} \E_i, \qquad \cQ_S = \prod_{i\in S} \cQ_i\]
with the conventions $\E_\emptyset[\x]=\cQ_\emptyset[\x]=\x$.

The following fluctuation averaging statements in the $L^p$-norms on $\cX$ are
similar to those of the scalar setting, see
e.g.\ \cite[Theorem 4.7]{ErdosEtAl2013Local}
and \cite[Lemma A.2]{FanJohnstoneSun2018Spiked}.

\begin{lem}[Fluctuation averaging]\label{lemma:fluctuationavg}
Suppose that $\{\E_i,\cQ_i:i=1,\ldots,n\}$ commute.
Let $\{\x_i\}_{i=1}^n$ and $\{\x_{ij}\}_{1 \leq i \neq j \leq n}$ be 
$\cX$-valued random variables such that for any $p,k \in [1,\infty)$
and some constants $C_{p,k}>0$,
we have $\E[\|\x_i\|_p^k],\E[\|\x_{ij}\|_p^k] \leq n^{C_{p,k}}$
for all $i \neq j$. 
\begin{enumerate}[(a)]
\item Suppose $\E_i[\x_i]=0$, and for some $\alpha,\beta>0$ and each fixed
$l \geq 1$, uniformly over $S \subseteq \{1,\ldots,n\}$ with
$|S| \leq l$ and over $i \notin S$, we have
\begin{equation}\label{eq:QSassump3}
\left\|\cQ_S[\x_i]\right\|_l \prec n^{-\alpha-\beta|S|}.
\end{equation}
Denote $\beta'=\min\{1/2, \beta\}$. Then uniformly over deterministic vectors
$\u=(u_1,\ldots,u_n) \in \C^n$,
    \begin{align*}
        \sum_{i=1}^n u_i\x_i \prec n^{-\alpha-\beta'}(n\|\u\|_\infty)
    \end{align*} 
\item Suppose $\E_i[\x_i]=0$, and for some $\alpha>0$ and each fixed
$l \geq 1$, uniformly over $S \subseteq \{1,\ldots,n\}$ with $|S|
\leq l$ and over $i \notin S$, we have
\begin{equation}\label{eq:QSassump1}
\left\|\cQ_S[\x_i]\right\|_l \prec n^{-\alpha-|S|/2}.
\end{equation}
Then uniformly over deterministic vectors
$\u=(u_1,\ldots,u_n) \in \C^n$,
    \begin{align*}
        \sum_{i=1}^n u_i\x_i \prec n^{-\alpha}\|\u\|_2
    \end{align*}
\item Suppose $\E_i[\x_{ij}]=\E_j[\x_{ij}]=0$, and
for some $\alpha>0$ and each fixed $l \geq 1$,
uniformly over $S \subseteq \{1,\ldots,n\}$ with $|S| \leq l$
and over $i,j \notin S$ with $i \neq j$, we have
\begin{equation}\label{eq:QSassump2}
        \left\|\cQ_S[\x_{ij}]\right\|_l \prec n^{-\alpha-|S|/2}.
\end{equation}
Then uniformly over $(u_{ij})_{i,j=1}^n \in \C^{n \times n}$,
\begin{align*}
\sum_{i\neq j} u_{ij}\x_{ij} \prec n^{-\alpha}\left(\sum_{i\neq j} |u_{ij}|^2\right)^{1/2}
    \end{align*}
    \end{enumerate}
\end{lem}
\begin{proof}
    See Appendix \ref{Fluctuation averaging}.
\end{proof}

\subsection{Minors and resolvent identities}\label{subsec:resolventidentities}

Let $H \in \C^{n \times n}$ and $\x \in \cX$ be self-adjoint, and let
\[\z=\sum_{i=1}^n E_{ii} \otimes \z_i \in \cD \otimes \cX,
\qquad \z_i \in \cX^+=\{\x \in \cX:\Im \x \geq \epsilon \text{ for some }
\epsilon>0\}.\]
Consider the generalized resolvent
\[R=(H \otimes \x-\z)^{-1} \in \C^{n \times n} \otimes \cX.\]
We note that this generalized resolvent
exists by the following standard lemma,
since $\Im(H \otimes \x-\z)={-}\Im \z \leq {-}\epsilon$.
\begin{lem}[\cite{HaagerupThorbjornsen2005}, Lemma 3.1]
    \label{im bound inverse}
    Suppose $\x \in \cX$ is such that for some $\epsilon>0$, either $\Im \x \geq
    \epsilon$ or $\Im \x \leq -\epsilon$.
    Then $\x$ is invertible and $\|\x^{-1}\|_{\op}\leq 1/\epsilon$.
    \end{lem}
For an index set $S \subseteq \{1,\ldots,n\}$, we define $H^{(S)} \in \C^{n
\times n}$ and $R^{(S)} \in \C^{n \times n} \otimes \cX$ by
\[H_{ij}^{(S)}=\begin{cases} H_{ij} & \text{ if } i,j \notin S \\
0 & \text{ otherwise} \end{cases},
\qquad R^{(S)}=(H^{(S)} \otimes \x-\z)^{-1},\]
with $H^{(\emptyset)}=H$ and $R^{(\emptyset)}=R$.
We will use the indexing $R_{ij}=(\e_i\otimes 1)^*R(\e_j\otimes 1)$ for the
$\cX$-valued entries of $R$, and the notations
\begin{align*}
\sum_i^{(S)}=\sum_{i \in \{1,...,n\}\setminus S},
\qquad \frac{1}{R_{ii}^{(S)}}=(R_{ii}^{(S)})^{-1},
\qquad iS=S \cup \{i\}.
\end{align*}

\begin{lem}[Resolvent identities]\label{resolvent identities}
Suppose $H \in \C^{n \times n}$ and $\x \in \cX$ are self-adjoint,
and $\z_i \in \cX^+$ for each $i=1,\ldots,n$.
For any $S \subseteq\{1,\ldots,n\}$:
\begin{enumerate}[(a)]
\item For any $i \notin S$, $R_{ii}^{(S)} \in \cX$ is invertible, and
\[\frac{1}{R_{ii}^{(S)}} =
h_{ii}\x-\z_i-\x\left(\sum_{r,s}^{(iS)}
h_{ir}R_{rs}^{(iS)}h_{si}\right)\x\]
\item For any $i,j \notin S$ with $i \neq j$,
\begin{align*}
R_{ij}^{(S)}&=
{-}R_{ii}^{(S)}\x\left(\sum_r^{(iS)}h_{ir}R_{rj}^{(iS)}\right)
={-}\left(\sum_s^{(jS)}R_{is}^{(jS)}h_{sj}\right)\x R_{jj}^{(S)}\\
&={-}h_{ij} R_{ii}^{(S)}\x R_{jj}^{(iS)}
+R_{ii}^{(S)}\x \left(\sum_{r,s}^{(ijS)} h_{ir}R_{rs}^{(ijS)}h_{sj}\right)\x R_{jj}^{(iS)}.
\end{align*}
\item For any $i,j,r \notin S$ (including $i=j$) with $r \notin \{i,j\}$,
        \begin{align*}
            R_{ij}^{(S)} &= R_{ij}^{(rS)} + 
            R_{ir}^{(S)}\frac{1}{R_{rr}^{(S)}}R_{rj}^{(S)},\\
            \frac{1}{R_{ii}^{(S)}} &= \frac{1}{R_{ii}^{(rS)}} - 
            \frac{1}{R_{ii}^{(S)}}R_{ir}^{(S)}\frac{1}{R_{rr}^{(S)}}R_{ri}^{(S)}
            \frac{1}{R_{ii}^{(rS)}}.
        \end{align*}
    \end{enumerate}
\end{lem}
\begin{proof}
These follow from Schur-complement inversion identities applied to
\[H^{(S)}-\z
=\begin{pmatrix} h_{11}^{(S)}\x-\z_1 & h_{12}^{(S)}\x & \cdots &
h_{1n}^{(S)}\x \\
h_{21}^{(S)}\x & h_{22}^{(S)}\x-\z_2 & \cdots & h_{2n}^{(S)}\x \\
\vdots & \vdots & \ddots & \vdots \\
h_{n1}^{(S)}\x & h_{n2}^{(S)}\x & \cdots & h_{nn}^{(S)}\x-\z_n
\end{pmatrix} \in \C^{n \times n} \otimes \cX,\]
which are purely algebraic and identical to those of the scalar setting, see
e.g.\ \cite[Lemma 4.2]{ErdosYauYin2012BulkUniversality}.
\end{proof}

\subsection{Maximum modulus principle}

We will use the following quantitative version of the maximum modulus
principle, following \cite[Appendix A]{SpeicherVargas2012}.

\begin{lem}\label{lemma:maximummodulus}
Fix any $a>0$. For each $r \in (-\infty,0)$, define the circle in $\C^+$
\[S_r=\{z \in \C^+:|z-ia\tfrac{1+e^{2r}}{1-e^{2r}}|=\tfrac{2ae^r}{1-e^{2r}}\}.\]
Let $f:\C^+ \to \C$ be any analytic function not identically equal to 0, and set
\[M(r)=\max_{z \in S_r} \log |f(z)|.\]
Then $r \mapsto M(r)$ is increasing and convex over $r \in (-\infty,0)$.
\end{lem}
\begin{proof}
For any analytic function $g:\D \to \C$ (not identically 0)
on the unit disk $\D=\{z:|z|<1\}$, the function
\[r \mapsto \max_{z \in \D:|z|=e^r} \log |g(z)|\]
is increasing and
convex \cite[Section 1]{hardy1915mean}. Fixing $a>0$, consider the conformal mapping
$\psi_a:\D \rightarrow\C^+$ given by
$\psi_a(z)=ia\frac{1+z}{1-z}$. For any $z \in \D$,
\[\left|\psi_a(z)-ia\frac{1+|z|^2}{1-|z|^2}\right|
=\left|\frac{2a(z-|z|^2)}{(1-z)(1-|z|^2)}\right|
=\frac{2a|z|}{1-|z|^2}\left|\frac{1-\bar{z}}{1-z}\right|=\frac{2a|z|}{1-|z|^2},\]
so $\psi_a$ maps each circle $\{z \in \D:|z|=e^r\}$ bijectively to
$S_r \subset \C^+$. The result then follows from applying the above monotonicity
and convexity to the function $g(z)=f(\psi_a(z))$.
\end{proof}

\section{Analysis of a generalized resolvent}\label{sec:general}

Much of the analysis for Theorem \ref{rmt:main theorem}
may be stated at the level of a generalized resolvent
\begin{equation}\label{eq:generalR}
R=(H \otimes \x-\z)^{-1} \in \C^{n \times n} \otimes \cX,
\qquad \z=\sum_{k=1}^K D_k \otimes \x_k \in \cD \otimes \cX
\end{equation}
where $\x,\x_1,\ldots,\x_K$ are elements of an abstract von Neumann
algebra $\cX$, and $H$ is a Wigner random matrix.
We will then specialize to $\cX=\C^{n \times n}$ and $\cX=\cA$ for the two
stages of analysis of the Kronecker deformed Wigner model (\ref{eq:Q}).
We emphasize that it is important for our arguments to hold when $\cX$ is
infinite-dimensional, for the analysis in the second stage with $\cX=\cA$.

We collect here the assumptions of this general setting.

\begin{asp}\label{assump:general}
There exist constants $\gamma \geq 1$, $\upsilon \geq \delta>0$, and $K \geq 1$
such that:
\begin{enumerate}[(a)]
\item $H \in \C^{n \times n}$ is a random Wigner matrix satisfying the
conditions of Assumption \ref{assump:Wigner},
and $D_1,\ldots,D_K \in \C^{n \times n}$ are
deterministic and diagonal matrices.
\item $(\cA,\|\cdot\|_\op,\tau)$ is the von Neumann algebra of
Assumption \ref{assump:free} with trace $\tau$,
containing the subalgebra $\cM \equiv \C^{n \times n}$
and a semicircular element, denoted here as $\h$, that is free of $\cM$.
\item $(\cX,\|\cdot\|_\op,\phi)$ is a von Neumann algebra with faithful, normal,
tracial state $\phi:\cX \to \C$,
and $\x \in \cX$ is self-adjoint and invertible.
\item We have
\begin{equation}\label{eq:xzbounds}
\|\x\|_\op,\|\x^{-1}\|_\op \leq \gamma,
\quad \|\z\|_\op \leq \sum_{k=1}^K \|D_k\|_\op\|\x_k\|_\op \leq \upsilon,
\quad \Im \z=\Im \sum_{k=1}^K D_k \otimes \x_k \geq \delta.
\end{equation}
\end{enumerate}
\end{asp}

We denote by $\cD \subset \cM$ the subalgebra of diagonal matrices, and
$\tau^\cD:\cA \to \cD$ the diagonal projection (\ref{eq:tauD}) onto $\cD$.
We remark that under this assumption, $\z$ admits an equivalent representation
\[\z=\sum_{i=1}^n E_{ii} \otimes \z_i,
\qquad \z_i=\sum_{k=1}^K [D_k]_{ii}\x_k \in \cX^+,
\qquad \Im \z_i \geq \delta,\]
and we will use these representations of $\z$ interchangeably.
We denote the limiting operator for the generalized resolvent $R$ by
\[\r=\left(\h \otimes \x-\z\right)^{-1}
\in \cA \otimes \cX,\]
its projection under $\tau^\cD \otimes 1$
(the $\cD\otimes\cX$-valued Stieltjes transform of
$\h\otimes\x$ evaluated at $\z$) by
\[\r_0=(\tau^\cD \otimes 1)[\r] \in \cD \otimes \cX,\]
and its projection under $\tau \otimes 1$ by
\[\m_0=(\tau \otimes 1)[\r] = (n^{-1}\Tr \otimes 1)[\r_0] \in \cX.\]
We use the indexing $R_{ij}=(\e_i \otimes 1)^* R(\e_j \otimes 1)$ and
$(\r_0)_{ij}=(\e_i \otimes 1)^* \r_0(\e_j \otimes 1)$,
and write $\x \prec \zeta$ or $\x=O_\prec(\zeta)$ 
for stochastic domination in the $L^p$-norms for $p \in [1,\infty)$
as discussed in Section \ref{subsec:domination}.

Our main result in this context is the following theorem.

\begin{thm}\label{aux:main theorem}
Uniformly over
$\x,\z$ satisfying Assumption \ref{assump:general},
\begin{equation}
(n^{-1}\Tr\,\otimes 1)[R]-\m_0 \prec n^{-1}.
\label{X-valued Stieltjes transform estimates}
\end{equation}
Also uniformly over $i \neq j \in \{1,\ldots,n\}$,
    \begin{align}
R_{ii}-(\r_0)_{ii} &\prec n^{-1/2},\label{DX-valued Stieltjes transform
estimates}\\
R_{ij} &\prec n^{-1/2},\label{off diagonal estimates}
    \end{align}
and uniformly over deterministic vectors $\u,\v\in\C^n$ with
$\|\u\|_2,\|\v\|_2 \leq \upsilon$,
    \begin{align}\label{isotropic estimates}
        (\u \otimes 1)^* [R-\r_0](\v \otimes 1) \prec n^{-1/2}.
    \end{align}
\end{thm}

In the remainder of this section, we prove Theorem \ref{aux:main theorem}.
Section~\ref{subsec:stability} discusses existence, uniqueness, and stability
of the solution to the relevant operator-valued fixed-point equation.
Section~\ref{subsec:weakestimate} proves a preliminary estimate in operator norm
\[\|R_{ii}-(\r_0)_{ii}\|_\op \prec n^{-\alpha},
\qquad \|R_{ij}\|_\op \prec n^{-\alpha} \text{ for } i \neq j\]
for some $\alpha>0$, by conducting the analysis for sufficiently large $|z|$
and extending the result to all fixed $z \in \C^+$ using the maximum modulus
principle. Section \ref{subsec:bootstrapping} improves this to the optimal
estimate of $n^{-1/2}$ in the $L^p$-norms for $p \in [1,\infty)$,
using the non-commutative Khintchine-type inequalities, fluctuation averaging
techniques, and an iterative bootstrapping argument. Section
\ref{subsec:proofgeneral} concludes the proof of Theorem \ref{aux:main theorem}.

All stochastic domination statements of this section are implicitly uniform over
$\x,\z$ satisfying Assumption \ref{assump:general}, indices $i,j,k,\ldots \in
\{1,\ldots,n\}$, subsets $S \subset \{1,\ldots,n\}$ having cardinality $|S| \leq
l$ for any fixed ($n$-independent) value $l \geq 1$, and unit vectors
$\u,\v \in \C^n$ satisfying $\|\u\|_2,\|\v\|_2 \leq \upsilon$.

\subsection{Operator fixed-point equations}\label{subsec:stability}

In the following we recall a few relevant notions from operator-valued free probability
theory \cite[Chapter 9.2]{MingoSpeicher2017}:
Let $\NC(m)$ be the lattice of non-crossing partitions of $\{1,\ldots,m\}$.
Associated to any von Neumann subalgebra $\cB$ of $\cY=\cA \otimes \cX$
and its conditional expectation $\tau^\cB:\cY \to \cB$
(c.f.\ Lemma \ref{conditional expectation}) is a system of $\cB$-valued
cumulants $(\kappa_\pi^\cB)_{\pi\in \NC(m)}$, which are $\C$-multilinear maps
$\kappa_\pi^\cB:\cY^m \to \cB$ satisfying the free moment-cumulant relation
\begin{equation}\label{eq:momentcumulant}
\tau^\cB(\y_1\y_2\ldots \y_m)=\sum_{\pi \in \NC(m)}
\kappa_\pi^\cB(\y_1,\y_2,\ldots,\y_m)
\end{equation}
and the bimodule properties
\begin{equation}\label{eq:bimodule}
\begin{aligned}
\kappa_\pi^\cB(\b\y_1,\y_2,\ldots,\y_{m-1},\y_m\b')&=\b
\kappa_\pi^\cB(\y_1,\ldots,\y_{m-1},\y_m)\b',\\
\kappa_\pi^\cB(\y_1,\ldots,\y_r\b,\y_{r+1},\ldots,\y_m)
&=\kappa_\pi^\cB(\y_1,\ldots,\y_r,\b\y_{r+1},\ldots,\y_m)
\end{aligned}
\end{equation}
for any $\y_1,\ldots,\y_m \in \cY$ and $\b,\b' \in \cB$. Fixing a self-adjoint
element $\y \in \cY$, for any invertible $\b \in \cB$ with $\|\b^{-1}\|_\op$
small enough, we may define the $\cB$-valued Cauchy-transform
\[G_\y^\cB(\b)=\tau^\cB[(\b-\y)^{-1}]=\sum_{m\geq 0}
\tau^\cB[\b^{-1}(\y\b^{-1})^m] \in \cB,\]
and for any $\b \in \cB$ with $\|\b\|_\op$ small enough, we may define the
$\cB$-valued R-transform
\[\cR_\y^\cB(\b)=\sum_{m\geq 1}\kappa_m^\cB(\y\b,\ldots,\y\b,\y) \in \cB\]
where $\kappa_m=\kappa_\pi$ for the partition
$\pi=\{\{1,\ldots,m\}\}$. Then, for any invertible $\b \in \cB$ with
$\|\b^{-1}\|_\op$ small enough, these transforms satisfy the Cauchy-R relation
\begin{equation}\label{eq:Cauchy-R relation}
G_\y^\cB(\b)=(\b-\cR_\y^\cB(G_\y^\cB(\b)))^{-1}.
\end{equation}

\begin{lem}\label{lem:uniquefixedpoint}
Let $\cD \subset \cM \subset \cA$ and $\cX$
be as in Assumption \ref{assump:general}. Let $\x \in \cX$ be self-adjoint,
and let $\h \in \cA$ be a semicircular element free of $\cM$.
\begin{enumerate}
\item[(a)] ($\cD \otimes \cX$-valued fixed point) For any
$\z \in (\cD \otimes \cX)^+=\{\z \in \cD \otimes \cX:\Im \z \geq \epsilon \text{
for some }\epsilon>0\}$, there exists a unique solution $\s \in
(\cD \otimes \cX)^+$ to the fixed-point equation
\begin{equation}\label{eq:generalfixedpoint}
\s=\left({-}\z-I_{n \times n}\otimes[\x\left(n^{-1}\Tr\otimes
1[\s]\right)\x]\right)^{-1}.
\end{equation}
This solution is given by $\r_0=(\tau^\cD \otimes 1)[(\h \otimes \x-\z)^{-1}]$.
\item[(b)] ($\cX$-valued fixed point) For any $\z_1,\ldots,\z_n \in \cX^+$,
there exists a unique solution $\m \in \cX^+$ to the fixed-point equation
\begin{equation}\label{eq:generalfixedpointreduced}
\m=\frac{1}{n}\sum_{i=1}^n \left({-}\z_i-\x\m\x\right)^{-1}.
\end{equation}
Setting $\z=\sum_{i=1}^n E_{ii} \otimes \z_i \in \cD \otimes \cX$,
this solution is given by $\m_0=(\tau \otimes 1)[(\h \otimes \x-\z)^{-1}]$.
\end{enumerate}
\end{lem}
\begin{proof}
For existence and uniqueness in part (a), we apply the general
result of \cite[Theorem 2.1]{HeltonFarSpeicher2007Operator}: Let $\cB$ be any
$C^*$-algebra and denote its right operator half-plane
\[\cB'=\{\a \in \cB:\Re \a=\tfrac{1}{2}(\a+\a^*) \geq
\epsilon \text{ for some } \epsilon>0\}.\]
Then for any $\b \in \cB'$ and any analytic mapping $\eta:\cB' \to \cB'$ that is
bounded on bounded domains of $\cB'$, there exists a unique solution $\s \in
\cB'$ to the equation $\s=[\b+\eta(\s)]^{-1}$.

Fixing $\z \in (\cD \otimes \cX)^+$, we take
$\cB=\cD \otimes \cX$, $\b={-}i\z/2$, and
$\eta(\s)={-}i\z/2+I \otimes \x(n^{-1}\Tr \otimes 1[\s])\x$, noting that
$\Re[\b]=\Im \z/2 \geq \epsilon>0$ and
$\Re \eta(\s)=\Im z/2+I \otimes \x(n^{-1}\Tr \otimes 1[\Re s])\x
\geq \epsilon>0$ for some $\epsilon>0$ whenever $\Re s \geq 0$ and
$\z \in (\cD \otimes \cX)^+$, because $\x$ is self-adjoint and
$n^{-1}\Tr \otimes 1$ is a positive map (Lemma \ref{conditional expectation}).
Thus, there exists a unique solution $\s' \in (\cD \otimes \cX)'$ to
\[\s'=\big({-}i\z+I \otimes \x(n^{-1}\Tr \otimes 1[\s'])\x\big)^{-1}.\]
Multiplying by $i$, there exists a unique solution $\s=i\s' \in (\cD \otimes
\cX)^+$ to (\ref{eq:generalfixedpoint}).

To show that this solution is given by $\r_0=(\tau^\cD \otimes 1)[(\h \otimes
\x-\z)^{-1}]$, we claim that the map $\s \mapsto I \otimes \x(n^{-1}\Tr \otimes
1[\s])\x$ is the $\cD \otimes \cX$-valued R-transform of $\h \otimes \x$,
and (\ref{eq:generalfixedpoint}) is the Cauchy-R relation over $\cD \otimes
\cX$. Namely, for any $\s \in \cD\otimes\cX$ with $\|\s\|_\op$ small enough,
\begin{equation}\label{eq:DX-valued R transform}
\cR_{\h \otimes \x}^{\cD \otimes
\cX}(\s)=I_{n \times n}\otimes\x\big(n^{-1}\Tr\otimes 1[\s]\big)\x.
\end{equation}
Indeed, since $\h$ is free of $\cM$ and hence of $\cD \subset \cM$,
it is readily checked by definition (c.f.\ \cite[Definition
9.5]{MingoSpeicher2017}) that $\h \otimes \x$ is free
of $\cB=\cD \otimes \cX$ with amalgamation over $1 \otimes \cX$
under the conditional expectation $\tau \otimes 1:\cA \otimes \cX \to \cX \cong
1 \otimes \cX$.
Then by \cite[Theorem 3.6]{NicaShlyakhtenkoSpeicher2002},
for any $\s\in\cD\otimes\cX$ with $\|\s\|_{\op}$ small enough,
    \begin{align*}
        \cR_{\h\otimes\x}^{\cD\otimes\cX}(\s)&=\sum_{m\geq
1}\kappa_m^{\cD\otimes\cX}((\h\otimes\x)\s,\ldots,(\h\otimes\x)\s,\h\otimes\x)\\
        &=\sum_{m\geq 1}\kappa_m^{1 \otimes \cX}((\h\otimes\x)(\tau\otimes
1)[\s],\ldots,(\h\otimes\x)(\tau\otimes 1)[\s],\h\otimes\x)\\
        &=\sum_{m\geq 1}\left(1 \otimes\left[\x \left(n^{-1}\Tr\otimes
1[\s]\right)\right]^{m-1}\x\right)\kappa_m^{1\otimes\cX}(\h\otimes 1,...,\h\otimes 1)
    \end{align*}
    where we used in the last step that
$\tau\otimes 1[\s]=n^{-1}\Tr\otimes 1[\s]$
for any $\s\in\cD\otimes\cX$, that $\h\otimes 1$ commutes with $1 \otimes\cX$,
and the bimodule properties (\ref{eq:bimodule}).
On the other hand, we have
\[\kappa_m^{1\otimes\cX}(\h \otimes 1,\ldots,\h \otimes 1)
=\kappa_m(\h,\ldots,\h)(1 \otimes 1)\]
where $\kappa_m(\h,\ldots,\h)$ are the scalar-valued free cumulants of $\h$:
This may be verified by expressing
$\kappa_m^{1\otimes\cX}(\h \otimes 1,\ldots,\h \otimes 1)$
in terms of moments under the conditional expectation $\tau \otimes 1$
using M\"obius-inversion of
(\ref{eq:momentcumulant}) (c.f.\ \cite[Eq.\ (9.19)]{MingoSpeicher2017}),
applying the identity $(\tau \otimes 1)[\a \otimes 1]=\tau(\a)(1 \otimes 1)$ for
each moment term, and
re-applying (\ref{eq:momentcumulant}) to express the result back in terms of
the scalar cumulants $\kappa_m(\h,\ldots,\h)$.
Here, $\h \in \cA$ is semi-circular, so
$\kappa_m(\h,\ldots,\h)=1$ if $m=2$ and 0 otherwise
(c.f.\ \cite[Example 11.21]{NicaSpeicher2006}). Thus only the $m=2$ term above
remains, and we obtain the claim (\ref{eq:DX-valued R transform}).

Then, identifying $\r_0(\z)=(\tau^\cD \otimes 1)[(\h \otimes \x-\z)^{-1}]
={-}G_{\h \otimes \x}^{\cD \otimes \cX}(\z)$, the Cauchy-R relation 
(\ref{eq:Cauchy-R relation}) implies that
\[\r_0(\z)=\left({-}\z-I \otimes \x\left(n^{-1}\Tr\otimes 1[\r_0(\z)]
\right)\x\right)^{-1}\]
for all $\z \in (\cD \otimes \cX)^+$ with $\|\z^{-1}\|_\op$ small enough.
Since the two sides of this equation are analytic over the operator half-plane
$\z \in (\cD\otimes\cX)^+$
and are equal on an open subset of $(\cD\otimes\cX)^+$, it follows from the identity principle 
they must be equal for all
$\z \in (\cD\otimes\cX)^+$, showing part (a) that $\r_0$ is the unique
solution of (\ref{eq:generalfixedpoint}) for any such $\z$.

For part (b),  observe that for any $\m \in \cX^+$,
\[(\z+I_{n \times n} \otimes \x\m\x)^{-1}
=\left(\sum_{i=1}^n E_{ii} \otimes (\z_i+\x\m\x)\right)^{-1}
=\sum_{i=1}^n E_{ii} \otimes (\z_i+\x\m\x)^{-1}.\]
Then setting $\m_0=n^{-1}\Tr \otimes 1[\r_0]=(\tau \otimes 1)[(\h \otimes
\x-\z)^{-1}]$ and
taking $n^{-1}\Tr \otimes 1$ on both sides of
(\ref{eq:generalfixedpoint}) with $\s=\r_0$ shows that $\m_0$ solves
(\ref{eq:generalfixedpointreduced}). To see that $\m_0$ is the unique solution,
observe that if $\m \in \cX^+$
solves (\ref{eq:generalfixedpointreduced}), then
defining $\s_i=(-\z_i - \x\m\x)^{-1}$
and $\s=\sum_i E_{ii} \otimes \s_i \in (\cD \otimes \cX)^+$, it follows from
(\ref{eq:generalfixedpointreduced}) that $\m=n^{-1}\Tr \otimes 1[\s]$, so
$\s$ solves (\ref{eq:generalfixedpoint}).
Then by the uniqueness claim of part (a), we must have $\s=\r_0$,
so $\m=n^{-1}\Tr \otimes 1[\r_0]=\m_0$. Hence this solution $\m_0 \in \cX^+$ to
(\ref{eq:generalfixedpointreduced}) is unique.
\end{proof}

We deduce from the above the following stability statements for approximate
solutions of these fixed-point equations.

\begin{cor}\label{cor:Stability of DX-valued fixed point}
Under Assumption \ref{assump:general}:
\begin{enumerate}[(a)]
\item Suppose $\s \in (\cD\otimes\cX)^+$ and $\Delta \in \cD\otimes\cX$
satisfy $\Im(\z+\Delta) \geq \delta/2$ and
\[\s=\left(-\z-\Delta - I_{n \times n}\otimes [\x\left(n^{-1}\Tr\otimes
1[\s]\right)\x]\right)^{-1}.\]
Then for any $p \in [1,\infty]$ (where $\|\cdot\|_\infty=\|\cdot\|_\op$),
\[\pnorm{\s-\r_0}{p} \leq 2\delta^{-2}\pnorm{\Delta}{p}.\]
\item Suppose $\m \in\cX^+$ and $\Delta\in\cX$ satisfy
$\Im(\z_i+\x\Delta\x) \geq \delta/2$ for all $i=1,\ldots,n$ and
\begin{align}\label{eq:X-valued perturbed fixed point}
\m=\frac{1}{n}\sum_{i=1}^n \left(-\z_i-\x\m\x\right)^{-1}+\Delta.
\end{align}
Then for any $p \in [1,\infty]$,
\[\pnorm{\m-\m_0}{p} \leq (1+2\gamma^2\delta^{-2})\pnorm{\Delta}{p}\]
\end{enumerate}
\end{cor}
\begin{proof}
For part (a),
defining $\r(\z)=(\h \otimes \x-\z)^{-1}$ and $\r_0(\z)=(\tau^\cD \otimes
1)[\r(\z)]$, Lemma \ref{lem:uniquefixedpoint}(a) implies that
$\s=\r_0(\z+\Delta)$. Hence
    \begin{align*}
        \pnorm{\s-\r_0}{p} &= \pnorm{\r_0(\z+\Delta)-\r_0(\z)}{p} = \pnorm{(\tau^\cD \otimes 1)[\r(\z+\Delta)\Delta\r(\z)]}{p}\\
        &\leq\|\r(\z+\Delta)\|_{\op}\|\Delta\|_p\|\r(\z)\|_\op\leq 2\delta^{-2}\pnorm{\Delta}{p}
    \end{align*}
    where we used $L^p$-contractivity of $\tau^\cD\otimes 1$ (Lemma
\ref{lemma:Lpcontraction}),
H\"older's inequality (Lemma \ref{Holder's inequality}), the given
conditions $\Im \z \geq \delta$ and $\Im(\z+\Delta) \geq \delta/2$,
and Lemma \ref{im bound inverse}.

Similarly for part (b), defining $\m_0(\z)=(n^{-1}\Tr \otimes 1)[\r(\z)]$
and setting $\m'=\m-\Delta$, we have
$\m'=n^{-1}\sum_i ({-}\z_i-\x\Delta\x-\x\m'\x)^{-1}$ so
Lemma \ref{lem:uniquefixedpoint}(a) implies that
$\m'=\m_0(\z+I \otimes \x\Delta\x)$. Then, using also $\|\x\|_\op \leq \gamma$,
\begin{align*}
\|\m-\m_0\|_p
&\leq \|\Delta\|_p+\|\m_0(\z+I \otimes \x\Delta\x)-\m_0(\z)\|_p\\
&\leq \|\Delta\|_p+2\delta^{-2}\|I \otimes \x\Delta\x\|_p
\leq (1+2\gamma^2\delta^{-2})\|\Delta\|_p.
\end{align*}
\end{proof}

We close this section with an analysis of the quantitative invertibility of two
linear operators $\cL_1,\cL_2:\cX \to \cX$ in the $L^p$-norms, defined as
\[\cL_1(\a)=\a-\frac{1}{n}\sum_{i=1}^n (\r_0)_{ii}\x\a\x(\r_0)_{ii},
\qquad \cL_2(\a)=\a-\frac{1}{n}\sum_{i=1}^n (\r_0)_{ii}\x\a\x(\r_0)_{ii}^*.\]
The invertibility of $\cL_1$ is more immediate, and follows from
differentiating the preceding fixed-point equation.
We state this in the following lemma.

    \begin{lem}\label{lemma:linearmap1}
Under Assumption \ref{assump:general}, consider the linear operator
$\cL_1:\cX \to \cX$ given by
\[\cL_1(\a)=\a-\frac{1}{n}\sum_{i=1}^n (\r_0)_{ii}\x\a\x(\r_0)_{ii}.\]
Then $\cL_1$ is invertible, and for any $\a \in \cX$ and $p \in [1,\infty]$
(where $\|\cdot\|_\infty=\|\cdot\|_\op$),
        \begin{equation*}
            \|\cL_1^{-1}(\a)\|_p \leq \gamma^2(\gamma^2+\delta^{-2})\|\a\|_p.
        \end{equation*}
    \end{lem}
    \begin{proof}
        For any $\b \in \cX$ with $\|\b\|_\op<\delta/2$
(which ensures $\Im(\z+1 \otimes \b)>\delta/2>0$), define
\[f(\b)=(\tau \otimes 1)[(\h \otimes \x-\z-1 \otimes \b)^{-1}],
\qquad \omega(\b)=\x^{-1}\b\x^{-1}+f(\b).\]
Then Lemma \ref{lem:uniquefixedpoint}(b) applied at $\z+1 \otimes \b$ shows
that $f(\b)$ is the unique solution in $\cX^+$ to the fixed-point equation
$f(\b)=n^{-1}\sum_i ({-}\z_i-\b-\x f(\b)\x)^{-1}$, i.e.
        \begin{equation}\label{eq:fixedpointOmega}
            \x^{-1}\b\x^{-1}=\omega(\b)-\frac{1}{n}\sum_{i=1}^n ({-}\z_i-\x\,\omega(\b)\x)^{-1}.
        \end{equation}
For all $\omega \in \cX^+$, define
        \begin{equation}\label{eq:rmt subordination}
        g(\omega)=\omega-\frac{1}{n}\sum_{i=1}^n ({-}\z_i-\x\,\omega\x)^{-1}.
        \end{equation}
Then for all $\b \in \cX$ with $\|\b\|_\op<\delta/2$, since
$\omega(\b)$ satisfies (\ref{eq:fixedpointOmega}), we have
$g(\omega(\b))=\x^{-1}\b\x^{-1}$.

Let us write $D\omega(\b),Dg(\omega)$ for the Fr\'echet derivatives of
$\omega(\b)$ and $g(\omega)$ as linear maps on $\cX$. 
Recalling $\r=(\h \otimes \x-\z)^{-1}$ and
differentiating $\omega(\b)=\x^{-1}\b\x^{-1}+(\tau \otimes 1)[(\h \otimes
\x-\z-1 \otimes \b)^{-1}]$ at $\b=0$, for any $\a \in \cX$ and $p \in
[1,\infty]$ we have
\begin{align}
\|D\omega(0)[\a]\|_p &=\|\x^{-1}\a\x^{-1}+(\tau \otimes 1)[\r(1 \otimes
\a)\r]\|_p\notag\\
&\leq \gamma^2\|\a\|_p+\|\r\|_\op^2\|1 \otimes \a\|_p
\leq (\gamma^2+\delta^{-2})\|\a\|_p.\label{eq:Domegabound}
\end{align}
Here, the second line uses H\"older's inequality (Lemma \ref{Holder's
inequality}), $\|\x^{-1}\|_\op \leq \gamma$,
contractivity of the conditional expectation $\tau \otimes 1$ in $L^p$
(Lemma \ref{lemma:Lpcontraction}), 
and $\|\r\|_\op \leq \delta^{-1}$ by the assumption $\Im \z \geq \delta$ and
Lemma \ref{im bound inverse}.
On the other hand, differentiating (\ref{eq:rmt subordination})
at $\omega_0=\omega(0)$ and using $({-}\z_i-\x\omega_0\x)^{-1}
=({-}\z_i-\x\m_0\x)^{-1}=(\r_0)_{ii}$
by Lemma \ref{lem:uniquefixedpoint}(a), we have
\[Dg(\omega_0)[\a]=\a-\frac{1}{n}\sum_{i=1}^n (\r_0)_{ii}\x\a\x(\r_0)_{ii}
=\cL_1(\a).\]
Then differentiating both sides of the identity
$g(\omega(\b))=\x^{-1}\b\x^{-1}$ at $\b=0$ and evaluating the derivative at
$\x\a\x$ for any $\a \in \cX$ gives
\[\cL_1(D\omega(0)[\x\a\x])=Dg(\omega_0)[D\omega(0)[\x\a\x]]=\x^{-1}[\x\a\x]\x^{-1}=\a.\]
The bound (\ref{eq:Domegabound}) for $p=\infty$ shows that
$\a \mapsto D\omega(0)[\x\a\x]$ defines a bounded linear operator on
$\cX$, so $\cL_1$ is invertible with inverse given explicitly by
$\cL_1^{-1}(\a)=D\omega(0)[\x\a\x]$. Finally, (\ref{eq:Domegabound})
and the condition $\|\x\|_\op \leq \gamma$
imply that $\|\cL_1^{-1}(\a)\|_p \leq \gamma^2(\gamma^2+\delta^{-2})\|\a\|_p$.
\end{proof}

We now turn to the invertibility of $\cL_2$. We use an idea from
\cite[Section 4.2]{ajanki2019stability}, which relies on the observation
that $\cL_2^{-1}$ may be controlled at a single positive element of $\cX$
by taking imaginary parts of (\ref{eq:generalfixedpointreduced}), and that
the linear map $\a \mapsto (\r_0)_{ii}\x\a\x(\r_0)_{ii}^*$ is
positivity-preserving. For finite-dimensional matrix algebras, this implies
that this map has a Perron-Frobenius eigenvector $\a \geq 0$ in the
positive cone. The analyses of \cite{ajanki2019stability} construct a
symmetrized version of this map that is self-adjoint, so that its $L^2 \to L^2$
operator norm coincides with its spectral radius, and then apply a
$L^2$-inner-product of the Perron-Frobenius eigenvector with the imaginary part
of (\ref{eq:generalfixedpointreduced}) to deduce a quantitative bound on
$\|\cL_2^{-1}\|_{L^2 \to L^2}$.

We adapt this idea to address two additional challenges in our setting: First, 
as the map $\a \mapsto (\r_0)_{ii}\x\a\x(\r_0)_{ii}^*$ is non-compact and may
have continuous spectrum, we can only guarantee the existence of an approximate
Perron-Frobenius eigenvector (c.f.\ Lemma \ref{lemma:approxperronfrobenius}
below). For reasons that will be clear in the proof, the approximation error
must be controlled in $L^1$ rather than $L^2$, and thus we implement a version
of this argument in the $L^1$-$L^\infty$ duality rather than in a Hilbert space
setting, to obtain a bound for $\|\cL_2^{-1}\|_{L^\infty \to L^\infty}$.
Second, as we will also require a bound on $\|\cL_2^{-1}\|_{L^p \to L^p}$ for
each $p \in [1,\infty)$, we carry out a dual version of this argument
to obtain also a bound for $\|\cL_2^{-1}\|_{L^1 \to L^1}$, and hence
deduce a bound on $\|\cL_2^{-1}\|_{L^p \to L^p}$ via the Riesz-Thorin
interpolation.

\begin{lem}[Approximate Perron-Frobenius
eigenvector]\label{lemma:approxperronfrobenius}
Let $(\cX,\|\cdot\|)$ be a real Banach space, and $\cK \subset \cX$ a
closed convex cone such that $\cK \cap (-\cK)=\{0\}$ ($\cK$ is proper),
$\cX=\{x-y:x,y \in \cK\}$ ($\cK$ is generating), and for some $C>0$ we have
$\|x\| \leq C\|x+y\|$ whenever $x,y \in \cK$ ($\cK$ is normal).

Let $T:\cX \to \cX$ be a bounded linear operator such that $T(\cK) \subseteq
\cK$, and let $r(T)$ be the spectral radius of its complexification
$T_\C:\cX+i\cX \to \cX+i\cX$. Then $r(T)$ is an element of the spectrum of $T$.
Furthermore, for any $\epsilon>0$, there exists an approximate eigenvector
$x$ of $T$ such that
\[x \in \cK, \qquad \|x\|=1, \qquad \|T(x)-r(T)x\|<\epsilon.\]
\end{lem}
\begin{proof}
See \cite[Lemma 3.5]{gluck2021stability}.
\end{proof}

\begin{lem}\label{lemma:linearmap2}
Under Assumption \ref{assump:general}, consider the linear operator
$\cL_2:\cX \to \cX$ given by
\[\cL_2(\a)=\a-\frac{1}{n}\sum_{i=1}^n (\r_0)_{ii}\x\a\x(\r_0)_{ii}^*.\]
Then $\cL_2$ is invertible, and for any $\a \in \cX$ and $p \in [1,\infty]$
(where $\|\cdot\|_\infty=\|\cdot\|_\op$),
\begin{equation}\label{eq:Linvbound}
\|\cL_2^{-1}(\a)\|_p \leq \frac{2\gamma^4(\upsilon+\gamma^2\delta^{-1})^2}{\delta^2}\|\a\|_p.
\end{equation}
\end{lem}
\begin{proof}
Denote $L^p \equiv L^p(\cX)$ for the non-commutative $L^p$-spaces associated to
$\cX$ (Appendix \ref{appendix:background})
and define $\cF:L^1 \to L^1$ and $\cF':L^\infty \to L^\infty$ by
\[\cF(\b)=\frac{1}{n}\sum_{i=1}^n \x(\r_0)_{ii}^*\b(\r_0)_{ii}\x,
\qquad \cF'(\a)=\frac{1}{n}\sum_{i=1}^n (\r_0)_{ii}\x\a\x(\r_0)_{ii}^*.\]
Here $\cF$ and $\cF'$ are bounded linear operators on the Banach spaces
$(L^1,\|\cdot\|_1)$ and $(L^\infty,\|\cdot\|_\op)$,
by H\"older's inequality and the bounds $\|(\r_0)_{ii}\|_\op \leq \|\r_0\|_\op \leq
\delta^{-1}$ and $\|\x\|_\op \leq \gamma$.
Identifying the dual $(L^1)^*$ with $L^\infty$ via the isometry
$\a \in L^\infty \mapsto \ell_\a \in (L^1)^*$ where $\ell_\a(\b)=\phi(\b\a)$
(Lemma \ref{Holder's inequality}), for any $\a \in L^\infty$ and $\b \in L^1$
we have $\ell_\a(\cF(\b))=\phi(\cF(\b)\a)=\phi(b\cF'(\a))
=\ell_{\cF'(\a)}(\b)$, so $\cF'$ is the Banach space adjoint of $\cF$.

Let $r(\cF)$ be the spectral radius of $\cF$ as an operator on $L^1$. Note that
$\cF$ restricts to a bounded linear operator on the real Banach space of
self-adjoint elements $L^1_\text{sa}=\{\a \in L^1:\a=\a^*\}$ (whose
complexification is $\cF$ itself), which furthermore
preserves the positive cone $L^1_\text{pos}=\{\a \in L^1_\text{sa}:\a \geq 0\}$.
Any $\a \in L^1_\text{sa}$ may be decomposed as $\a=\a_+-\a_-$ with $\a_+,\a_-
\in L^1_\text{pos}$ via $\a_+=(|\a|+\a)/2$ and $\a_-=(|\a|-\a)/2$,
and we have $\|\a\|_1=\tau(\a) \leq \tau(\a+\b)=\|\a+\b\|_1$ for all
$\a,\b \in L^1_\text{pos}$ by positivity of the trace.
Thus $L^1_\text{pos}$ is a proper, generating, and
normal cone in $L^1_\text{sa}$, so Lemma \ref{lemma:approxperronfrobenius}
ensures the existence of some $\sv \in L^1$ satisfying
\begin{equation}
\sv=\sv^*, \qquad \sv \geq 0,
\qquad \|\sv\|_1=1, \qquad \cF(\sv)=r(\cF)\sv+\Delta \text{ where }
\|\Delta\|_1<\epsilon.
\end{equation}

Lemma \ref{lem:uniquefixedpoint} shows that $\m_0$ satisfies the
fixed-point equation $\m_0=n^{-1}\sum_{i=1}^n ({-}\z_i-\x\m_0\x)^{-1}
=n^{-1}\sum_{i=1}^n (\r_0)_{ii}$. Taking imaginary parts,
this gives $\Im \m_0=n^{-1}\sum_{i=1}^n (\r_0)_{ii}(\Im \z_i+\x[\Im
\m_0]\x)(\r_0)_{ii}^*$,
which may be rearranged as
\begin{equation}\label{eq:Lfixedpoint}
\cL_2[\Im \m_0]=(\Id-\cF')[\Im \m_0]=\su, \qquad
\text{ for } \su=\frac{1}{n}\sum_{i=1}^n (\r_0)_{ii}[\Im \z_i](\r_0)_{ii}^*.
\end{equation}
Multiplying on the left by the above approximate eigenvector $\sv$ and taking
the trace,
\[\phi(\sv\su)=\phi(\sv \Im \m_0)-\phi(\sv\cF'[\Im \m_0])
=\phi(\sv \Im \m_0)-\phi(\cF[\sv]\Im \m_0)
=(1-r(\cF))\phi(\sv\Im \m_0)-\phi(\Delta\Im \m_0).\]
Thus
\[1-r(\cF)=\frac{\phi(\sv\su)+\phi(\Delta\Im \m_0)}{\phi(\sv\Im \m_0)}.\]
By H\"older's inequality and the bound $\|\Im \m_0\|_\op \leq \|\m_0\|_\op \leq
\delta^{-1}$, we have
\begin{equation}\label{eq:specradbound}
0 \leq \phi(\sv\Im \m_0) \leq \delta^{-1}, \qquad
|\phi(\Delta\Im \m_0)| \leq \delta^{-1}\epsilon.
\end{equation}
We have also $\Im \z_i \geq \delta$
and $\|(\r_0)_{ii}^{-1}\|_\op=\|\z_i+\x\m_0\x\|_\op
\leq \|\z\|_\op+\|\x\|_\op^2\|\m_0\|_\op \leq \upsilon+\gamma^2\delta^{-1}$,
implying that
\begin{equation}\label{eq:ulowerbound}
\su \geq \delta \cdot \frac{1}{n}\sum_{i=1}^n
(\r_0)_{ii}(\r_0)_{ii}^* \geq c(\gamma,\upsilon,\delta)
\qquad \text{ for }
c(\gamma,\upsilon,\delta)=\delta(\upsilon+\gamma^2\delta^{-1})^{-2}.
\end{equation}
Then by positivity of $\sv$, positivity of the trace, and the normalization
$\|\sv\|_1=1$, we have
$\phi(\sv\su)=\phi(\sv^{1/2}\su\sv^{1/2})
\geq c(\gamma,\upsilon,\delta)\phi(\sv)=c(\gamma,\upsilon,\delta)$.
Applying these bounds above and then taking $\epsilon \to 0$, we obtain
\begin{equation}\label{eq:1minusr}
1-r(\cF) \geq \delta c(\gamma,\upsilon,\delta)>0.
\end{equation}

Since $\cF'$ is the adjoint of $\cF$ and thus shares its spectrum,
this shows $r(\cF')=r(\cF)<1$. Then
$\cL_2=\Id-\cF'$ is invertible as a bounded linear operator on $L^\infty$, and
its inverse has the Neumann series representation
\[\cL_2^{-1}=\sum_{k=0}^\infty (\cF')^k\]
which is convergent in the induced operator norm $\|\cdot\|_{L^\infty \to
L^\infty}$ (by Gelfand's formula $r(\cF')=\lim_{k \to \infty} \|{\cF'}^k\|_{L^\infty \to L^\infty}^{1/k}$). Thus, for
any $\a \in L^\infty$, we have $\cL_2^{-1}(\a)=\sum_{k=0}^\infty (\cF')^k[\a]$
which is convergent under $\|\cdot\|_\op$. We note that $(\cF')^k$ is also
positivity-preserving, i.e.\ if $\a \geq 0$, then $(\cF')^k[\a] \geq 0$. 
Thus if $\a \geq 0$, then $\cL_2^{-1}(\a) \geq 0$ since
the positive cone is closed under $\|\cdot\|_\op$. To summarize, we have
shown that $\cL_2:L^\infty \to L^\infty$ is invertible, and
\begin{equation}\label{eq:Linvpositive}
\a \geq 0 \Longrightarrow \cL_2^{-1}(\a) \geq 0.
\end{equation}
Now take any $\b \in L^\infty$ self-adjoint.
Applying (\ref{eq:ulowerbound}), we have
\[\cL_2(\b) \leq \|\cL_2(\b)\|_\op \cdot 1_\cX \leq \|\cL_2(\b)\|_\op \cdot
c(\gamma,\upsilon,\delta)^{-1} \cdot \su\]
so the monotonicity of $\cL_2^{-1}$ in (\ref{eq:Linvpositive}) and explicit form
of $\cL_2^{-1}[\su]$ in (\ref{eq:Lfixedpoint}) imply
\[\b \leq \|\cL_2(\b)\|_\op \cdot c(\gamma,\upsilon,\delta)^{-1} \cL_2^{-1}[\su]
=\|\cL_2(\b)\|_\op \cdot c(\gamma,\upsilon,\delta)^{-1} \Im \m_0.\]
Similarly $\b \geq {-}\|\cL_2(\b)\|_\op \cdot c(K,\upsilon,\delta)^{-1} \Im \m_0$.
Applying again $\|\Im \m_0\|_\op \leq \delta^{-1}$, this shows
for every $\b \in L^\infty$ self-adjoint that
$\|\b\|_\op \leq \delta^{-1}c(K,\upsilon,\delta)^{-1} \|\cL_2(\b)\|_\op$.
Then for any (non-self-adjoint) $\a \in L^\infty$,
noting that $(\|\Re \a\|_\op+\|\Im \a\|_\op)/2
\leq \|\a\|_\op \leq \|\Re \a\|_\op+\|\Im \a\|_\op$ and that $\Re
\cL_2(\a)=\cL_2(\Re \a)$ and $\Im \cL_2(\a)=\cL_2(\Im \a)$, this implies
\begin{equation}\label{eq:Linftystability}
\|\a\|_\op \leq 2\delta^{-1}c(\gamma,\delta,\upsilon)^{-1}\|\cL_2(\a)\|_\op
=\frac{2(\upsilon+\gamma^2\delta^{-1})^2}{\delta^2}\|\cL_2(\a)\|_\op,
\end{equation}
which implies (\ref{eq:Linvbound}) for $p=\infty$.

Next, we show (\ref{eq:Linvbound}) for $p=1$ using a dual argument: As $L^1$
is not reflexive, we reverse the roles of
$\cF,\cF'$ and define the bounded linear operators $\cG:L^1 \to L^1$
and $\cG':L^\infty \to L^\infty$ by
\[\cG(\b)=\frac{1}{n}\sum_{i=1}^n (\r_0)_{ii}\x\b\x(\r_0)_{ii}^*,
\qquad \cG'(\a)=\frac{1}{n}\sum_{i=1}^n \x(\r_0)_{ii}^*\a(\r_0)_{ii}\x.\]
Then again $\cG'$ is the adjoint of $\cG$. For any $\epsilon>0$, Lemma
\ref{lemma:approxperronfrobenius} shows there exists $\sv \in L^1$ satisfying
\[\sv=\sv^*, \qquad \sv \geq 0, \qquad \|\sv\|_1=1,
\qquad \cG(\sv)=r(\cG)\sv+\Delta \text{ where } \|\Delta\|_1<\epsilon.\]
Now taking imaginary parts of the fixed-point equation
$\m_0^*=n^{-1}\sum_{i=1}^n ({-}\z_i^*-\x\m_0^*\x)^{-1}$, we have
$\Im(\m_0^*)=n^{-1}\sum_{i=1}^n (\r_0)_{ii}^*(\Im (\z_i^*)+\x[\Im(\m_0^*)]\x)
(\r_0)_{ii}$.
Negating and applying $\Im(\a^*)={-}\Im \a$, this gives
\[\Im \m_0=\x^{-1}\cG'(\x[\Im \m_0]\x)\x^{-1}+\frac{1}{n}\sum_{i=1}^n
(\r_0)_{ii}^*[\Im \z_i](\r_0)_{ii}.\]
Thus, in place of (\ref{eq:Lfixedpoint}) we have
\begin{equation}\label{eq:Gfixedpoint}
(\Id-\cG')[\x(\Im \m_0)\x]=\w,
\qquad \text{ for } \w=\frac{1}{n}\sum_{i=1}^n \x(\r_0)_{ii}^*[\Im
\z_i](\r_0)_{ii}\x.
\end{equation}
In place of (\ref{eq:specradbound}) and (\ref{eq:ulowerbound}), we may apply
$\phi(\sv\x[\Im \m_0]\x) \leq \gamma^2\delta^{-1}$ and
$\w \geq \gamma^{-2}c(\gamma,\upsilon,\delta)$,
where the first inequality uses $\|\x\|_\op \leq \gamma$ and the
second uses $\|\x^{-1}\|_\op \leq \gamma$ and
$c(\gamma,\upsilon,\delta)$ as defined in (\ref{eq:ulowerbound}).
Then, multiplying (\ref{eq:Gfixedpoint}) by $\sv$, taking the trace, and then
taking the limit $\epsilon \to 0$, we obtain similarly to (\ref{eq:1minusr})
that $1-r(\cG) \geq \delta \gamma^{-4}c(\gamma,\upsilon,\delta)>0$.
This implies that $\Id-\cG':L^\infty \to L^\infty$
is invertible with positivity-preserving inverse,
and repeating the preceding arguments gives, for any $\a \in L^\infty$,
\[\|\a\|_\op \leq
\frac{2\gamma^4(\upsilon+\gamma^2\delta^{-1})^2}{\delta^2}\|(\Id-\cG')[\a]\|_\op,\]
Since $(\Id-\cG')^{-1}$ is the adjoint of $\cL_2^{-1}=(\Id-\cG)^{-1}:L^1
\to L^1$, we have
$\|(\Id-\cG')^{-1}\|_{L^\infty \to L^\infty}=\|\cL_2^{-1}\|_{L^1 \to L^1}$.
Thus, this shows also for any $\a \in L^1$ that
\begin{equation}\label{eq:L1stability}
\|\a\|_{L^1} \leq \frac{2\gamma^4(\upsilon+\gamma^2\delta^{-1})^2}{\delta^2}
\|\cL_2(\a)\|_{L^1},
\end{equation}
which is the desired result (\ref{eq:Linvbound}) for $p=1$.

Finally, the result (\ref{eq:Linvbound}) for general $p \in [1,\infty]$ follows
from the bounds for $\|\cL_2^{-1}\|_{L^\infty \to L^\infty}$ and
$\|\cL_2^{-1}\|_{L^1 \to L^1}$ already shown, and the Riesz-Thorin interpolation
(Lemma \ref{lemma:rieszthorin}).
\end{proof}

\subsection{Weak estimates in operator norm}\label{subsec:weakestimate}

\begin{lem}\label{lemma:weakopnormprelim}
Under Assumption \ref{assump:general}, there exists a constant $\alpha \in
(0,1/2)$
depending only on $\gamma,\upsilon,\delta$ such that for any $l \geq 1$, $D>0$,
and all $n \geq n_0(l,K,\gamma,\upsilon,\delta,D)$,
with probability at least $1-n^{-D}$,
\begin{align}
\sup_{S \subset \{1,\ldots,n\}:|S| \leq l}\;\sup_{i \notin S}
\pnorm{\sum_r^{(iS)}\left(|h_{ir}|^2-\frac{1}{n}\right)R_{rr}^{(iS)} + 
\sum_{r\neq s}^{(iS)}h_{ir}R_{rs}^{(iS)}h_{si}}{\op} < n^{-\alpha},
\label{eq:opnormest1}\\
\sup_{S \subset \{1,\ldots,n\}:|S| \leq l}\;
\sup_{i,j \notin S:i \neq j}
\pnorm{\sum_r^{(iS)}h_{ir}R_{rj}^{(iS)}}{\op} < n^{-\alpha}.
\label{eq:opnormest2}
\end{align}
\end{lem}
\begin{proof}
We present the argument for (\ref{eq:opnormest1}): 
Take any $S \subset \{1,\ldots,n\}$ with $|S| \leq l$, and any $i \notin S$.
Let $H^{(iS)}$ be as defined in Section \ref{subsec:resolventidentities},
and let $\bh_i^{(iS)} \in \C^n$ be the $i^\text{th}$ column of $H^{(iS)}$, i.e.\
the vector with entries $(\bh_i^{(iS)})_j=h_{ji}$
for $j \notin S \cup \{i\}$ and 0 otherwise.
Let $(\cH,\langle \cdot,\cdot \rangle_\cH)$ be the Hilbert space on
which $\cX$ acts. For any
unit vectors $\xi,\zeta \in \cH$, define the linear functional
$f_{i,\xi,\zeta}^{(S)}:\C^{n \times n} \otimes \cX \to \C$ by
\begin{align*}
f_{i,\xi,\zeta}^{(S)}(M)&=\left\langle
\xi,\,\left(\sum_r^{(iS)}\left(|h_{ir}|^2-\frac{1}{n}\right)M_{rr} +
\sum_{r\neq s}^{(iS)}h_{ir}M_{rs}h_{si}\right)\zeta\right\rangle_\cH\\
&=\left\langle
\xi,\,\left((\bh_i^{(iS)} \otimes 1)^* M (\bh_i^{(iS)} \otimes 1)
-(n^{-1}\Tr^{(iS)} \otimes 1) M\right)\zeta\right\rangle_\cH
\quad \text{ for } \Tr^{(iS)} M:=\sum_j^{(iS)} M_{jj}
\end{align*}
where $M_{rs}=(\e_r \otimes 1)^* M(\e_s \otimes 1)$ is the $\cX$-valued $(r,s)$
entry of $M$. Define
$M^{(iS)}:\C^+ \to \C^{n \times n} \otimes \cX$ by
\[M^{(iS)}(z)=\left(H^{(iS)}\otimes\x-\z+i(\delta/2)-z\right)^{-1}
\in \C^{n \times n} \otimes \cX.\]
Note that $\Im (\z-i(\delta/2)+z)>\delta/2$ for all $z \in \C^+$, so by Lemma
\ref{im bound inverse}, this inverse is well-defined and
\begin{equation}\label{eq:Mibound}
\|M^{(iS)}(z)\|_\op \leq 2/\delta \text{ for all } z \in \C^+.
\end{equation}
By definition we have $R^{(iS)}=M^{(iS)}(i\delta/2)$, so
the operator norm to be bounded in (\ref{eq:opnormest1}) is
\[\pnorm{\sum_r^{(iS)}\left(|h_{ir}|^2-\frac{1}{n}\right)R_{rr}^{(iS)} + 
\sum_{r\neq s}^{(iS)}h_{ir}R_{rs}^{(iS)}h_{si}}{\op}
=\sup_{\xi,\zeta \in \cH:\|\xi\|_\cH^2=\|\zeta\|_\cH^2=1}\,
f_{i,\xi,\zeta}^{(S)} \circ M^{(iS)}(i\delta/2).\]

Set $C_0=3\gamma+(3/2)\upsilon$, $\cD=\{z\in\C^+:|z| \geq 2C_0\}$, and
$\epsilon=0.1$, and define the event
\[\cE=\bigcap_{S \subset \{1,\ldots,n\}:|S| \leq l}\;\bigcap_{i \notin S}\;
\bigcap_{\xi,\zeta:\|\xi\|_\cH^2=\|\zeta\|_\cH^2=1}
\left\{\sup_{z\in \cD}|f_{i,\xi,\zeta}^{(S)} \circ M^{(iS)}(z)|\leq
n^{-1/2+\epsilon}\right\} \cap\{\|\bh_i^{(iS)}\|_2\leq 3\}.\]
Noting that $f_{i,\xi,\zeta}^{(S)} \circ M^{(iS)}(z)$ is an analytic function of
$z \in \C^+$,
we apply Lemma \ref{lemma:maximummodulus} on this event $\cE$: Let $a=3C_0$,
and set $r_0<r_1<r_2<0$ such that
\[a\frac{1-e^{r_0}}{1+e^{r_0}}=2C_0,
\qquad a\frac{1-e^{r_1}}{1+e^{r_1}}=\delta/2,
\qquad a\frac{1-e^{r_2}}{1+e^{r_2}}=\delta/4.\]
Then defining $S_r$ as in Lemma \ref{lemma:maximummodulus}, we have
$S_{r_0} \subset \cD$ and $i\delta/2 \in S_{r_1}$. On $\cE$, the
inclusion $S_{r_0} \subset \cD$ implies
$|f_{i,\xi,\zeta}^{(S)} \circ M^{(iS)}(z)|\leq n^{-1/2+\epsilon}$ for all $z \in S_{r_0}$,
and the bounds (\ref{eq:Mibound}) and $\|\bh_i^{(iS)}\|_2 \leq 3$ imply
$|f_{i,\xi,\zeta}^{(S)} \circ M^{(iS)}(z)| \leq 20/\delta$ for
all $z \in \C^+$. Thus Lemma \ref{lemma:maximummodulus} shows
\begin{align*}
\log |f_{i,\xi,\zeta}^{(S)} \circ M^{(iS)}(i\delta/2)|
&\leq \sup_{z \in S_{r_1}} \log |f_{i,\xi,\zeta}^{(S)} \circ M^{(iS)}(z)|\\
&\leq \frac{|r_1-r_2|}{|r_2-r_0|}\sup_{z \in S_{r_0}} \log
|f_{i,\xi,\zeta}^{(S)} \circ M^{(iS)}(z)|
+\frac{|r_1-r_0|}{|r_2-r_0|}\sup_{z \in S_{r_2}} \log |f_{i,\xi,\zeta}^{(S)}
\circ M^{(iS)}(z)|\\
&\leq -\alpha \log n
\end{align*}
for some constant $\alpha \in (0,1/2)$ depending only on
$\gamma,\upsilon,\delta$, and for all $n \geq n_0(\gamma,\upsilon,\delta)$. Thus
(\ref{eq:opnormest1}) holds on the event $\cE$ for all such $n$.

We now check that $\P[\cE] \geq 1-n^{-D}$ for all $n \geq n_0(l,K,D)$:
By a standard tail bound for the operator norm
(see e.g.\ \cite[Theorem 7.3]{ErdosEtAl2013Local}), the event
\begin{align*}
\cE_0&=\bigcap_{S \subset \{1,\ldots,n\}:|S| \leq l}\; \bigcap_{i \notin S}
 \left\{\|H^{(iS)}\|_{\op} \leq 3\right\}
\subset \{\|H\|_\op \leq 3\}
\end{align*}
holds with probability at least $1-n^{-D}$ for all $n \geq n_0(D)$. On
$\cE_0$, we have also $\|\bh_i^{(iS)}\|_2 \leq 3$.
Set for notational convenience $D_0={-}H^{(iS)}$
(which depends implicitly on $i$ and $S$), $\x_0=\x$, $D_{K+1}={-}i\delta/2$,
and $\x_{K+1}=1$ , so that
\[H^{(iS)}\otimes\x-\z+i(\delta/2) ={-}\sum_{k=0}^{K+1} D_k\otimes\x_k.\]
By the assumptions (\ref{eq:xzbounds}), we have on $\cE_0$ that
\[\pnorm{H^{(iS)}\otimes\x-\z+i(\delta/2)}{\op}
\leq \sum_{k=0}^{K+1} \|D_k\|_\op \|\x_k\|_\op
\leq 3\gamma+(3/2)\upsilon=C_0.\]
Then, since $|z| \geq 2C_0$ for $z \in \cD$, the series
expansion of $M^{(iS)}(z)$ in $z^{-1}$ is absolutely convergent in operator
norm, and
\begin{equation}\label{eq:Mseries}
M^{(iS)}(z)={-}\sum_{t=0}^{T(n)}
z^{-(t+1)}\left({-}\sum_{k=0}^{K+1} D_k \otimes \x_k\right)^t + J^{(iS)}(z),
\qquad \|J^{(iS)}(z)\|_\op \leq n^{-1}
\end{equation}
for some $T(n) \leq C\log n$ and an absolute constant $C>0$.
On $\cE_0$, using again $\|\bh_i^{(iS)}\|_2 \leq 3$,
we have $|f_{i,\xi,\zeta}^{(S)}(M)| \leq
10\|M\|_\op$ for any $M \in \C^{n \times n} \otimes \cX$.
Thus, for the remainder of (\ref{eq:Mseries}),
\begin{equation}\label{eq:remainderbound}
|f_{i,\xi,\zeta}^{(S)} \circ J^{(iS)}(z)|
\leq 10\|J^{(iS)}(z)\|_\op \leq 10n^{-1}.
\end{equation}

For the leading terms of (\ref{eq:Mseries}), let us expand
\begin{equation}\label{eq:tensorexpand}
\left(\sum_{k=0}^{K+1} D_k\otimes\x_k\right)^t
=\sum_{w \in \cW_t} \underbrace{w(D_0,\ldots,D_{K+1})}_{:=w(D^{(iS)})} \otimes
\underbrace{w(\x_0,\ldots,\x_{K+1})}_{:=w(\x)}
\end{equation}
where $\cW_t$ denotes the set of all length-$t$ words (i.e.\ non-commutative
degree-$t$ monic monomials) in $K+2$ variables. For $t=0$, we use the convention
$\cW_t=\{o\}$ where $o$ is the word of length 0, with $o(D^{(iS)})=I$ and
$o(\x)=1$. Then for each $w \in \cW_t$,
by the definition of $f_{i,\xi,\zeta}^{(S)}$,
we have the factorization
\begin{align}
f_{i,\xi,\zeta}^{(S)}\big(w(D^{(iS)}) \otimes w(\x)\big)
&=\left(\bh_i^{(iS)*} w(D^{(iS)})\bh_i^{(iS)}
-n^{-1}\Tr^{(iS)} w(D^{(iS)})\right) \cdot
\langle \xi,w(\x)\zeta\rangle_\cH.\label{eq:tensorfactorize}
\end{align}
Since $\E[|h_{ji}|^2]=n^{-1}$,
the scalar version of Lemma \ref{lemma:concentration}(a,c) (i.e.\ with
$\cX=\C$) implies that
uniformly over deterministic matrices $M \in \C^{n \times n}$,
\begin{equation}\label{eq:scalarconcentration}
\bh_i^{(iS)*} M\bh_i^{(iS)}-n^{-1}\Tr^{(iS)} M
\prec n^{-1}\left(\sum_{r,s}^{(iS)} |M_{rs}|^2\right)^{1/2}
\prec n^{-1/2}\|M\|_\op.
\end{equation}
Note that $w(D^{(iS)})$ is independent of $\bh_i^{(iS)}$, and
$|\{S \subset \{1,\ldots,n\}:|S| \leq l\}| \leq n^l$
and $|\cW_t|=(K+2)^t \leq n^{C\log(K+2)}$ for all $t \leq T(n) \leq C\log n$.
Then, fixing $\epsilon'=0.05$, applying (\ref{eq:scalarconcentration})
to each matrix $M=w(D^{(iS)})$ conditional on $H^{(iS)}$,
and taking a union bound, the event
\[\cE_1=\bigcap_{S \subset \{1,\ldots,n\}:|S| \leq l}\;\bigcap_{i \notin S}
\bigcap_{t=0}^{T(n)} \bigcap_{w \in \cW_t} 
\left\{\left|\bh_i^{(iS)*} w(D^{(iS)})\bh_i^{(iS)}-n^{-1}\Tr^{(iS)} w(D^{(iS)})
\right| \leq n^{-1/2+\epsilon'}\|w(D^{(iS)})\|_\op\right\}\]
holds with probability at least $1-n^{-D}$ for all $n \geq n_0(l,K,D)$.
On $\cE_1$, applying (\ref{eq:tensorfactorize}) to (\ref{eq:tensorexpand})
gives, for each $t=0,\ldots,T(n)$,
\begin{align}
\left|f_{i,\xi,\zeta}^{(S)}\left(\left(\sum_{k=0}^{K+1}
D_k\otimes\x_k\right)^t\right)\right|
&\leq n^{-1/2+\epsilon'} \sum_{w \in \cW_t}\|w(D^{(iS)})\|_\op
\|w(\x)\|_\op\notag\\
&\leq n^{-1/2+\epsilon'} \sum_{w \in \cW_t}
w(\|D_0\|_\op\|\x_0\|_\op,\ldots,\|D_{K+1}\|_\op\|\x_{K+1}\|_\op)\notag\\
&=n^{-1/2+\epsilon'}
\Big(\|D_0\|_\op\|\x_0\|_\op+\ldots+\|D_{K+1}\|_\op\|\x_{K+1}\|_\op\Big)^t\notag\\
&\leq n^{-1/2+\epsilon'} C_0^t.\label{eq:leadingtermsbound}
\end{align}
Then, applying (\ref{eq:leadingtermsbound}) and (\ref{eq:remainderbound})
to (\ref{eq:Mseries}), we have on $\cE_0 \cap \cE_1$ for all $z \in \cD$ that
\[|f_{i,\xi,\zeta}^{(S)} \circ M^{(iS)}(z)|
\leq \sum_{t=0}^{T(n)} |z|^{-(t+1)}n^{-1/2+\epsilon'}C_0^t
+10n^{-1} \leq n^{-1/2+\epsilon},\]
the final inequality holding for our preceding choices of $\epsilon=0.1$,
$\epsilon'=0.05$, and all $n \geq n_0$ since $|z| \geq 2C_0$.
So $\cE_0 \cap \cE_1 \subseteq \cE$, implying that
$\P[\cE] \geq 1-n^{-D}$ for $n \geq n_0(l,K,D)$, as claimed.

This shows that (\ref{eq:opnormest1}) holds with probability $1-n^{-D}$.
The proof for (\ref{eq:opnormest2}) is the same, applying these arguments with
the function
\[f_{i,j,\xi,\zeta}^{(S)}(M)=\left\langle \xi,\left(\sum_r^{(iS)}
h_{ir}M_{rj}\right)\zeta\right\rangle
=\left\langle \xi,\left((\bh_i^{(iS)} \otimes 1)^* M(\e_j \otimes 1)
\right)\zeta\right\rangle\]
in place of $f_{i,\xi,\zeta}^{(S)}$.
\end{proof}

\begin{lem}\label{lemma:weakopnormestimates}
Under Assumption \ref{assump:general}, there exists a constant
$\alpha \in (0,1/2)$ depending only on $\gamma,\upsilon,\delta$
such that for any $l \geq 1$, $D>0$, and all
$n \geq n_0(l,K,\gamma,\upsilon,\delta,D)$, with probability at least $1-n^{-D}$,
\begin{equation}\label{eq:weakopnormestimates}
\sup_{S \subset \{1,\ldots,n\}:|S| \leq l}\;\sup_{i \notin S}
\|R_{ii}^{(S)}-(\r_0)_{ii}\|_\op<n^{-\alpha}, \qquad
\sup_{S \subset \{1,\ldots,n\}:|S| \leq l}\;
\sup_{i,j \notin S:i \neq j} \|R_{ij}^{(S)}\|_\op<n^{-\alpha}.
\end{equation}
\end{lem}
\begin{proof}
Let $\alpha \in (0,1/2)$ be as in
Lemma \ref{lemma:weakopnormprelim}. Fixing $l \geq 1$, let $\cE$ be
the event on which the statements (\ref{eq:opnormest1}--\ref{eq:opnormest2}) of
Lemma \ref{lemma:weakopnormprelim} hold, and in addition,
$\sup_{i=1}^n |h_{ii}|<n^{-\alpha}$. Lemma \ref{lemma:weakopnormprelim}
and Assumption \ref{assump:Wigner} for $H$ imply that
$\cE$ holds with probability at least $1-n^{-D}$ for all $n \geq
n_0(l,K,\gamma,\upsilon,\delta,D)$.

Take any $S \subset \{1,\ldots,n\}$ with $|S| \leq l$, and any $i,j \notin S$
with $i \neq j$. By Lemma \ref{resolvent identities}(b), on $\cE$,
\begin{equation}\label{eq:weakRijbound}
\|R_{ij}^{(S)}\|_\op \leq \|R_{ii}^{(S)}\|_\op\|\x\|_\op
\left\|\sum_r^{(iS)} h_{ir}R_{rj}^{(iS)}\right\|_\op
\leq \gamma\delta^{-1}n^{-\alpha}
\end{equation}
where we have used also $\|R_{ii}^{(S)}\|_\op \leq \|R^{(S)}\|_\op\leq
\delta^{-1}$ by Lemma \ref{im
bound inverse} and the assumption $\|\x\|_\op \leq \gamma$. Adjusting the
value of $\alpha$ yields the second statement of (\ref{eq:weakopnormestimates}).

For the first statement of (\ref{eq:weakopnormestimates}),
take any $S \subset \{1,\ldots,n\}$ with $|S| \leq l$. Define
$\widetilde R_{ii}^{(S)}=R_{ii}^{(S)}$ for $i \notin S$, and
\begin{equation}\label{eq:tildeRdef}
\widetilde R_{ii}^{(S)}=\bigg({-}\z_i-\x\bigg[\frac{1}{n}\sum_j^{(S)}
R_{jj}^{(S)}\bigg]\x\bigg)^{-1} \text{ for } i \in S.
\end{equation}
By this definition and Lemma \ref{resolvent identities}(a), for each
$i=1,\ldots,n$ we have
\begin{equation}\label{eq:Riiapproxfixedpoint}
\widetilde R_{ii}^{(S)}=\bigg({-}\z_i-\Delta_i^{(S)}
-\x\bigg[\frac{1}{n}\sum_{j=1}^n
\widetilde R_{jj}^{(S)}\bigg]\x\bigg)^{-1}
\end{equation}
where 
    \begin{align}
        \Delta_i^{(S)} &= {-}h_{ii}\x + \x\left(\sum_{r,s}^{(iS)}
h_{ir}R_{rs}^{(iS)}h_{si}-\frac{1}{n}\sum_{j=1}^n \widetilde
R_{jj}^{(S)}\right)\x\notag\\
&={-}h_{ii}\x +
\x\Bigg[\underbrace{\sum_r^{(iS)}\left(|h_{ir}|^2-\frac{1}{n}\right)R_{rr}^{(iS)} + 
\sum_{r\neq s}^{(iS)}h_{ir}R_{rs}^{(iS)}h_{si}}_{:=\mathrm{I}}
-\underbrace{\frac{1}{n}\sum_{j \in S \cup \{i\}} \widetilde R_{jj}^{(S)}}_{:=\mathrm{II}}
+\underbrace{\frac{1}{n}\sum_j^{(iS)}(R_{jj}^{(iS)}-R_{jj}^{(S)})}_{:=\mathrm{III}}\Bigg]\x\notag\\
&\text{ for all } i \notin S\label{eq:expression of Delta_i}
    \end{align}
and
\begin{equation}\label{eq:Deltaalt}
\Delta_i^{(S)}={-}\x\bigg[\underbrace{\frac{1}{n}\sum_{j \in S} \widetilde
R_{jj}^{(S)}}_{:=\mathrm{IV}}\bigg]\x \qquad \text{ for all } i \in S.
\end{equation}

On $\cE$, we have $\|\mathrm{I}\|_\op \leq n^{-\alpha}$ by (\ref{eq:opnormest1}).
For $\mathrm{II}$ and $\mathrm{IV}$, Lemma \ref{im bound inverse} implies
$\|\widetilde R_{ii}^{(S)}\|_\op \leq \delta^{-1}$ for both $i \in S$ and $i
\notin S$, the latter because $\|R_{ii}^{(S)}\|_\op \leq \|R^{(S)}\|_\op \leq
\delta^{-1}$ and the former because $\Im R^{(S)} \geq 0$ so
$\Im R_{jj}^{(S)} \geq 0$ and $\Im \z_i \geq
\delta$ in the definition (\ref{eq:tildeRdef}). Then
\begin{equation}\label{eq:DeltaII}
\|\mathrm{II}\|_\op,\|\mathrm{IV}\|_\op \leq \frac{(|S|+1)\delta^{-1}}{n}.
\end{equation}
For $\mathrm{III}$, we have by Lemma \ref{resolvent identities}(b--c) that
\begin{equation}\label{eq:DeltaIII}
\mathrm{III}={-}\frac{1}{n}\sum_j^{(iS)} R_{ji}^{(S)}\frac{1}{R_{ii}^{(S)}}
R_{ij}^{(S)}
=\frac{1}{n}\sum_j^{(iS)} R_{ji}^{(S)}\x
\sum_r^{(iS)} h_{ir}R_{rj}^{(iS)}.
\end{equation}
Thus on $\cE$, $\|\mathrm{III}\|_\op \leq \gamma^2\delta^{-1}n^{-2\alpha}$
by (\ref{eq:opnormest2}) and (\ref{eq:weakRijbound}).
Collecting these bounds and applying also $|h_{ii}| \leq n^{-\alpha}$ on $\cE$,
we have for a constant $C(\gamma,\delta)>0$ and every $i=1,\ldots,n$ that
\begin{equation}\label{eq:Deltaopnormbound}
\|\Delta_i^{(S)}\|_\op \leq C(\gamma,\delta)n^{-\alpha} \text{ on } \cE.
\end{equation}

Taking the Kronecker product with $E_{ii}$ on both sides of
(\ref{eq:Riiapproxfixedpoint}), summing over $i=1,\ldots,n$,
recalling $\z=\sum_i E_{ii} \otimes \z_i$, and setting
$\Delta^{(S)}=\sum_i E_{ii} \otimes \Delta_i^{(S)}$, we have
\begin{equation}\label{eq:RiiSfixedpoint}
\sum_{i=1}^n E_{ii}\otimes \widetilde R_{ii}^{(S)}
=\left(-\mathsf{z}-\Delta^{(S)}-I_{n \times n}
\otimes\x\left(\frac{1}{n}\Tr\otimes 1\left[\sum_{i=1}^n E_{ii}\otimes
\widetilde R_{ii}^{(S)}\right]\right)\x \right)^{-1}.
\end{equation}
On $\cE$, by the estimate $\|\Delta^{(S)}\|_\op=\max_i\|\Delta_i^{(S)}\|_\op \leq
C(\gamma,\delta)n^{-\alpha}$ from (\ref{eq:Deltaopnormbound}),
for all $n \geq n_0(\gamma,\upsilon,\delta)$ this
implies $\Im(\z+\Delta^{(S)}) \geq \delta/2$, so Corollary
\ref{cor:Stability of DX-valued fixed point}(a) shows
\begin{equation}\label{eq:weakRiibound}
\max_{i=1}^n \|\widetilde R_{ii}^{(S)}-(\r_0)_{ii}\|_\op
=\pnorm{\sum_{i=1}^n E_{ii} \otimes \widetilde R_{ii}^{(S)} - \r_0}{\op}\leq
2\delta^{-2}\|\Delta^{(S)}\|_\op \leq 2\delta^{-1}C(\gamma,\delta)n^{-\alpha}.
\end{equation}
Then specializing this to $i \notin S$ and 
adjusting the value of $\alpha$
yields the first statement of (\ref{eq:weakopnormestimates}).
\end{proof}

\subsection{Iterative bootstrapping}\label{subsec:bootstrapping}

We now improve the preceding estimates of Lemma \ref{lemma:weakopnormestimates}
in the $L^p$-norms for $p<\infty$, using the following bootstrapping lemma.

\begin{lem}\label{lemma:bootstrapping}
Suppose Assumption \ref{assump:general} holds. Suppose also that, for some
$\alpha \in (0,1/2)$ and any fixed $l \geq 1$,
uniformly over $S \subset \{1,\ldots,n\}$ with
$|S| \leq l$ and over $i,j \notin S$ with $i \neq j$, we have
\begin{equation}\label{eq:bootstrappingassumption}
R_{ii}^{(S)}-(\r_0)_{ii} \prec n^{-\alpha}, \qquad
R_{ij}^{(S)} \prec n^{-\alpha}.
\end{equation}
Set $\alpha'=\min(\frac{3\alpha}{2},\frac{1}{2})$.
Then for any fixed $l \geq 1$, uniformly over $S \subset \{1,\ldots,n\}$ with
$|S| \leq l$ and over $i,j \notin S$ with $i \neq j$,
\[R_{ii}^{(S)}-(\r_0)_{ii} \prec n^{-\alpha'}, \qquad
R_{ij}^{(S)} \prec n^{-\alpha'}.\]
\end{lem}

Before proving this lemma, we derive an estimate (Lemma \ref{lemma:bootstrapFA}
below) that we will use in conjunction with the fluctuation averaging result of
Lemma \ref{lemma:fluctuationavg}. Define
\begin{equation}\label{eq:partialexpectations}
\E_i[\cdot]=\E[\;\cdot \mid H^{(i)}], \qquad  \cQ_i=1-\E_i,
\end{equation}
where $\E_i$ is the expectation over the entries in only row and
column $i$ of $H$. Note that
$\E_i\E_j=\E_j\E_i$ for all $i \neq j$, so $\{\E_i,\cQ_i:i=1,\ldots,n\}$ form a 
commuting system of projections. Set $\cQ_T=\prod_{i \in T} \cQ_i$.

\begin{lem}\label{lemma:invRiibound}
Under Assumption \ref{assump:general}, for any fixed $l \geq 1$, uniformly over
$S \subset \{1,\ldots,n\}$ with $|S| \leq l$ and $i \notin S$,
\begin{equation}\label{eq:invRiibound}
(R_{ii}^{(S)})^{-1} \prec 1.
\end{equation}
\end{lem}
\begin{proof}
Let $\bh_i^{(iS)}$ be the $i^\text{th}$ column of $H^{(iS)}$, having
entries $(\bh_i^{(iS)})_j=h_{ji}$ if $j \notin S \cup \{i\}$ and 0
otherwise. Lemma \ref{resolvent identities}(a) gives
\begin{align*}
\|(R_{ii}^{(S)})^{-1}\|_\op &\leq
    |h_{ii}|\pnorm{\x}{\op} + \pnorm{\z_i}{\op} +
    \|(\bh_i^{(iS)}\otimes\x)^*R^{(iS)}(\bh_i^{(iS)}\otimes\x)\|_\op\\
             &\leq |h_{ii}|\pnorm{\x}{\op} + \pnorm{\z_i}{\op} +
    \|R^{(iS)}\|_\op\pnorm{\x}{\op}^2\|\bh_i^{(iS)}\|_2^2.
        \end{align*}
We have $\|\x\|_\op,\|\z_i\|_\op \prec 1$ by (\ref{eq:xzbounds}),
$h_{ii} \prec n^{-1/2}$ by (\ref{eq:momentassump}), and
$\|\bh_i^{(iS)}\|_2 \prec 1$
by Lemma \ref{lemma:concentration}(a) in the scalar case of $\cX=\C$,
so the result follows.
\end{proof}

\begin{lem}\label{lemma:bootstrapFA}
For any $S \subset \{1,\ldots,n\}$ and $r,j \notin S$ with $r \neq j$, define
\[\cZ_{rj}^{(S)}=\cQ_r[R_{rj}^{(S)}R_{rj}^{(S)*}].\]
Fix any $l \geq 1$. Under the assumptions of Lemma \ref{lemma:bootstrapping},
uniformly over subsets $S \subset \{1,\ldots,n\}$ with
$|S| \leq l$, $r,j \notin S$ with $r \neq j$, and
$T \subset \{1,\ldots,n\} \setminus (S \cup \{r,j\})$ with $|T| \leq l$,
\[\cQ_T[\cZ_{rj}^{(S)}] \prec n^{-(2+|T|)\alpha}.\]
\end{lem}
\begin{proof}
For $|T|=0$, we have $R_{rj}^{(S)}R_{rj}^{(S)*} \prec n^{-2\alpha}$ by the
assumption (\ref{eq:bootstrappingassumption}) and H\"older's inequality. Then also
$\cZ_{rj}^{(S)}=R_{rj}^{(S)}R_{rj}^{(S)*}-\E_r[R_{rj}^{(S)}R_{rj}^{(S)*}]
\prec n^{-2\alpha}$ by Lemma \ref{lemma:domination}(c), so the assertion holds
for $|T|=0$.

If $|T| \geq 1$, suppose $i_1 \in T$. Then Lemma \ref{resolvent identities}(c)
gives
\begin{align*}
R_{rj}^{(S)}R_{rj}^{(S)*}
&=\underbrace{R_{rj}^{(i_1S)}R_{rj}^{(i_1S)*}}_{=L(\{i_1\})}\\
&\quad
+\underbrace{\Big(R_{ri_1}^{(S)}\frac{1}{R_{i_1i_1}^{(S)}}R_{i_1j}^{(S)}\Big)
R_{rj}^{(i_1S)*}
+R_{rj}^{(i_1S)}\Big(R_{ri_1}^{(S)}\frac{1}{R_{i_1i_1}^{(S)}}R_{i_1j}^{(S)}\Big)^*
+\Big(R_{ri_1}^{(S)}\frac{1}{R_{i_1i_1}^{(S)}}R_{i_1j}^{(S)}\Big)
\Big(R_{ri_1}^{(S)}\frac{1}{R_{i_1i_1}^{(S)}}R_{i_1j}^{(S)}\Big)^*}_{=P(\{i_1\})}.
\end{align*}
Here, the first term $L(\{i_1\})$ is independent of the entries of row and
column $i_1$ of $H$, so $\cQ_{i_1}[L(\{i_1\})]=0$. The remaining
terms constituting $P(\{i_1\})$ each have at least 3 off-diagonal resolvent
factors, i.e.\ factors of the form
$R_{pq}^{(S')}$ or $R_{pq}^{(S')*}$ for some $p \neq q$ and $p,q \notin S'$,
so (\ref{eq:invRiibound}) and
(\ref{eq:bootstrappingassumption}) imply $P(\{i_1\}) \prec n^{-3\alpha}$.

Now if $|T| \geq 2$ and $i_2 \in T$ with $i_2 \neq i_1$, we 
apply Lemma \ref{resolvent identities}(c) again
to expand each factor of $P(\{i_1\})$ over $i_2$, yielding
    \begin{align*}
R_{rj}^{(S)}R_{rj}^{(S)*}=L(\{i_1,i_2\})+P(\{i_1,i_2\})
    \end{align*}
    where
    \begin{align*}
    L(\{i_1,i_2\}) &=
L(\{i_1\})+\Big(R_{ri_1}^{(i_2S)}\frac{1}{R_{i_1i_1}^{(i_2S)}}R_{i_1j}^{(i_2S)}\Big)
R_{rj}^{(i_1i_2S)*}
+R_{rj}^{(i_1i_2S)}\Big(R_{ri_1}^{(i_2S)}\frac{1}{R_{i_1i_1}^{(i_2S)}}R_{i_1j}^{(i_2S)}\Big)^*\\
&\hspace{1in}
+\Big(R_{ri_1}^{(i_2S)}\frac{1}{R_{i_1i_1}^{(i_2S)}}R_{i_1j}^{(i_2S)}\Big)
\Big(R_{ri_1}^{(i_2S)}\frac{1}{R_{i_1i_1}^{(i_2S)}}R_{i_1j}^{(i_2S)}\Big)^*
\end{align*}
and $P(\{i_1,i_2\})$ collects all remaining terms in the expansion of
$P(\{i_1\})$. Each term of $L(\{i_1,i_2\})$ is independent of the
entries of either row and column $i_1$ or $i_2$ of $H$, so
$Q_{\{i_1,i_2\}}[L(\{i_1,i_2\})]=0$.
Each term of $P(\{i_1,i_2\})$ has at least 4 off-diagonal resolvent factors,
so $P(\{i_1,i_2\}) \prec n^{-4\alpha}$. Inductively applying
Lemma (\ref{resolvent identities})(c) to expand $P(\{i_1,\ldots,i_k\})$ in each
successive index $i_{k+1}$ of $T$, this shows
    \begin{align*}
R_{rj}^{(S)}R_{rj}^{(S)*}=L(T)+P(T)
    \end{align*}
    where
    \begin{itemize}
        \item Each term in $L(T)$ is independent of the entries of
row and column $i$ of $H$
for at least one index $i \in T$, so $\cQ_T[L(T)]=0$.
        \item $P(T)$ is a sum of at most $C_l$ summands, each summand a product
of at most $C_l$ factors, for a constant $C_l>0$ depending only on the given
upper bound $l$ for $|T|$.
        \item Each term of $P(T)$ has at least $2+|T|$ off-diagonal resolvent
factors, and hence by (\ref{eq:bootstrappingassumption})
is of size $\Oprec(n^{-(2+|T|)\alpha})$.
    \end{itemize}
Then $P(T) \prec n^{-(2+|T|)\alpha}$, uniformly over $T \subset \{1,\ldots,n\}
\setminus (S \cup \{r,j\})$ with $|T| \leq l$.
This implies $\cQ_{T \cup \{r\}}[L(T)]=0$ and
$\cQ_{T \cup \{r\}}[P(T)] \prec n^{-(2+|T|)\alpha}$ by
Lemma \ref{lemma:domination}(c), so
$\cQ_T[\cZ_{rj}^{(T)}]=\cQ_{T \cup \{r\}}[R_{rj}^{(T)}R_{rj}^{(T)*}]
\prec n^{-(2+|T|)\alpha}$ as desired.
\end{proof}

\begin{proof}[Proof of Lemma \ref{lemma:bootstrapping}]
In view of Lemma \ref{resolvent identities}(b) for the form of
$R_{ij}^{(S)}$, we consider first the quantity
\[\sum_r^{(iS)}h_{ir}R_{rj}^{(iS)}\]
for $S \subset \{1,\ldots,n\}$ with $|S| \leq l$ and $i,j \notin S$
with $i \neq j$.
Recalling the assumption $\E[(\sqrt{n}|h_{ij}|)^p] \leq C_p$ and
applying Lemma \ref{lemma:concentration}(a) conditional on $H^{(iS)}$,
for any $p \in [2,\infty)$ and a constant $C_p>0$,
\begin{equation}\label{eq:rosenthalapplication1}
\E\bigg[\bigg\|\sum_r^{(iS)}h_{ir}R_{rj}^{(iS)}\bigg\|_p^p\bigg]
\leq C_p n^{-p/2}\max\bigg\{\E\bigg[\bigg\|\bigg(\sum_r^{(iS)}
R_{rj}^{(iS)}R_{rj}^{(iS)*}\bigg)^{1/2}\bigg\|_p^p\bigg],
\E\bigg[\bigg\|\bigg(\sum_r^{(iS)}
R_{rj}^{(iS)*}R_{rj}^{(iS)}\bigg)^{1/2}\bigg\|_p^p\bigg]\bigg\}.
\end{equation}
The second bound of (\ref{eq:rosenthalapplication1})
may be controlled spectrally: Using that
$R^{(iS)}_{rj}=0$ for $r \in S \cup \{i\}$ and $j \notin S \cup \{i\}$,
\[\bigg\|\sum_r^{(iS)} R_{rj}^{(iS)*}R_{rj}^{(iS)}\bigg\|_\op
=\Big\|(\e_j \otimes 1) R^{(iS)*} R^{(iS)} (\e_j \otimes 1)\Big\|_\op
\leq \|R^{(iS)}\|_\op^2 \leq \delta^{-2}.\]
Thus, applying monotonicity of the $L^p$-norm in $p$
(Lemma \ref{Holder's inequality}),
\begin{equation}\label{eq:rosenthalseconderrorbound}
\E\bigg[\bigg\|\bigg(\sum_r^{(iS)}
R_{rj}^{(iS)*}R_{rj}^{(iS)}\bigg)^{1/2}\bigg\|_p^p\bigg]
\leq \E\bigg[\bigg\|\sum_r^{(iS)} R_{rj}^{(iS)*}R_{rj}^{(iS)}\bigg\|_\op^{p/2}
\bigg] \leq \delta^{-p}.
\end{equation}

For the first bound of (\ref{eq:rosenthalapplication1}), let us write
\[\sum_r^{(iS)} R_{rj}^{(iS)}R_{rj}^{(iS)*}
=R_{jj}^{(iS)}R_{jj}^{(iS)*}+\sum_r^{(ijS)} \E_r[R_{rj}^{(iS)}R_{rj}^{(iS)*}]
+\sum_r^{(ijS)} \cQ_r[R_{rj}^{(iS)}R_{rj}^{(iS)*}].\]
Applying $\|R_{jj}^{(iS)}\|_\op \leq \delta^{-1}$ for the first term, and
Lemma \ref{lemma:fluctuationavg}(a) with the estimates of Lemma
\ref{lemma:bootstrapFA} for the third term, this gives
\begin{equation}\label{eq:nonspectralexpand2}
\sum_r^{(iS)} R_{rj}^{(iS)}R_{rj}^{(iS)*}
=\sum_r^{(ijS)} \E_r[R_{rj}^{(iS)}R_{rj}^{(iS)*}]+\Oprec(1+n^{1-3\alpha}).
\end{equation}
We expand the resolvent for the remaining term of (\ref{eq:nonspectralexpand2}),
applying Lemma \ref{lemma:concentration}(b) to write, for any $r \notin S \cup
\{i,j\}$,
\[R_{rj}^{(iS)}={-}R_{rr}^{(iS)}\x\sum_s^{(irS)} h_{rs}R_{sj}^{(irS)}.\]
Here $R_{rj}^{(iS)} \prec n^{-\alpha}$ by the second statement of
(\ref{eq:bootstrappingassumption}),
so multiplying by $(R_{rr}^{(iS)})^{-1}$ and $\x^{-1}$ and applying
(\ref{eq:invRiibound}) and $\|\x^{-1}\|_\op \leq \gamma$, also
$\sum_s^{(irS)} h_{rs}R_{sj}^{(irS)} \prec n^{-\alpha}$.
Then the first statement of (\ref{eq:bootstrappingassumption}) gives
\[R_{rj}^{(iS)}={-}(\r_0)_{rr}\x\sum_s^{(irS)} h_{rs}R_{sj}^{(irS)}
+\Oprec(n^{-2\alpha}).\]
Applying this and independence of $R_{sj}^{(irS)}$ with the variables in
row/column $r$ of $H$,
\begin{align*}
\E_r[R_{rj}^{(iS)}R_{rj}^{(iS)*}]
&=(\r_0)_{rr}\x\bigg(\sum_{s,t}^{(irS)}
\E[h_{rs}h_{tr}]R_{sj}^{(irS)}R_{jt}^{(irS)*}\bigg)\x(\r_0)_{rr}^*
+\Oprec(n^{-3\alpha})\\
&=\frac{1}{n}(\r_0)_{rr}\x\bigg(\sum_s^{(irS)}
R_{sj}^{(irS)}R_{js}^{(irS)*}\bigg)\x(\r_0)_{rr}^*
+\Oprec(n^{-3\alpha}).
\end{align*}
Recalling $r \notin S \cup \{i,j\}$, we have
\begin{align*}
\sum_s^{(irS)} R_{sj}^{(irS)}R_{js}^{(irS)*}
&=R_{jj}^{(irS)}R_{jj}^{(irS)*}
+\sum_s^{(ijrS)} R_{sj}^{(irS)}R_{js}^{(irS)*}\\
&=R_{jj}^{(irS)}R_{jj}^{(irS)*}
+\sum_s^{(ijrS)} R_{sj}^{(iS)}R_{js}^{(iS)*}+\Oprec(n^{1-3\alpha})\\
&=\sum_s^{(iS)} R_{sj}^{(iS)}R_{js}^{(iS)*}+\Oprec(1+n^{1-3\alpha}),
\end{align*}
the second line using
$R_{sj}^{(irS)}=R_{sj}^{(iS)}+\Oprec(n^{-2\alpha})$ 
for $r \notin S \cup \{i,j,s\}$ by Lemma
\ref{resolvent identities}(c), (\ref{eq:bootstrappingassumption}), and
(\ref{eq:invRiibound}), and the third line using $\|R_{jj}^{(irS)}\|_\op \leq
\delta^{-1}$ and $\|R_{sj}^{(iS)}\|_\op \leq \delta^{-1}$ for $s \in \{j,r\}$.
So
\[\E_r[R_{rj}^{(iS)}R_{rj}^{(iS)*}]
=\frac{1}{n}(\r_0)_{rr}\x\bigg(\sum_s^{(iS)}
R_{sj}^{(iS)}R_{js}^{(iS)*}\bigg)\x(\r_0)_{rr}^*
+\Oprec(n^{-1}+n^{-3\alpha}).\]
Summing over $r \notin S \cup \{i,j\}$, this gives
\begin{align*}
\sum_r^{(ijS)} \E_r[R_{rj}^{(iS)}R_{rj}^{(iS)*}]
&=\sum_r^{(ijS)} \frac{1}{n}(\r_0)_{rr}\x
\bigg(\sum_s^{(iS)} R_{sj}^{(iS)}R_{js}^{(iS)*}\bigg)\x(\r_0)_{rr}^*
+\Oprec(1+n^{1-3\alpha})\\
&=\sum_{r=1}^n \frac{1}{n}(\r_0)_{rr}\x
\bigg(\sum_s^{(iS)} R_{sj}^{(iS)}R_{js}^{(iS)*}\bigg)\x(\r_0)_{rr}^*
+\Oprec(1+n^{1-3\alpha}),
\end{align*}
where the second step applies
the trivial bound $\|\sum_s^{(iS)} R_{sj}^{(iS)}R_{js}^{(iS)*}\|_\op
\leq n\|R^{(iS)}\|_\op^2 \leq n\delta^{-2}$ to include the summands for
$r \in S \cup \{i,j\}$ with an additional $\Oprec(1)$ error. Applying this back to (\ref{eq:nonspectralexpand2}), we
obtain
\[\cL_2\bigg(\sum_r^{(iS)} R_{rj}^{(iS)}R_{rj}^{(iS)*}\bigg) \prec
1+n^{1-3\alpha}\]
where here $\cL_2$ is the linear operator of Lemma \ref{lemma:linearmap2}.
By the quantitative invertibility of $\cL_2$ established in
Lemma \ref{lemma:linearmap2}, this implies
\[\sum_r^{(iS)} R_{rj}^{(iS)}R_{rj}^{(iS)*}
\prec 1+n^{1-3\alpha}.\]
Then, for any fixed $p \in [2,\infty)$, we get (from this and Lemma
\ref{lemma:domination})
\begin{equation}\label{eq:rosenthalfirsterrorbound}
\E\bigg[\bigg\|\bigg(\sum_r^{(iS)}
R_{rj}^{(iS)}R_{rj}^{(iS)*}\bigg)^{1/2}\bigg\|_p^p\bigg]
\prec 1+n^{(1-3\alpha)(p/2)},
\end{equation}
which controls the first bound of (\ref{eq:rosenthalapplication1}).

Then, applying (\ref{eq:rosenthalfirsterrorbound}) and
(\ref{eq:rosenthalseconderrorbound}) to
(\ref{eq:rosenthalapplication1}), for any fixed $p \in [2,\infty)$ and
all $n \geq n_0(p)$,
\[\E\bigg[\bigg\|\sum_r^{(iS)} h_{ir}R_{rj}^{(iS)}\bigg\|_p^p\bigg]
\leq n^{-p\alpha'+0.1}\]
where we have set $\alpha'=\min(\frac{3\alpha}{2},\frac{1}{2})$.
Then, for any fixed $q \in [1,\infty)$ and $\epsilon,D>0$, choosing $p \geq q$
large enough so that $p\epsilon>D+0.1$ and applying Markov's inequality and
monotonicity of $\|\cdot\|_p$ in $p$ (Lemma \ref{Holder's inequality}),
for all $n \geq n_0(q,\epsilon,D)$,
\begin{equation}\label{eq:Markov1}
\P\bigg[\bigg\|\sum_r^{(iS)}
h_{ir}R_{rj}^{(iS)}\bigg\|_q \geq n^{-\alpha'+\epsilon}\bigg]
\leq n^{p\alpha'-p\epsilon}
\E\bigg[\bigg\|\sum_r^{(iS)}
h_{ir}R_{rj}^{(iS)}\bigg\|_q^p\bigg]
\leq n^{p\alpha'-p\epsilon}
\E\bigg[\bigg\|\sum_r^{(iS)}
h_{ir}R_{rj}^{(iS)}\bigg\|_p^p\bigg]<n^{-D}.
\end{equation}
This shows $\sum_r^{(iS)} h_{ir}R_{rj}^{(iS)} \prec n^{-\alpha'}$.
Then by Lemma \ref{resolvent identities}(b) and the bound
$\|R_{ii}^{(S)}\|_\op \leq \delta^{-1}$, we have the improved estimate
(uniformly over $S \subset \{1,\ldots,n\}$ with $|S| \leq l$ and
$i,j \notin S$ with $i \neq j$)
\begin{equation}\label{eq:offdiagimproved}
R_{ij}^{(S)} \prec n^{-\alpha'}.
\end{equation}

Now to show $R_{ii}^{(S)}-(\r_0)_{ii} \prec n^{-\alpha'}$, recall from
(\ref{eq:Riiapproxfixedpoint}) that
\[\widetilde
R_{ii}^{(S)}=\bigg({-}\z_i-\Delta_i^{(S)}-\x\bigg[\frac{1}{n}\sum_{j=1}^n
\widetilde R_{jj}^{(S)}\bigg]\x\bigg)^{-1}\]
where $\widetilde R_{ii}^{(S)}=R_{ii}^{(S)}$ for all $i \notin S$, and
$\Delta_i^{(S)}$ is the error defined by (\ref{eq:expression of Delta_i})
and (\ref{eq:Deltaalt}). For
the term $\mathrm{I}$ of (\ref{eq:expression of Delta_i}), applying
Lemma \ref{lemma:concentration}(a) conditional on $H^{(iS)}$ gives,
for any $p \geq 2$ and a constant $C_p>0$,
\begin{align*}
&\E\bigg[\bigg\|\sum_r^{(iS)}\left(|h_{ir}|^2-\frac{1}{n}\right)R_{rr}^{(iS)}
\bigg\|_p^p\bigg]\\
&\leq C_pn^{-p}\max\bigg\{
\E\bigg[\bigg\|\bigg(\sum_r^{(iS)}R_{rr}^{(iS)}R_{rr}^{(iS)*}\bigg)^{1/2}\bigg\|_p^p\bigg],\;
\E\bigg[\bigg\|\bigg(\sum_r^{(iS)}R_{rr}^{(iS)*}R_{rr}^{(iS)}\bigg)^{1/2}\bigg\|_p^p\bigg]\bigg\}\\
&\leq C_pn^{-p}\max\bigg\{
\E\bigg[\bigg\|\bigg(\sum_{r,s=1}^n R_{rs}^{(iS)}R_{rs}^{(iS)*}\bigg)^{1/2}\bigg\|_p^p\bigg],\;
\E\bigg[\bigg\|\bigg(\sum_{r,s=1}^n R_{rs}^{(iS)*}R_{rs}^{(iS)}\bigg)^{1/2}\bigg\|_p^p\bigg]\bigg\}
\end{align*}
where the second inequality applies monotonicity of the operator square-root $0
\leq \x \leq \y \Rightarrow \x^{1/2} \leq \y^{1/2}$ and monotonicity of the
$L^p$-norm over the positive cone (Lemma \ref{Holder's inequality}). Similarly
applying Lemma \ref{lemma:concentration}(c) conditional on $H^{(iS)}$ for the
summation over $r \neq s$, the term $\mathrm{I}$ of (\ref{eq:expression
of Delta_i}) is bounded as
\begin{align}
&\E\bigg[\bigg\|\sum_r^{(iS)}\left(|h_{ir}|^2-\frac{1}{n}\right)R_{rr}^{(iS)}
+\sum_{r\neq s}^{(iS)}h_{ir}R_{rs}^{(iS)}h_{si}\bigg\|_p^p\bigg]\notag\\
&\leq C_pn^{-p}\max\bigg\{
\E\bigg[\bigg\|\bigg(\sum_{r,s=1}^n R_{rs}^{(iS)}R_{rs}^{(iS)*}\bigg)^{1/2}\bigg\|_p^p\bigg],\;
\E\bigg[\bigg\|\bigg(\sum_{r,s=1}^n R_{rs}^{(iS)*}R_{rs}^{(iS)}\bigg)^{1/2}\bigg\|_p^p\bigg],\notag\\
&\hspace{2in}n\,\E[\|R^{(iS)}-\diag(R^{(iS)})\|_p^p],\;n\,\E[\|{R^{(iS)}}^\st-\diag(R^{(iS)})\|_p^p]\bigg\}
\label{eq:rosenthal2}
\end{align}
where ${R^{(iS)}}^\st=\sum_{k,l=1}^n E_{kl} \otimes R^{(iS)}_{lk}$
is the partial transpose of $R^{(iS)}$ as defined in
(\ref{eq:partialtranspose}), and here
$\diag(R^{(iS)})=\sum_{k=1}^n E_{kk} \otimes
R_{kk}^{(iS)}$ is the operator that has only the diagonal entries
of $R^{(iS)}$. The first three bounds on the right side of
(\ref{eq:rosenthal2}) may be controlled spectrally: For the third bound,
\begin{equation}\label{eq:rosenthal2error3}
n\,\E[\|R^{(iS)}-\diag(R^{(iS)})\|_p^p]
\leq n\,\E[(\|R^{(iS)}\|_\op+\|\diag(R^{(iS)})\|_\op)^p] 
\leq n(2\delta^{-1})^p.
\end{equation}
For the first and second bounds,
\begin{align*}
\bigg\|\sum_{r,s=1}^n R_{rs}^{(iS)}R_{rs}^{(iS)*}\bigg\|_\op
&=n\,\Big\|(n^{-1}\Tr \otimes 1)R^{(iS)}R^{(iS)*}\Big\|_\op
\leq n\|R^{(iS)}\|_\op^2 \leq n\delta^{-2},\\
\bigg\|\sum_{r,s=1}^n R_{rs}^{(iS)*}R_{rs}^{(iS)}\bigg\|_\op
&=n\,\Big\|(n^{-1}\Tr \otimes 1)R^{(iS)*}R^{(iS)}\Big\|_\op
\leq n\|R^{(iS)}\|_\op^2 \leq n\delta^{-2}.
\end{align*}
Thus
\begin{equation}\label{eq:rosenthal2error1}
\E\bigg[\bigg\|\bigg(\sum_{r,s=1}^n
R_{rs}^{(iS)}R_{rs}^{(iS)*}\bigg)^{1/2}\bigg\|_p^p\bigg],\;
\E\bigg[\bigg\|\bigg(\sum_{r,s=1}^n R_{rs}^{(iS)*}R_{rs}^{(iS)}\bigg)^{1/2}\bigg\|_p^p\bigg]
\leq n^{p/2}\delta^{-p}.
\end{equation}

For the fourth bound $n\,\E[\|{R^{(iS)}}^\st-\diag(R^{(iS)})\|_p^p]$ of (\ref{eq:rosenthal2}),
we apply the following argument:
For any $M=\sum_{j,k}E_{jk} \otimes M_{jk} \in \C^{n \times n} \otimes \cX$,
note that
\[M^\st-\diag(M)=\sum_{1 \leq j \neq k \leq n} E_{jk} \otimes M_{kj}
=\sum_{j=1}^{n-1} M[j], \qquad \text{ for } M[j]=\sum_{k=1}^n E_{j+k,k}
\otimes M_{k,j+k}\]
where $j+k$ is interpreted modulo $n$. Here
\begin{align*}
\|M[j]\|_p^p&=\|(M[j]M[j]^*)^{1/2}\|_p^p
=\bigg\|\sum_{k=1}^n E_{j+k,j+k} \otimes (M_{k,j+k}M_{k,j+k}^*)^{1/2}
\bigg\|_p^p\\
&=(n^{-1}\Tr \otimes \tau)
\sum_{k=1}^n E_{j+k,j+k} \otimes (M_{k,j+k}M_{k,j+k}^*)^{p/2}
=\frac{1}{n}\sum_{k=1}^n \|M_{k,j+k}\|_p^p
\leq \max_{k=1}^n \|M_{k,j+k}\|_p^p,
\end{align*}
so
\[\|M^\st-\diag(M)\|_p \leq \sum_{j=1}^{n-1} \|M[j]\|_p
\leq \sum_{j=1}^{n-1}\Big(\max_{k=1}^n \|M_{k,j+k}\|_p\Big).\]
Applying this to $M=R^{(iS)}$, observe that for each $j=1,\ldots,n-1$ we have
$R^{(iS)}_{k,j+k} \prec n^{-\alpha'}$ if $k,j+k \notin S \cup \{i\}$ as
already shown in (\ref{eq:offdiagimproved}),
and $R^{(iS)}_{k,j+k}=0$ by definition if $k \in S \cup \{i\}$ or $j+k
\in S \cup \{i\}$ with $k \neq j+k$. Thus
$\max_k \|R_{k,j+k}^{(iS)}\|_p \prec n^{-\alpha'}$, so
$\|{R^{(iS)}}^\st\|_p \prec n^{1-\alpha'}$, and
\begin{equation}\label{eq:rosenthal2error4}
n\,\E[\|{R^{(iS)}}^\st\|_p^p]
\prec n \cdot n^{p(1-\alpha')}.
\end{equation}
Applying (\ref{eq:rosenthal2error3}), (\ref{eq:rosenthal2error1}),
and (\ref{eq:rosenthal2error4}) to
(\ref{eq:rosenthal2}), for any fixed $p \in [2,\infty)$ 
and all $n \geq n_0(p)$,
\[\E\bigg[\bigg\|\sum_r^{(iS)}\left(|h_{ir}|^2-\frac{1}{n}\right)R_{rr}^{(iS)}
+\sum_{r\neq s}^{(iS)}h_{ir}R_{rs}^{(iS)}h_{si}\bigg\|_p^p\bigg]
\leq n^{\max(1-p,-p/2,1-p\alpha')+0.1}
=n^{-p\alpha'+1.1},\]
the last equality using $\alpha' \leq 1/2$.
Then, for any fixed $q \in [1,\infty)$ and
$\epsilon,D>0$, applying this with $p \geq q$ large enough so that
$p\epsilon>D+1.1$, we obtain similarly to (\ref{eq:Markov1})
\[\P\bigg[\bigg\|\sum_r^{(iS)}\left(|h_{ir}|^2-\frac{1}{n}\right)R_{rr}^{(iS)}
+\sum_{r\neq s}^{(iS)}h_{ir}R_{rs}^{(iS)}h_{si}\bigg\|_q \geq
n^{-\alpha'+\epsilon}\bigg]<n^{-D}.\]
Hence, for the term $\mathrm{I}$ defining $\Delta_i^{(S)}$ in
(\ref{eq:expression of Delta_i}), we have $\mathrm{I} \prec n^{-\alpha'}$.

For the other terms of (\ref{eq:expression of Delta_i}), we have
$\mathrm{II},\mathrm{IV} \prec n^{-1}$ by (\ref{eq:DeltaII}),
and $\mathrm{III} \prec n^{-2\alpha'}$ by its form (\ref{eq:DeltaIII})
and (\ref{eq:offdiagimproved}) and (\ref{eq:invRiibound}).
Applying also $h_{ii} \prec n^{-1/2}$, this gives (uniformly over
$i=1,\ldots,n$)
\begin{equation}\label{eq:Deltapnormbound}
\Delta_i^{(S)} \prec n^{-\alpha'}.
\end{equation}
Then, setting $\Delta^{(S)}=\sum_i E_{ii} \otimes \Delta_i^{(S)}$, we have
$\|\Delta^{(S)}\|_p^p=n^{-1}\sum_i \|\Delta_i^{(S)}\|_p^p$ for any $p \in [1,\infty)$ so
also $\|\Delta^{(S)}\|_p \prec n^{-\alpha'}$ for each fixed $p$.
Recalling the weak estimate for
$\|\Delta^{(S)}\|_\op$ from (\ref{eq:Deltaopnormbound}) on the high probability event
$\cE$, which ensures $\Im(\z+\Delta^{(S)}) \geq \delta/2$,
this implies by (\ref{eq:RiiSfixedpoint}) and
Corollary \ref{cor:Stability of DX-valued fixed point}(a) applied now in the
norm $\|\cdot\|_p$ that
\[\bigg\|\sum_{i=1}^n E_{ii} \otimes \widetilde R_{ii}^{(S)}-\r_0\bigg\|_p
\prec n^{-\alpha'}.\]
Thus, for any $i \notin S$, any $q \in [1,\infty)$, and any $\epsilon,D>0$,
choosing $p \geq q$ large enough such that $p\epsilon>1.1$,
\[\|R_{ii}^{(S)}-(\r_0)_{ii}\|_q^p
\leq \sum_{i=1}^n \|\widetilde R_{ii}^{(S)}-(\r_0)_{ii}\|_p^p
=n\bigg\|\sum_{i=1}^n E_{ii} \otimes \widetilde R_{ii}^{(S)}-\r_0\bigg\|_p^p
\leq n^{-\alpha'p+1.1}<n^{(-\alpha'+\epsilon)p}\]
with probability at least $1-n^{-D}$ for $n \geq n_0(q,\epsilon,D)$.
This shows $R_{ii}^{(S)}-(\r_0)_{ii} \prec n^{-\alpha'}$.
\end{proof}

\subsection{Proof of Theorem \ref{aux:main theorem}}\label{subsec:proofgeneral}

We now prove the main result of this section,
Theorem \ref{aux:main theorem}.

\begin{proof}[Proof of Theorem \ref{aux:main theorem}, (\ref{DX-valued Stieltjes
transform estimates}--\ref{off diagonal estimates})]

Lemma \ref{lemma:weakopnormestimates} implies
$R_{ii}^{(S)}-(\r_0)_{ii} \prec n^{-\alpha}$ and $R_{ij}^{(S)}
\prec n^{-\alpha}$ for some $\alpha>0$, uniformly over sets $S$ with $|S| \leq
l$ and $i,j \notin S$ with $i \neq j$. Then, iterating
Lemma \ref{lemma:bootstrapping} a constant number of times, we get
$R_{ii}^{(S)}-(\r_0)_{ii} \prec n^{-1/2}$ and $R_{ij}^{(S)} \prec n^{-1/2}$,
implying for $S=\emptyset$ the statements
(\ref{DX-valued Stieltjes transform estimates}--\ref{off diagonal estimates}).
\end{proof}

For the remaining statements of Theorem \ref{aux:main theorem}, we collect here
several estimates needed for additional applications of fluctuation averaging
(Lemma \ref{lemma:fluctuationavg}).

\begin{lem}\label{lemma:mainproofFA}
Under Assumption \ref{assump:general}, define
\begin{align*}
\cZ_i&=(\r_0)_{ii}\cQ_i[R_{ii}^{-1}](\r_0)_{ii},\\
\cZ_{ij}&=\cQ_i\cQ_j[R_{ij}] \text{ for } i \neq j,\\
\cK_r^{(ij)}&=(\r_0)_{ii}\x(\r_0)_{rr}\x(\r_0)_{ii}\x\cQ_r[R_{rj}^{(i)}]
\text{ for distinct } r,i,j.
\end{align*}
Fix any $l \geq 1$.
\begin{enumerate}[(a)]
\item Uniformly over $S \subset \{1,\ldots,n\}$ with $|S| \leq l$
and $i \notin S$, $\cQ_S[\cZ_i] \prec n^{-1/2-|S|/2}$.
\item Uniformly over $S \subset \{1,\ldots,n\}$ with $|S| \leq l$
and $i,j \notin S$ with $i \neq j$, $\cQ_S[\cZ_{ij}] \prec n^{-1/2-|S|/2}$.
\item Uniformly over $S \subset \{1,\ldots,n\}$ with $|S| \leq l$
and distinct $i,j,r \notin S$, $\cQ_S[\cK_r^{(ij)}] \prec n^{-1/2-|S|/2}$.
\end{enumerate}
\end{lem}
\begin{proof}
For (a), since $\r_0$ does not depend on $H$, we have
$\cQ_S[\cZ_i]=(\r_0)_{ii}\cQ_{S \cup \{i\}}[R_{ii}^{-1}](\r_0)_{ii}$,
so by H\"older's inequality and the bound
$\|(\r_0)_{ii}\|_\op \leq \delta$ it suffices to show
\[\cQ_S[\cQ_i[R_{ii}^{-1}]] \prec n^{-1/2-|S|/2}.\]
We remark that for $S=\emptyset$, this follows from
$\cQ_i[R_{ii}^{-1}]=-\Delta_i+\Oprec(n^{-1}) \prec n^{-1/2}$ (c.f.\
(\ref{eq:DeltaZequiv}) and (\ref{eq:optimalDeltaibound}) below).
For (c), similarly it suffices to show
\[\cQ_S[\cQ_r[R_{rj}^{(i)}]] \prec n^{-1/2-|S|/2}.\]
The argument for all three parts is then the same as in Lemma
\ref{lemma:bootstrapFA}, applying Lemma \ref{resolvent identities}(c) to
iteratively expand $R_{ii}^{-1}$, $R_{ij}$, and $R_{rj}^{(i)}$ in the indices
of $S$, and applying now the optimal estimates
$R_{pq}^{(S')} \prec n^{-1/2}$ from (\ref{off diagonal estimates})
to bound the terms of the expansion. We omit further details for brevity.
\end{proof}

\begin{proof}[Proof of Theorem \ref{aux:main theorem}, (\ref{X-valued Stieltjes transform estimates})]

Let us specialize (\ref{eq:expression of Delta_i}) to $S=\emptyset$, and write
simply $\Delta_i=\Delta_i^{(S)}$. Then
\begin{equation}\label{eq:Deltaapprox}
\Delta_i={-}h_{ii}\x+
\x\Bigg[\sum_r^{(i)}\left(|h_{ir}|^2-\frac{1}{n}\right)R_{rr}^{(i)} + 
\sum_{r\neq s}^{(i)}h_{ir}R_{rs}^{(i)}h_{si}\Bigg]\x
+\Oprec(n^{-1}),
\end{equation}
using the bounds $\mathrm{II} \prec n^{-1}$ from (\ref{eq:DeltaII}) and
and $\mathrm{III} \prec n^{-1}$ from (\ref{eq:DeltaIII}), 
(\ref{off diagonal estimates}), and (\ref{eq:invRiibound}). Furthermore,
(\ref{eq:Deltapnormbound}) from the final iteration of bootstrapping with
$\alpha'=1/2$ shows
\begin{equation}\label{eq:optimalDeltaibound}
\Delta_i \prec n^{-1/2}.
\end{equation}

Define
\begin{equation}\label{eq:R_ii leave out error}
    \widehat{R}_{ii} = \bigg({-}\z_i -
\x\bigg[\frac{1}{n}\sum_{j=1}^n R_{jj}\bigg]\x\bigg)^{-1}.
\end{equation}
By (\ref{DX-valued Stieltjes transform estimates}) already shown, we have
    $R_{ii} = (\r_0)_{ii} + \Oprec(n^{-1/2})$ uniformly in $i \in
\{1,\ldots,n\}$. Then averaging over $i$ gives
    \begin{align*}
        \frac{1}{n}\sum_{i=1}^n R_{ii} = \frac{1}{n}\Tr \otimes 1[\r_0] +
\Oprec(n^{-1/2}) = \m_0 + \Oprec(n^{-1/2}).
    \end{align*}
Applying this, the identity $\a^{-1}-\b^{-1}=\a^{-1}(\b-\a)\b^{-1}$, and
$\|\widehat R_{ii}\|_\op,\|(-\z_i-\x\m_0\x)^{-1}\|_\op \leq \delta^{-1}$
by Lemma \ref{im bound inverse}, we may approximate $\widehat R_{ii}$
in (\ref{eq:R_ii leave out error}) as
    \begin{align*}
        \widehat{R}_{ii} = (-\z_i - \x\m_0\x + \Oprec(n^{-1/2}))^{-1} = (-\z_i -
\x\m_0\x)^{-1} + \Oprec(n^{-1/2}) = (\r_0)_{ii} + \Oprec(n^{-1/2}),
    \end{align*}
the last identity using the characterization of $\r_0$ via the fixed-point
equation of Lemma \ref{lem:uniquefixedpoint}(a). Therefore,
applying again $\a^{-1}-\b^{-1}=\a^{-1}(\b-\a)\b^{-1}$
with $\widehat R_{ii}^{-1}=R_{ii}^{-1}+\Delta_i$ from
(\ref{eq:Riiapproxfixedpoint}), we have
\begin{equation}\label{eq:RhatR}
        R_{ii} = \widehat{R}_{ii} + R_{ii}\Delta_i\widehat{R}_{ii} =
\widehat{R}_{ii} + (\r_0)_{ii}\Delta_i(\r_0)_{ii} + \Oprec(n^{-1})
\end{equation}
where this applies (\ref{DX-valued Stieltjes transform estimates}) and
(\ref{eq:optimalDeltaibound}) to bound the error.

Comparing the form
(\ref{eq:Deltaapprox}) for $\Delta_i$ with the expansion of $R_{ii}^{-1}$ in
Lemma \ref{resolvent identities}(a), and noting that $\cQ_i[\z_i]=0$ because
$\z_i$ does not depend on $H$, we have
\begin{equation}\label{eq:DeltaZequiv}
\Delta_i={-}\cQ_i[R_{ii}^{-1}]+\Oprec(n^{-1}).
\end{equation}
Thus, setting $\cZ_i=(\r_0)_{ii}\cQ_i[R_{ii}^{-1}](\r_0)_{ii}$ and averaging
again over $i$,
    \begin{align*}
        \frac{1}{n}\sum_{i=1}^n R_{ii} = \frac{1}{n}\sum_{i=1}^n \widehat{R}_{ii}
-\frac{1}{n}\sum_{i=1}^n \cZ_i+\Oprec(n^{-1}).
    \end{align*}
By Lemma \ref{lemma:fluctuationavg}(a) (or (b)) applied with the estimates of
Lemma \ref{lemma:mainproofFA}(a), we have $n^{-1}\sum_i \cZ_i \prec n^{-1}$.
Thus, defining
\[\widehat\Delta=\frac{1}{n}\sum_{i=1}^n R_{ii}-\frac{1}{n}\sum_{i=1}^n
\widehat{R}_{ii}
=\frac{1}{n}\sum_{i=1}^n R_{ii}-\frac{1}{n}\sum_{i=1}^n \bigg({-}\z_i
-\x\bigg[\frac{1}{n}\sum_{j=1}^n R_{jj}\bigg]\x\bigg)^{-1},\]
this shows $\widehat \Delta \prec n^{-1}$. Observe also that
    \begin{align*}
        \|\widehat{\Delta}\|_\op
=\bigg\|\frac{1}{n}\sum_{i=1}^n (R_{ii}-\widehat R_{ii})\bigg\|_\op
=\bigg\|\frac{1}{n}\sum_{i=1}^n
R_{ii}\Delta_i\widehat{R}_{ii}\bigg\|_\op
\leq\delta^{-2}\max_{i=1}^n \pnorm{\Delta_i}{\op}.
    \end{align*}
By the weak bound (\ref{eq:Deltaopnormbound}) for $\|\Delta_i\|_\op$ on
a high-probability event $\cE$, this implies
$\Im(\z_i+\x\widehat\Delta\x) \geq \delta/2$ for all $i=1,\ldots,n$ and all $n
\geq n_0(\gamma,\delta)$.
On this event $\cE$,
Corollary \ref{cor:Stability of DX-valued fixed point}(b) implies,
for each $p \in [1,\infty)$,
\[\left\|\frac{1}{n}\sum_{i=1}^n R_{ii}-\m_0\right\|_p \leq
(1+2\gamma^2\delta^{-2})\|\widehat \Delta\|_p.\]
Thus $(n^{-1}\Tr\otimes 1)[R]-\m_0 \prec n^{-1}$, showing
(\ref{X-valued Stieltjes transform estimates}).
\end{proof}

\begin{proof}[Proof of Theorem \ref{aux:main theorem},
(\ref{isotropic estimates})]

By the polarization identity
\begin{align*}
(\u\otimes 1)^*R(\v\otimes 1)&=
\frac{1}{4}\Big[((\v+\u)\otimes 1)^*R((\v+\u)\otimes 1)
    -((\v-\u)\otimes 1)^*R((\v-\u)\otimes 1)\\
    &\qquad+i((\v+i\u)\otimes 1)^*R((\v+i\u)\otimes 1)
    -i((\v-i\u)\otimes 1)^*R((\v-i\u)\otimes 1)\Big]
\end{align*}
it suffices to prove the result uniformly over $\u=\v \in \C^n$ satisfying
$\|\v\|_2 \leq \sqrt{2}\upsilon$. We have
    \begin{align*}
        (\v\otimes 1)^*R(\v\otimes 1) &= \sum_i|v_i|^2R_{ii} + \sum_{i\neq j}\bar{v}_iv_jR_{ij}\\
        &= \sum_i|v_i|^2(\r_0)_{ii} + \sum_{i\neq j}\bar{v}_iv_jR_{ij} +
\Oprec(n^{-1/2})\\
&=(\v \otimes 1)^*\r_0(\v \otimes 1)+\sum_{i\neq j}\bar{v}_iv_jR_{ij} +
\Oprec(n^{-1/2})
    \end{align*}
    where the second line applies  $\|\v\|_2 \prec 1$ and
(\ref{DX-valued Stieltjes transform estimates}). It remains to
show $\sum_{i\neq j}\bar{v}_iv_jR_{ij} \prec n^{-1/2}$.

Separating $R_{ij}$ into its conditional mean and fluctuations,
    \begin{align*}
        \sum_{i\neq j}\bar{v}_iv_jR_{ij} = \sum_{i\neq
j}\bar{v}_iv_j\E_i\E_j[R_{ij}] + \sum_{i\neq j}\bar{v}_iv_j\E_i\cQ_j[R_{ij}] +
\sum_{i\neq j}\bar{v}_iv_j\E_j\cQ_i[R_{ij}] + \sum_{i\neq
j}\bar{v}_iv_j\cQ_i\cQ_j[R_{ij}].
    \end{align*}
    We first examine $\E_i[R_{ij}]$. Recall $\widehat{R}_{ii}$ from
(\ref{eq:R_ii leave out error}). Observe that
$R_{ii}=(\r_0)_{ii}+\Oprec(n^{-1/2})$ by
(\ref{DX-valued Stieltjes transform estimates}),
whereas $\widehat R_{ii}=(\r_0)_{ii}+\Oprec(n^{-1})$ with the smaller error
$n^{-1}$ by its definition and the estimate
(\ref{X-valued Stieltjes transform estimates}) already shown. Then,
applying Lemma \ref{resolvent identities}(b) 
and (\ref{eq:RhatR}) to expand $R_{ij}$,
    \begin{align*}
        {-}\E_i[R_{ij}] &=
\E_i\left[R_{ii}\x\sum_r^{(i)}h_{ir}R_{rj}^{(i)}\right] = \E_i\left[\widehat{R}_{ii}\x\sum_r^{(i)}h_{ir}R_{rj}^{(i)}\right]
         + \E_i\left[R_{ii}\Delta_i\widehat{R}_{ii}\x\sum_r^{(i)}h_{ir}R_{rj}^{(i)}\right]\\
        &=
\underbrace{\E_i\left[(\r_0)_{ii}\x\sum_r^{(i)}h_{ir}R_{rj}^{(i)}\right]}_{=\mathrm{I}}
+
\underbrace{\E_i\left[(\r_0)_{ii}\Delta_i(\r_0)_{ii}\x\sum_r^{(i)}h_{ir}R_{rj}^{(i)}\right]}_{=\mathrm{II}}
         + \underbrace{\E_i\left[\Oprec(n^{-1}) \cdot
\x\sum_r^{(i)}h_{ir}R_{rj}^{(i)}\right]}_{=\mathrm{III}}.
    \end{align*}
In these expressions, by (\ref{off diagonal estimates}),
(\ref{eq:invRiibound}), and Lemma \ref{resolvent identities}(b), we have
\begin{equation}\label{eq:offdiagexpansionbound}
\sum_r^{(i)} h_{ir} R_{rj}^{(i)}
={-}\x^{-1}R_{ii}^{-1}R_{ij} \prec n^{-1/2}
\end{equation}
so $\mathrm{III} \prec n^{-3/2}$.
Since $R^{(i)}_{rj}$ is a function only of $H^{(i)}$
(and $\r_0,\x$ do not depend on $H$), it follows from $\E[h_{ir}]=0$
that $\mathrm{I}=0$. For $\mathrm{II}$, recall the form of
$\Delta_i$ from (\ref{eq:Deltaapprox}).
Substituting this expression of $\Delta_i$ into $\mathrm{II}$ gives
    \begin{align*}
        \mathrm{II} &= - (\r_0)_{ii}\x(\r_0)_{ii}\x\sum_r^{(i)}\E_i[h_{ii}h_{ir}]R_{rj}^{(i)}
        + \sum_{t,r}^{(i)}\E_i\left[\left(|h_{it}|^2 - \frac{1}{n}\right)h_{ir}\right](\r_0)_{ii}\x R_{tt}^{(i)}\x(\r_0)_{ii}\x R_{rj}^{(i)}\\
        &\qquad+ \sum_r^{(i)} \sum_{t\neq
s}^{(i)}\E_i[h_{it}h_{si}h_{ir}](\r_0)_{ii}\x
R_{ts}^{(i)}\x(\r_0)_{ii}\x R_{rj}^{(i)} + \Oprec(n^{-3/2}).
    \end{align*}
The first term is 0 since $r\neq i$.
Similarly, the second term is 0 for summands $r \neq t$,
and the third term is 0 since at least one of $r,s,t$ is distinct from
the other two. Thus
\begin{align*}
        \mathrm{II}= 
        \sum_{r}^{(i)}\E_i[|h_{ir}|^2h_{ir}](\r_0)_{ii}\x
R_{rr}^{(i)}\x(\r_0)_{ii}\x R_{rj}^{(i)} + \Oprec(n^{-3/2}).
\end{align*}
Note that the single summand for $r=j$ is $\Oprec(n^{-3/2})$, 
that all summands for $r \neq j$ are $\Oprec(n^{-2})$ by (\ref{off diagonal
estimates}), and that
$R_{rr}^{(i)}=R_{rr}+\Oprec(n^{-1})=(\r_0)_{rr}+\Oprec(n^{-1/2})$. Then
we may further write this as
\begin{align*}
        \mathrm{II} &=
\sum_{r}^{(ij)}\E_i[|h_{ir}|^2h_{ir}](\r_0)_{ii}\x
R_{rr}^{(i)}\x(\r_0)_{ii}\x R_{rj}^{(i)}+\Oprec(n^{-3/2})\\
&=\sum_{r}^{(ij)}\E_i[|h_{ir}|^2h_{ir}](\r_0)_{ii}\x(\r_0)_{rr}\x(\r_0)_{ii}\x
R_{rj}^{(i)}+\Oprec(n^{-3/2})\\
&=\sum_{r}^{(ij)}\E_i[|h_{ir}|^2h_{ir}](\r_0)_{ii}\x(\r_0)_{rr}\x(\r_0)_{ii}\x(\E_r+\cQ_r)[R_{rj}^{(i)}]+\Oprec(n^{-3/2}).
\end{align*}
By Lemma \ref{lemma:fluctuationavg}(a) applied with the estimates of Lemma
\ref{lemma:mainproofFA}(c), we have
\[\sum_{r}^{(ij)}\E_i[|h_{ir}|^2h_{ir}]
(\r_0)_{ii}\x(\r_0)_{rr}\x(\r_0)_{ii}\x\cQ_r[R_{rj}^{(i)}]
\prec n^{-3/2}.\]
    We now examine $\E_r[R_{rj}^{(i)}]$.
Again by Lemmas \ref{resolvent identities}(b) and \ref{lemma:domination}(c),
    \begin{align*}
        {-}\E_r[R_{rj}^{(i)}]=\E_r\left[R_{rr}^{(i)}\x\sum_s^{(ir)}h_{rs}R_{sj}^{(ir)}\right]
= (\r_0)_{rr}\x\E_r\left[\sum_s^{(ir)}h_{rs}R_{sj}^{(ir)}\right] +
O_\prec(n^{-1}).
    \end{align*}
    This first term is 0, so $\E_r[R_{rj}^{(i)}]\prec n^{-1}$. Combining the
above gives $\mathrm{II} \prec n^{-3/2}$, and hence
    $\E_i[R_{ij}]\prec n^{-3/2}$. By symmetry, also $\E_j[R_{ij}] \prec
n^{-3/2}$, so we conclude that
    \begin{align*}
        \sum_{i\neq j}\bar{v}_iv_jR_{ij} = \sum_{i\neq
j}\bar{v}_iv_j\cQ_i\cQ_j[R_{ij}] + \sum_{i\neq j}\bar{v}_iv_j \cdot
\Oprec(n^{-3/2})
=\sum_{i\neq j}\bar{v}_iv_j\cQ_i\cQ_j[R_{ij}]+\Oprec(n^{-1/2}),
    \end{align*}
where the second equality applies $\sum_i |v_i| \leq \sqrt{n}\|\v\|_2 \prec
\sqrt{n}$.
Finally, by Lemma \ref{lemma:fluctuationavg}(c) applied with the
estimates of Lemma \ref{lemma:mainproofFA}(b), we have
\[\sum_{i \neq j} \bar v_i v_j \cQ_i\cQ_j[R_{ij}] \prec \frac{1}{\sqrt{n}}
\left(\sum_{i \neq j} |v_i|^2|v_j|^2\right)^{1/2} \prec n^{-1/2}.\]
So $\sum_{i\neq j}\bar{v}_iv_j R_{ij} \prec n^{-1/2}$
as desired, completing the proof.
\end{proof}

\section{Analysis of the Kronecker deformed Wigner model}\label{sec:resolvent}

We now prove Proposition \ref{rmt:fixed point characterization} and
Theorem \ref{rmt:main theorem}. For spectral arguments $z \in \C^+$,
recall the following quantities from Section \ref{sec:model}:
\begin{align*}
    Q&=A \otimes I+I \otimes B+\Theta \otimes \Xi \in
\C^{n^2 \times n^2},\\
\q&=\a \otimes 1+1 \otimes \b+\Theta \otimes \Xi \in \cA \otimes \cA,\\
    G(z)&=(Q-z\,I \otimes I)^{-1}, \qquad m(z)=n^{-2}\Tr G(z),\\
\g(z)&=(\q-z\,1 \otimes 1)^{-1},\qquad G_0(z)=(\tau^\cD \otimes
\tau^\cD)[\g(z)], \qquad m_0(z)=\tau \otimes \tau[\g(z)].
\end{align*}

We first show Proposition \ref{rmt:fixed point characterization} and
all statements of Theorem \ref{rmt:main theorem} except the estimate
$G_{ij,\alpha\beta} \prec n^{-1}$ of (\ref{rmt:1/n entries}) using
the analyses of Section \ref{sec:general}.

\begin{proof}[Proof of Proposition \ref{rmt:fixed point characterization}]
We apply Lemma \ref{lem:uniquefixedpoint}(b) with $\cX=\cA$, $\x=1_\cA$, and
\[\z={-}I \otimes \b-\Theta \otimes \Xi
+z\,I \otimes 1_\cA \in \C^{n \times n} \otimes \cA.\]
Then $\Im \z=(\Im z)(I \otimes 1)$, so $\z_i={-}\b-\theta_i\Xi+z \in \cA^+$
for each $i=1,\ldots,n$. Then
Lemma \ref{lem:uniquefixedpoint}(b) ensures that
$\m_b(z)=(\tau \otimes 1)[\g(z)]$ is the unique fixed point
in $\cA^+$ to the fixed-point equation (\ref{eq:mb}). Similarly
$\m_a(z)=(1 \otimes \tau)[\g(z)]$ is the unique fixed point
in $\cA^+$ to (\ref{eq:ma}), and the identity
$m_0(z)=\tau[\m_a(z)]=\tau[\m_b(z)]$ follows from taking a second trace $\tau$
for either $\m_a$ or $\m_b$.
Furthermore, Lemma \ref{lem:uniquefixedpoint}(a) shows that
\[(\tau^\cD \otimes 1)[\g(z)]=
({-}\z-I \otimes \m_b(z))^{-1}=\sum_{i=1}^n E_{ii} \otimes
(\b+\theta_i\Xi-z-\m_b(z))^{-1}.\]
Then applying $1 \otimes \tau^\cD$ shows
\[G_0=\sum_{i=1}^n E_{ii} \otimes
\tau^\cD\big[(\b+\theta_i\Xi-z-\m_b(z))^{-1}\big],\]
and similarly
\[G_0=\sum_{\alpha=1}^n 
\tau^\cD\big[(\a+\xi_\alpha\Theta-z-\m_a(z))^{-1}\big]
\otimes E_{\alpha\alpha}.\]
\end{proof}

\begin{proof}[Proof of Theorem \ref{rmt:main theorem}, (\ref{rmt:Stieltjes
transform}--\ref{rmt:root n entries}) and (\ref{rmt:isotropic entries})]

Define
\begin{equation}\label{eq:tildeg}
\tilde \g(z)
=(\a \otimes I+1_\cA \otimes B+\Theta \otimes \Xi-z1_\cA \otimes I)^{-1} \in \cA
\otimes \C^{n \times n}
\end{equation}
and abbreviate $G=G(z)$, $\tilde \g=\tilde \g(z)$, and $\g=\g(z)$.
Fix any $\epsilon>0$ and consider the event $\cE=\{\|B\|_\op \leq 3\}$.
We apply Theorem \ref{aux:main theorem} conditional on $B$ and
this event $\cE$, with $\cX=\C^{n \times n}$, $R=G$, $\r_0=\tilde \g$,
$H=A$, $\h=\a$, $\x=I$, and
\[\z={-}I \otimes B-\Theta \otimes \Xi+z\,I \otimes I.\]
Then Assumption \ref{assump:general} holds (for $\gamma=1$ and modified
constants $\upsilon,\delta>0$).
Conditional on $B$ and the event $\cE$, Theorem \ref{aux:main theorem} shows
over the randomness of $A$, uniformly
in $i \neq j$ and $\u,\u' \in \C^n$ with $\|\u\|_2,\|\u'\|_2 \leq \upsilon$,
\begin{equation}\label{eq:layer1}
\begin{gathered}
(n^{-1}\Tr \otimes I)[G]-(\tau \otimes I)[\tilde \g] \prec n^{-1},\\
G_{ii}-(\tau^\cD \otimes I)[\tilde \g]_{ii} \prec n^{-1/2}, \qquad
G_{ij} \prec n^{-1/2},\\
(\u \otimes I)^* G (\u' \otimes I)
-(\u \otimes I)^* (\tau^\cD \otimes I)[\tilde \g](\u' \otimes I)
\prec n^{-1/2}.
\end{gathered}
\end{equation}
In light of the bound $G_{ii}^{-1} \prec 1$ from (\ref{eq:invRiibound}),
the form $(\tau^{\cD} \otimes I)[\tilde\g]_{ii}=(B+\theta_i\Xi-z-M_B(z))^{-1}$
from Lemma \ref{lem:uniquefixedpoint}(a)
where $M_B(z)=(\tau \otimes 1)[\tilde \g(z)]$,
and the identity $A^{-1}-B^{-1}=A^{-1}(B-A)B^{-1}$,
the second statement here shows also
\[G_{ii}^{-1}-\big((\tau^\cD \otimes I)[\tilde\g]_{ii}\big)^{-1}
=G_{ii}^{-1}-(B+\theta_i\Xi-z-M_B(z)) \prec n^{-1/2},\]
hence
\[G_{ii}^{-1}-G_{jj}^{-1}=(\theta_i-\theta_j)\Xi+\Oprec(n^{-1/2}).\]
As $\P[\cE]>1-n^{-D}$ for any fixed $D>0$ and all $n \geq n_0(D)$,
these statements then also hold unconditionally.
In particular, by Remark \ref{rek:op holds in finite dimension} this implies
$\|G_{ij}\|_\op \prec n^{-1/2}$ and
$\|G_{ii}^{-1}-G_{jj}^{-1}-(\theta_i-\theta_j)\Xi\|_\op \prec n^{-1/2}$.
The argument for $G_{\alpha\alpha}$ and $G_{\alpha\beta}$ is symmetric,
so this shows (\ref{rmt:diagonalblocksa}--\ref{rmt:offdiagonalblocks}).
The bounds (\ref{rmt:root n entries}) are an immediate consequence of
(\ref{rmt:offdiagonalblocks}).

To prove the remaining statements (\ref{rmt:Stieltjes transform}),
(\ref{rmt:diagonal entries}), and (\ref{rmt:isotropic entries}), we
apply Theorem \ref{aux:main theorem} again to the second tensor factor,
with $\cX=\cA$, $R=\tilde \g$, $\r_0=\g$, $H=B$, $\h=\b$, $\x=1_\cA$, and
\[\z={-}\a \otimes I-\Theta \otimes \Xi+z\,1_\cA \otimes I.\]
This gives, uniformly in $\alpha \neq \beta$ and
$\v,\v' \in \C^n$ with $\|\v\|_2,\|\v'\|_2 \leq \upsilon$,
\begin{equation}\label{eq:layer2}
\begin{gathered}
(1 \otimes n^{-1}\Tr)[\tilde \g]-(1 \otimes \tau)[\g] \prec n^{-1},\\
\tilde \g_{\alpha\alpha}-(1 \otimes \tau^\cD)[\g]_{\alpha\alpha} \prec
n^{-1/2}, \qquad \tilde \g_{\alpha\beta} \prec n^{-1/2},\\
(1 \otimes \v)^* \tilde \g (1 \otimes \v')
-(1 \otimes \v)^* (1 \otimes \tau^\cD)[\g](1 \otimes \v')
\prec n^{-1/2}.
\end{gathered}
\end{equation}
We remark that if $T:\cA \to \cB$ and $T':\cA' \to \cB'$ are two linear maps
between vector spaces, then for any $x \in \cA \otimes \cA'$,
\begin{equation}\label{eq:linearmapcomp}
(T \otimes T')[x]=(T \otimes 1)(1 \otimes T')[x]
=(1 \otimes T')(T \otimes 1)[x].
\end{equation}
Thus we may
combine the first statements of (\ref{eq:layer1}) and (\ref{eq:layer2}) to get
as desired
    \begin{align*}
        m(z) &= (n^{-1}\Tr\,\otimes\,n^{-1}\Tr)[G]
=n^{-1}\Tr\big[(n^{-1}\Tr \otimes I)[G]\big]\\
        &= (\tau \otimes n^{-1}\Tr)[\tilde \g]+n^{-1}\Tr[\Delta_a]
=\tau\big[(1 \otimes n^{-1}\Tr)[\tilde \g]\big]+n^{-1}\Tr[\Delta_a]\\
        &= \tau \otimes \tau[\g] + n^{-1}\Tr[\Delta_a] + \tau[\Delta_b]
        = m_0(z) + n^{-1}\Tr[\Delta_a] + \tau[\Delta_b]
    \end{align*}
where $\Delta_a\in\C^{n \times n}$ and $\Delta_b \in \cA$ are errors satisfying
$\Delta_a, \Delta_b \prec n^{-1}$. We have
$|n^{-1}\Tr \Delta_a | \leq \|\Delta_a\|_1 \prec n^{-1}$
and $|\tau[\Delta_b]| \leq \|\Delta_b\|_1 \prec n^{-1}$ (Lemma \ref{Holder's
inequality}), showing (\ref{rmt:Stieltjes transform}). Similarly,
applying
(\ref{eq:linearmapcomp}) with $(T,T')=(\u^*[\cdot]\u',\v^*[\cdot]\v')$
and $(T,T')=(\tau^\cD,\v^*[\cdot]\v')$,
we may combine the last
statements of (\ref{eq:layer1}) and (\ref{eq:layer2}) to get
\begin{align*}
        (\u \otimes \v)^* G (\u' \otimes \v')&=
\v^*\big[(\u \otimes I)^* G (\u' \otimes I)\big]\v'\\
&=(\u \otimes \v)^* (\tau^\cD \otimes I)
[\tilde \g](\u' \otimes \v')+\v^*\Delta_1'\v'\\
&=\u^*\tau^\cD[(1 \otimes \v)^*\tilde \g(1 \otimes \v')]\u'+\v^*\Delta_1'\v'\\
        &=(\u \otimes \v)^* (\tau^\cD \otimes \tau^\cD)[\g](\u' \otimes \v')
+\v^*\Delta_1'\v' + \u^*\tau^\cD[\Delta_2']\u'
\end{align*}
where $\Delta_1'\in\C^{n \times n}$ and $\Delta_2' \in \cA$ satisfy
$\Delta_1',\Delta_2' \prec n^{-1/2}$. Since $\tau^\cD$ is a $L^p$-contraction
for each fixed $p$ (Lemma \ref{lemma:Lpcontraction}), this implies also
$\tau^\cD[\Delta_2'] \prec n^{-1/2}$. Then, since 
$\Delta_1',\tau^\cD(\Delta_2') \in \C^{n \times n}$,
by Remark \ref{rek:op holds in finite dimension} we have
$\|\Delta_1'\|_\op,\|\tau^\cD(\Delta_2')\|_\op \prec n^{-1/2}$ and hence
$\v^*\Delta_1'\v' \prec n^{-1/2}$ and $\u^*\tau^\cD[\Delta_2']\u' \prec
n^{-1/2}$,
showing (\ref{rmt:isotropic entries}). Finally, specializing this to
$\u=\u'=\e_i$ and $\v=\v'=\e_\alpha$ shows (\ref{rmt:diagonal entries}).
\end{proof}

In the remainder of this section, we show the final estimate (\ref{rmt:1/n
entries}) of Theorem \ref{rmt:main theorem},
$G_{ij,\alpha\beta} \prec n^{-1}$
when $i \neq j$ and $\alpha \neq \beta$. Recalling $\tilde \g$ from
(\ref{eq:tildeg}), define the matrices in $\C^{n \times n}$
\begin{equation}\label{eq:Midef}
M_B=(\tau \otimes I)[\tilde \g], \qquad
M_i=(B+\theta_i\Xi-zI-M_B)^{-1} \text{ for } i=1,\ldots,n
\end{equation}
so that (by Lemma \ref{lem:uniquefixedpoint}(b))
$M_B=n^{-1}\sum_i M_i$. Symmetrically, letting
\[\check\g=(A \otimes 1_\cA+I \otimes \b
+\Theta \otimes \Xi-z\,I \otimes 1_\cA)^{-1} \in \C^{n \times n} \otimes \cA,\]
define
\begin{equation}\label{eq:Midef}
M_A=(I \otimes \tau)[\check \g], \qquad
M_\alpha=(A+\xi_\alpha\Theta-zI-M_A)^{-1} \text{ for } \alpha=1,\ldots,n
\end{equation}
so that $M_A=n^{-1}\sum_\alpha M_\alpha$.
We denote the commuting projection operators
\[\E_i=\E[\cdot|A^{(i)},B],\qquad \cQ_i=1-\E_i,\qquad
\E_\alpha=\E[\cdot|A,B^{[\alpha]}],\qquad\cQ_\alpha=1-\E_\alpha\]
and write as before $\cQ_S=\prod_{i \in S} \cQ_i$.
With slight abuse of notation, we will use the distinction between Greek and
Roman indices to distinguish between $M_\alpha$ and $M_i$, $\E_\alpha$ and
$\E_i$, and $\cQ_\alpha$ and $\cQ_i$.

\begin{lem}\label{lemma:offdiagFA}
In the setting of Theorem \ref{rmt:main theorem}, for any
distinct $i,j,k$ and any $\alpha \neq \beta$, define
\[\Delta^k=G_{kj}E_{\beta\alpha}G_{ik} \in \C^{n \times n}.\]
Fix any $l \geq 1$. Then uniformly over 
$S \subseteq \{1,\ldots,n\}$ with $|S| \leq l$,
distinct $i,j,k \notin S$, and $\alpha \neq \beta$,
\[\cQ_{S \cup \{k\}}[\Delta^k] \prec n^{-1-|S|/2}.\]
\end{lem}
\begin{proof}
The proof is the same as that of Lemmas \ref{lemma:bootstrapFA}
and \ref{lemma:mainproofFA}, using Lemma \ref{resolvent identities}(c)
to expand $G_{kj},G_{ik}$ in the indices of $S$ and applying
the estimates $\|E_{\beta\alpha}\|_\op \leq 1$ and
$G_{pq}^{(S')} \prec n^{-1/2}$ for 
and $p,q \notin S'$ with $p \neq q$. We omit the details for brevity.
\end{proof}

\begin{proof}[Proof of Theorem \ref{rmt:main theorem}, (\ref{rmt:1/n entries})]
Throughout this proof, $G_{ij},G_{\alpha\beta}$ etc.\ are all matrices in
$\C^{n \times n}$, so the stochastic domination notation $\prec$ may be
understood in the operator norm sense,
c.f.\ Remark \ref{rek:op holds in finite dimension}.
We recall our convention of Roman indices $i,j,k,\ldots$ for the first tensor
factor and Greek indices $\alpha,\beta,\gamma,\ldots$ for the second tensor
factor. In the following, we use a superscript $(\cdot)$ to denote a minor 
on $A$ and $[\cdot]$ to denote a minor on $B$. 

Fix any indices $i \neq j$ and $\alpha \neq \beta$. Applying Lemma
\ref{resolvent identities}(b) to the second tensor factor,
\[G_{\alpha\beta}={-}G_{\alpha\alpha}\sum_\gamma^{[\alpha]} b_{\alpha\gamma}
G_{\gamma\beta}^{[\alpha]}.\]
Recall (c.f.\ (\ref{eq:offdiagexpansionbound})) that
\begin{equation}\label{eq:onestepexpansiona}
\sum_\gamma^{[\alpha]} b_{\alpha\gamma} G_{\gamma\beta}^{[\alpha]} \prec
n^{-1/2}.
\end{equation}
By (\ref{DX-valued Stieltjes transform estimates})
applied to the second tensor factor, we have
$G_{\alpha\alpha}=M_\alpha+\Oprec(n^{-1/2})$, where $M_\alpha$ depends only on
$A$ and not on $B$. Then by Lemma \ref{lemma:domination}(c), also
$\E_\alpha G_{\alpha\alpha}=M_\alpha+\Oprec(n^{-1/2})$, so
\begin{equation}\label{eq:onestepexpansionb}
G_{\alpha\alpha}-\E_\alpha G_{\alpha\alpha} \prec n^{-1/2}.
\end{equation}
Then applying both (\ref{eq:onestepexpansiona}) and
(\ref{eq:onestepexpansionb}),
\[G_{\alpha\beta}={-}(\E_\alpha G_{\alpha\alpha})
\sum_\gamma^{[\alpha]} b_{\alpha\gamma}
G_{\gamma\beta}^{[\alpha]}+\Oprec(n^{-1}),\]
so
\begin{align}
G_{ij,\alpha\beta}={-}\sum_\gamma^{[\alpha]} b_{\alpha\gamma}
\e_i^*(\E_\alpha G_{\alpha\alpha})G_{\gamma \beta}^{[\alpha]}\e_j+\Oprec(n^{-1})
&\prec \frac{1}{\sqrt{n}}\left(\sum_\gamma^{[\alpha]}
\left|\e_i^*(\E_\alpha G_{\alpha\alpha})
G_{\gamma\beta}^{[\alpha]}\e_j\right|^2\right)^{1/2}
+\frac{1}{n}\nonumber\\
&\prec \frac{1}{\sqrt{n}}\,\E_\alpha \left(\sum_\gamma^{[\alpha]}
\left|\e_i^*G_{\alpha\alpha}G_{\gamma\beta}^{[\alpha]}\e_j\right|^2\right)^{1/2}
+\frac{1}{n},\label{eq:Goffdiag1}
\end{align}
the first inequality applying independence of $(b_{\alpha\gamma})_{\gamma=1}^n$
with $(\E_\alpha G_{\alpha\alpha})G_{\gamma\beta}^{[\alpha]}$ and
the scalar version of Lemma \ref{lemma:concentration}(a), and the
second line applying Jensen's inequality and convexity of the $\ell_2$-norm.

Fixing $i \neq j$ and $\alpha \neq \beta$,
define $\Delta \in \C^{n \times n}$ as the matrix with entries
\begin{equation}\label{eq:Deltadef}
\Delta_{\gamma\nu}
=\e_i^* G_{\alpha\nu}G_{\gamma\beta}\e_j
=\sum_k G_{ik,\alpha\nu}G_{kj,\gamma\beta}=\sum_k\underbrace{\e_\gamma^*G_{kj}
\e_\beta\e_\alpha^* G_{ik}\e_\nu}_{=\Delta_{\gamma\nu}^k}.
\end{equation}
We claim that (uniformly over $i \neq j$ and $\alpha \neq \beta$)
\begin{equation}\label{eq:Deltabound}
\|\Delta\|_\op \prec n^{-1/2}.
\end{equation}
Then, for every $\gamma \neq \alpha$ (including $\gamma=\beta$), applying
\[G_{\alpha\alpha}G_{\gamma\beta}^{[\alpha]}
=G_{\alpha\alpha}G_{\gamma\beta}
-G_{\alpha\alpha}G_{\gamma\alpha}G_{\alpha\alpha}^{-1}G_{\alpha\beta},
\quad G_{\alpha\alpha},G_{\alpha\alpha}^{-1} \prec 1,
\quad G_{\gamma\alpha},G_{\alpha\beta} \prec n^{-1/2},\]
which follow from Lemma \ref{resolvent identities}(c), (\ref{eq:invRiibound}),
and (\ref{off diagonal estimates}) of
Theorem \ref{aux:main theorem}, we have

\[\e_i^*G_{\alpha\alpha}G_{\gamma\beta}^{[\alpha]}\e_j
=\Delta_{\gamma\alpha}-
\e_i^*G_{\alpha\alpha}G_{\gamma\alpha}G_{\alpha\alpha}^{-1}G_{\alpha\beta}\e_j
=\Delta_{\gamma\alpha}+\Oprec(n^{-1})\]
and hence by (\ref{eq:Deltabound}),
\[\left(\sum_\gamma^{[\alpha]}
\left|\e_i^*G_{\alpha\alpha}G_{\gamma\beta}^{[\alpha]}\e_j\right|^2\right)^{1/2}
\leq \|\Delta\|_\op+n^{-1/2} \prec n^{-1/2}.\]
Applying this in (\ref{eq:Goffdiag1}) yields the desired bound
$G_{ij,\alpha\beta} \prec n^{-1}$.

It remains to show (\ref{eq:Deltabound}). For this, defining
$\Delta=\sum_k \Delta^k$ where $\Delta^k=G_{kj}E_{\beta\alpha}G_{ik}$
as in (\ref{eq:Deltadef}), observe that
\[\|\Delta^k\|_\op \leq \|G_{kj}\|_\op \cdot \|G_{ik}\|_\op
\prec \begin{cases} n^{-1/2} & \text{ if } k \in \{i,j\}\\
n^{-1} & \text{ if } k \notin \{i,j\} \end{cases}\]
Then
\[\Delta=\sum_k^{(ij)} (\E_k+\cQ_k)[\Delta^k]+\Oprec(n^{-1/2}).\]
By Lemma \ref{lemma:fluctuationavg}(a) applied with the estimates of
Lemma \ref{lemma:offdiagFA}, we have
$\sum_k^{(ij)} \cQ_k[\Delta^k] \prec n^{-1/2}$, so
\[\Delta=\sum_k^{(ij)} \E_k[\Delta^k]+\Oprec(n^{-1/2}).\]
For any $k \notin \{i,j\}$, we have analogously to (\ref{eq:onestepexpansiona})
and (\ref{eq:onestepexpansionb}) that
\[G_{kk}-M_k \prec n^{-1/2},
\qquad \sum_\ell^{(k)} a_{k \ell} G_{\ell j}^{(k)} \prec n^{-1/2},
\qquad \sum_\ell^{(k)} G_{i\ell}^{(k)}a_{\ell k}  \prec n^{-1/2}.\]
Hence, applying the resolvent identities of Lemma \ref{resolvent identities}(b),
\begin{align*}
G_{ik}&={-}\sum_\ell^{(k)} G_{i\ell}^{(k)}a_{\ell k}G_{kk}
={-}\sum_\ell^{(k)} G_{i\ell}^{(k)}a_{\ell k} M_k+\Oprec(n^{-1})\\
G_{kj}&={-}\sum_\ell^{(k)} G_{kk} a_{k \ell} G_{\ell j}^{(k)}
={-}\sum_\ell^{(k)} M_k a_{k \ell} G_{\ell j}^{(k)}+\Oprec(n^{-1}).
\end{align*}
Applying this into the definition of $\Delta^k$,
\begin{align*}
\sum_k^{(ij)} \E_k[\Delta^k]
&=\sum_k^{(ij)}\E_k\Big[G_{kj}E_{\beta\alpha}G_{ik}\Big]\\
&=\sum_k^{(ij)} \sum_{\ell,m}^{(k)}
\E_k[a_{km}a_{\ell k}]M_kG_{mj}^{(k)}E_{\beta\alpha}
G_{i\ell}^{(k)}M_k+\Oprec(n^{-1/2})\\
&=\frac{1}{n}\sum_k^{(ij)} 
\sum_\ell^{(ijk)} M_kG_{\ell j}^{(k)}E_{\beta\alpha}
G_{i\ell}^{(k)}M_k+\Oprec(n^{-1/2}),
\end{align*}
the last equality using $\E[a_{km}a_{\ell k}]=n^{-1}\1\{\ell=m\}$
and then absorbing the summands with $\ell \in \{i,j\}$ into the
$\Oprec(n^{-1/2})$ error. Now applying $G_{\ell j}-G_{\ell
j}^{(k)},G_{i\ell}-G_{i\ell}^{(k)} \prec n^{-1}$
(as follows from Lemma \ref{resolvent identities}(c)),
$M_k \prec 1$,
and $G_{\ell j}^{(k)},G_{i\ell}^{(k)} \prec n^{-1/2}$ for all
$\ell \notin \{i,j,k\}$, we get
\begin{align*}
\sum_k^{(ij)} \E_k[\Delta^k]&=\frac{1}{n}\sum_k^{(ij)} \sum_\ell^{(ijk)}
M_kG_{\ell j} E_{\beta\alpha} G_{i\ell} M_k+\Oprec(n^{-1/2})\\
&=\frac{1}{n}\sum_k \sum_\ell M_kG_{\ell j}E_{\beta\alpha}
G_{i\ell}M_k+\Oprec(n^{-1/2}),
\end{align*}
the second line introducing an additional $\Oprec(n^{-1/2})$ errors upon
including the summands with $\ell \in \{i,j,k\}$, followed by $k \in \{i,j\}$.
Observing that
\[\frac{1}{n}\sum_{k,\ell=1}^n M_kG_{\ell j}E_{\beta\alpha}
G_{i\ell}M_k = \frac{1}{n}\sum_{k=1}^n M_k\left[\sum_{\ell=1}^n
G_{\ell j}E_{\beta\alpha}
G_{i\ell}\right]M_k=\frac{1}{n}\sum_{k=1}^n M_k\Delta M_k,\]
this gives
\[\Delta=\sum_k^{(ij)} \E_k[\Delta^k]+\Oprec(n^{-1/2})
=\frac{1}{n}\sum_{k=1}^n M_k\Delta M_k+\Oprec(n^{-1/2}).\]
Thus $\cL_1(\Delta) \prec n^{-1/2}$ where $\cL_1$ is the linear
operator of Lemma \ref{lemma:linearmap1} (in the current setting with $\m_0=M_B$
and $(\r_0)_{ii}=M_i$). By the quantitative invertibility
of $\cL_1$ shown in Lemma \ref{lemma:linearmap1}, this implies the claim
(\ref{eq:Deltabound}),
completing the proof of Theorem \ref{rmt:main theorem}.
\end{proof} 

\section{Analysis of least-squares problem}\label{sec:optimization}

In this section, we prove Theorem \ref{optimization: main theorem},
Corollary \ref{cor:optimization}, and Proposition \ref{prop:computation}.
Recall the optimization objective from Section \ref{sec:model},
\[f(X)=\frac{1}{2}\|XA+BX\|_F^2+\frac{1}{2}\sum_{i,j=1}^n \xi_i\theta_j
x_{ij}^2\]
and its minimizer under a linear constraint,
\begin{equation}\label{eq:optimization objective}
    \widehat{X}=\argmin_{X \in \R^{n \times n}} f(X)
\text{ subject to } \frac{1}{n}\v^* X\u=1.
\end{equation}

\begin{proof}[Proof of Theorem \ref{optimization: main theorem}]
    Consider the following vectorization of (\ref{eq:optimization objective})
(where we use the convention of vectorization by column, i.e.\
$\bx=\sum_k\e_k \otimes X\e_k$)
    \begin{align*}
f(\bx)&=\frac{1}{2}\bx^*[(A \otimes I + I \otimes B)^2 + (\Theta \otimes
\Xi)]\bx,\\
        \hat{\bx}&=\argmin_{\bx \in \R^{n^2}} f(\bx)
        \text{ subject to } n^{-1}(\u \otimes \v)^*\bx=1.
    \end{align*}
Denote 
\begin{align*}
P&=[(A \otimes I + I \otimes B)^2 + (\Theta \otimes \Xi)]^{-1} \in \R^{n \times
n} \otimes \R^{n \times n},\\
\mathsf{p}&=[(\a \otimes 1 + 1 \otimes \b)^2 + (\Theta \otimes \Xi)]^{-1} \in
\cA \otimes \cA,\\
P_0&=(\tau^\cD \otimes \tau^\cD)[\mathsf{p}].
\end{align*}
Simple calculus yields the explicit forms for $\hat{\bx}$ and $f(\hat{\bx})$ as 
\begin{equation}\label{eq:Xexplicit}
        \hat{\bx}=\frac{1}{n^{-2}(\u \otimes \v)^*P(\u \otimes \v)}
\cdot n^{-1}P(\u \otimes \v), \qquad
        f(\hat{\bx})=\frac{1}{2}\frac{1}{n^{-2}(\u \otimes \v)^*P(\u \otimes
\v)}.
\end{equation}

Consider the linearization of $P$ given by
    \begin{align*}
        \widetilde{P}&=\begin{bmatrix}
            -i\,\Theta \otimes \Xi & A \otimes I + I \otimes B\\
            A \otimes I + I \otimes B & -i\,I \otimes I
        \end{bmatrix}^{-1}\\
        &=\left(\begin{bmatrix}
            0 & 1\\
            1 & 0
        \end{bmatrix} \otimes (A \otimes I + I \otimes B) -i
        \begin{bmatrix}
            1 & 0\\
            0 & 0
        \end{bmatrix} \otimes \Theta \otimes \Xi - i\begin{bmatrix}
            0 & 0\\
            0 & 1
        \end{bmatrix}\otimes I \otimes I\right)^{-1}.
    \end{align*}
Denote $\e_1=(1,0) \in \C^2$. Then by Schur's complement,
    \begin{equation}\label{eq:schurcomplement}
-i\widetilde{P}=\begin{bmatrix}
        \Theta \otimes \Xi & i(A \otimes I + I \otimes B)\\
        i(A \otimes I + I \otimes B) & I \otimes I
    \end{bmatrix}^{-1}=\begin{bmatrix}
        [(A \otimes I + I \otimes B)^2+\Theta \otimes \Xi]^{-1} & *\\
        * & *
    \end{bmatrix}
\end{equation}
so that $P=(\e_1 \otimes I \otimes I)^*[-i\widetilde P](\e_1 \otimes I \otimes I)$. Defining also
\begin{align}
\tilde\p&=\left(\begin{bmatrix}
            0 & 1\\
            1 & 0
        \end{bmatrix} \otimes (\a \otimes 1 + 1 \otimes \b) -i\begin{bmatrix}
            1 & 0\\
            0 & 0
        \end{bmatrix} \otimes \Theta \otimes \Xi - i\begin{bmatrix}
            0 & 0\\
            0 & 1
        \end{bmatrix}\otimes 1 \otimes 1\right)^{-1},\label{eq:ptilde}\\
\widetilde{P}_0&=(I_{2 \times 2} \otimes \tau^\cD \otimes
\tau^\cD)[\tilde\p],\notag
\end{align}
we have similarly $\p=(\e_1 \otimes 1 \otimes 1)^*[-i\tilde \p](\e_1 \otimes
1 \otimes 1)$. Then, by (\ref{eq:linearmapcomp}) applied with
$(T,T')=(\e_1^*[\cdot]\e_1,\tau^\cD \otimes \tau^\cD)$, we have also
\[P_0=(\e_1 \otimes 1 \otimes 1)^*[-i\widetilde P_0](\e_1 \otimes 1 \otimes
1).\]

By (\ref{isotropic estimates}) of
Theorem \ref{aux:main theorem}, uniformly over $\u,\v,\u',\v' \in \R^n$ with
$\|\u\|_2,\|\v\|_2,\|\u'\|_2,\|\v'\|_2 \leq \upsilon\sqrt{n}$,
    \begin{align}
        \frac{1}{n^2}(\u' \otimes \v')^*P(\u \otimes \v)&={-i}\left(\e_1
\otimes\frac{\u'}{\sqrt{n}} \otimes
\frac{\v'}{\sqrt{n}}\right)^*\widetilde{P}\left(\e_1 \otimes \frac{\u}{\sqrt{n}}
\otimes \frac{\v}{\sqrt{n}}\right)\notag\\
        &={-i}\left(\e_1 \otimes\frac{\u'}{\sqrt{n}} \otimes
\frac{\v'}{\sqrt{n}}\right)^*\widetilde{P}_0\left(\e_1 \otimes
\frac{\u}{\sqrt{n}} \otimes \frac{\v}{\sqrt{n}}\right)+\Oprec(n^{-1/2})\notag\\
        &=\frac{1}{n^2}(\u' \otimes \v')^*P_0(\u \otimes \v)+\Oprec(n^{-1/2}).
\label{eq:Pquadform}
    \end{align}
Here, in the second line, we have applied (\ref{isotropic estimates}) of
Theorem \ref{aux:main theorem} twice as in the proof of Theorem \ref{rmt:main
theorem} in Section \ref{sec:resolvent}, first to the second tensor factor
conditional on $B$ and the event $\cE=\{\|B\|_\op \leq 3\}$ with
$\cX=\C^{2 \times 2} \otimes \C^{n \times n}$ (the product of
first and third factors),
\[H=A,\quad \x=\begin{bmatrix} 0 & 1 \\ 1 & 0 \end{bmatrix} \otimes I,
\quad \z=i\begin{bmatrix} 1 & 0 \\ 0 & 0 \end{bmatrix}
\otimes \Theta \otimes \Xi+i\begin{bmatrix} 0 & 0 \\ 0 & 1 \end{bmatrix}
\otimes I \otimes I-\begin{bmatrix} 1 & 0 \\ 0 & 1 \end{bmatrix} \otimes I
\otimes B,\]
and then to the third tensor factor with $\cX=\C^{2 \times 2} \otimes \cA$
(the product of the first and second factors),
\[H=B, \quad \x=\begin{bmatrix} 0 & 1 \\ 1 & 0 \end{bmatrix} \otimes 1,
\quad \z=i\begin{bmatrix} 1 & 0 \\ 0 & 0 \end{bmatrix}
\otimes \Theta \otimes \Xi+i\begin{bmatrix} 0 & 0 \\ 0 & 1 \end{bmatrix}
\otimes 1 \otimes I-\begin{bmatrix} 1 & 0 \\ 0 & 1 \end{bmatrix} \otimes \a
\otimes I.\]

Since $\p$ is a positive operator, satisfying
$\p \geq \|\p^{-1}\|_\op^{-1}(1_\cA \otimes 1_\cA)$, it follows from positivity
of $\tau^\cD \otimes \tau^\cD$ (Lemma \ref{conditional expectation}) and the
bounds $\|\a\|_\op=\|\b\|_\op=2$ and $\|\Theta\|_\op,\|\Xi\|_\op \leq \upsilon$
that
    \[P_0=\tau^\cD \otimes \tau^\cD[\p]\geq \|\p^{-1}\|_{\op}^{-1}
\geq (16+\upsilon^2)^{-1}.\] 
    Therefore, for $\u=\u'$ and $\v=\v'$ satisfying $\|\u\|_2,\|\v\|_2 \geq
\upsilon^{-1}\sqrt{n}$, we have the constant lower bound
    \[\frac{1}{n^2}(\u \otimes \v)^*P_0(\u \otimes \v)
    \geq (16+\upsilon^2)^{-1}\frac{1}{n^2}\|\u\|_2^2\|\v\|_2^2
    \geq (16\upsilon^4+\upsilon^6)^{-1},\]
so we may apply the approximation (\ref{eq:Pquadform}) to both the numerator
and denominator of (\ref{eq:Xexplicit}) to get
\begin{align*}
f(\hat\bx)&=\frac{1}{2}
\frac{1}{n^{-2}(\u \otimes \v)^* P_0(\u \otimes \v)}+\Oprec(n^{-1/2}),\\
\frac{1}{n}(\u' \otimes \v')^* \hat\bx
&=\frac{n^{-2}(\u' \otimes \v')^* P_0(\u \otimes \v)}
{n^{-2}(\u \otimes \v)^* P_0(\u \otimes \v)}+\Oprec(n^{-1/2}).
\end{align*}
Finally, since $P_0 \in \cD \otimes \cD \subset \C^{n \times n} \otimes \C^{n
\times n}$ is a diagonal matrix, writing $\su=\diag(\u)$ etc.\ we have
\begin{align*}
n^{-2}(\u' \otimes \v')^* P_0(\u \otimes \v)
&=n^{-2}(\Tr \otimes \Tr)[(\su' \otimes \sv')^* P_0(\su \otimes \sv)]\\
&=(n^{-1}\Tr \otimes n^{-1}\Tr)(\tau^\cD \otimes \tau^\cD)
[(\su' \otimes \sv')^* \p(\su \otimes \sv)]\\
&=(\tau \otimes \tau)[(\su' \otimes \sv')^* \p(\su \otimes \sv)].
\end{align*}
Applying this identity to both the numerators and denominators above concludes
the proof.
\end{proof}

\begin{proof}[Proof of Corollary \ref{cor:optimization}]
Under the given assumptions, for any fixed
non-commutative polynomial $p$, we have
\[\lim_{n \to \infty} (\tau \otimes \tau)
\Big(p(\a \otimes 1,1 \otimes \b,\Theta \otimes \Xi,\su \otimes \sv,\su' \otimes
\sv')\Big)
=(\tau \otimes \tau)
\Big(p(\a \otimes 1,1 \otimes \b,\theta \otimes \xi,\sU \otimes \sU,\sU'
\otimes \sU')\Big)\]
where $\Theta,\su,\su',\Xi,\sv,\sv' \in \cD \subset \cA$ on the left are real
diagonal matrices, and $\theta,\sU,\sU',\xi,\sV,\sV' \in \cA$ on the right are 
commuting, self-adjoint limiting operators, free of $(\a,\b)$ and
for which $(\theta,\sU,\sU')$ and $(\xi,\sV,\sV')$
have joint laws under $\tau$ given by $\cP$ and $\cQ$, respectively.
The assumptions $\|\Theta\|_\op,\|\Xi\|_\op \leq \upsilon$ and $\Theta,\Xi \geq
\delta$ imply that
the spectrum of $(\a \otimes 1+1 \otimes \b)^2+\Theta \otimes \Xi$
is contained in $[\delta^2,16+\upsilon^2]$. The inverse function
$x \mapsto x^{-1}$ may be approximated uniformly by polynomials on this
interval, so the above convergence implies
\begin{align*}
&\lim_{n \to \infty}
(\tau \otimes \tau)[(\su' \otimes \sv')[(\a \otimes 1+1 \otimes
\b)^2+\Theta \otimes \Xi]^{-1}(\su \otimes \sv)]\\
&=(\tau \otimes \tau)[(\sU' \otimes \sV')[(\a \otimes 1+1 \otimes
\b)^2+\theta \otimes \xi]^{-1}(\sU \otimes \sV)]
\end{align*}
where this limit depends only on the joint laws $\cP,\cQ$.
Defining this limit quantity as $T'(\cP,\cQ)$, and the analogous limit with
$\su=\su'$ and $\sv=\sv'$ as $T(\cP,\cQ)$,
the corollary then follows from Theorem \ref{optimization: main theorem}.
\end{proof}

\begin{proof}[Proof of Proposition \ref{prop:computation}]
Following the above proof of Corollary \ref{cor:optimization}, 
we write $\theta,\sU,\sU',\xi,\sV,\sV' \in \cA$ for the limiting self-adjoint
operators, which are free of $(\a,\b)$ and such
that $(\theta,\sU,\sU')$ and $(\xi,\sV,\sV')$
have joint laws under $\tau$ given by $\cP$ and $\cQ$.
Thus $\|\sU\|_\infty \equiv \|\sU\|_\op$ in the statement of
Proposition \ref{prop:computation}. In this proof, we denote
by $\cD \subset \cA$ the von Neumann subalgebra generated by 
$\theta,\sU,\sU',\xi,\sV,\sV'$.

We set $\eta=\min\{\sqrt{x_ax_b}:x_a \in \supp(\theta),\,x_b \in \supp(\xi)\}$
and define a limiting linearized operator $\tilde \p$ analogous to
(\ref{eq:ptilde}),
\[\tilde \p=\begin{bmatrix} -i\eta^{-1}\,\theta \otimes \xi &
\a \otimes 1+1 \otimes \b \\ \a \otimes 1+1 \otimes \b & -i\eta\,1 \otimes 1
\end{bmatrix}^{-1}.\]
It is direct to check as in (\ref{eq:schurcomplement}) that we have
\begin{equation}\label{eq:schur2}
[(\a \otimes 1+1 \otimes \b)^2+\theta \otimes
\xi]^{-1}=(\e_1 \otimes 1 \otimes 1)^*[-i\eta^{-1}\tilde \p](\e_1 \otimes 1
\otimes 1).
\end{equation}
Introducing the shorthands
\[\d_a=\theta^{-1/2}, \qquad \d_b=\xi^{-1/2},\]
\[\tilde \d=\begin{bmatrix} \eta^{1/2}\d_a \otimes \d_b & 0 \\ 0 & \eta^{-1/2}
1 \otimes 1 \end{bmatrix},
\qquad \tilde \a=\begin{bmatrix} 0 & \d_a\a \otimes \d_b \\ 0 
& 0 \end{bmatrix}, \qquad \tilde \b=\begin{bmatrix} 0 & \d_a \otimes \d_b\b \\ 
0 & 0 \end{bmatrix}\]
and fixing a real argument $z>1$, we write this as
\begin{align*}
\tilde \p&=\left(\tilde \d^{-1}
\begin{bmatrix} -i\,1 \otimes 1 & \d_a\a \otimes \d_b+\d_a \otimes \d_b\b \\
\a\d_a \otimes \d_b+\d_a \otimes \b\d_b & -i\,1 \otimes 1
\end{bmatrix}\tilde\d^{-1}\right)^{-1}\\
&=\tilde\d\Big(\underbrace{\tilde\a+\tilde\a^*+\tilde\b+\tilde\b^*}_{=\tilde \q}
+\,i(z-1)-iz\Big)^{-1}\tilde\d.
\end{align*}
Note that
\[\|\d_a\a \otimes \d_b+\d_a \otimes \d_b\b\|_\op \leq
\|\d_a\|_\op\|\d_b\|_\op(\|\a\|_\op+\|\b\|_\op) \leq 4\eta^{-1},\]
so $\tilde\q=\tilde\a+\tilde\a^*+\tilde\b+\tilde\b^*$ is self-adjoint with
spectrum contained in $[-4\eta^{-1},4\eta^{-1}]$. Hence
\begin{equation}\label{eq:tildeqbound}
\|\tilde \q+i(z-1)\|_\op=\max\{|\lambda+i(z-1)|:\lambda \in
\spec(\tilde\q)\} \leq \sqrt{(z-1)^2+16\eta^{-2}}.
\end{equation}
Applying iteratively $(\tilde \q+i(z-1)-iz)^{-1}
={-}(iz)^{-1}+(iz)^{-1}[\tilde \q+i(z-1)](\tilde \q+i(z-1)-iz)^{-1}$ to
write a series expansion of $(\tilde \q+i(z-1)-iz)^{-1}$, we have
\begin{align*}
\tilde \p&={-}\sum_{k=0}^{M-1} (iz)^{-(k+1)}\tilde\d[\tilde \q+i(z-1)]^k
\tilde\d+\r_M\\
&={-}\sum_{k=0}^{M-1} (iz)^{-(k+1)} \sum_{m=0}^k
\binom{k}{m} [i(z-1)]^{k-m}\tilde\d\tilde\q^m\tilde\d+\r_M\\
&=i\sum_{m=0}^{M-1}\underbrace{\left(\frac{1}{i^m z(z-1)^m}
\sum_{k=m}^{M-1}
\binom{k}{m}\left(\frac{z-1}{z}\right)^k\right)}_{=C_m(z)}
\tilde\d\tilde\q^m\tilde\d+\r_M
\end{align*}
with remainder
\begin{align*}
\|\r_M\|_\op&=
\left\|z^{-M}\tilde\d[\tilde\q+i(z-1)]^M(\tilde\q-i)^{-1}\tilde\d\right\|_\op
\leq \eta^{-1} \left(\frac{\sqrt{(z-1)^2+16\eta^{-2}}}{z}\right)^M.
\end{align*}
Here, we have applied (\ref{eq:tildeqbound}), $\|(\tilde \q-i)^{-1}\|_\op
\leq 1$, and $\|\tilde \d\|_\op \leq \eta^{-1/2}$ as follows from its
definition. Then, recalling the Schur complement identity (\ref{eq:schur2})
and applying also norm contractivity of $\tau \otimes \tau$
(Lemma \ref{conditional expectation}), we obtain
\begin{align}
&(\tau \otimes \tau)[(\sU' \otimes \sV')[(\a \otimes 1+1 \otimes
\b)^2+\theta \otimes \xi]^{-1}(\sU \otimes \sV)]\notag\\
&=\sum_{m=0}^{M-1} C_m(z) (\tau \otimes \tau)
\Big[(\e_1 \otimes \sU' \otimes \sV')^* [\eta^{-1}\tilde\d\tilde\q^m\tilde\d]
(\e_1 \otimes \sU \otimes \sV)\Big]+r_M
\label{eq:computationexpansion}
\end{align}
where $r_M$ is an error satisfying (\ref{eq:computationremainder}).
We remark that the left side is real because all elements are self-adjoint,
while each summand on the right
is real for even $m$ and pure imaginary for odd $m$ by the definition of
$C_m(z)$. Thus, taking real parts,
this identity also holds with the summation restricted to even $m$.

We now analyze the summand for each even $m \in \{0,\ldots,M-1\}$.
Let $\cW_m$ be the
set of words in the letters $(\A,\A^*,\B,\B^*)$ starting with a letter in
$\{\A,\B\}$ and alternating between a letter $\{\A,\B\}$ and a letter
$\{\A^*,\B^*\}$, understood as non-commutative monic monomials of four variables.
Then, applying the definitions of $\tilde\d,\tilde\a,\tilde\b$, we have
\begin{align*}
&(\e_1 \otimes 1 \otimes 1)^* \eta^{-1}\tilde\d
(\tilde \a+\tilde\a^*+\tilde \b+\tilde\b^*)^m\tilde\d
(\e_1 \otimes 1 \otimes 1)^*\\
&=\sum_{w \in \cW_m} \d_a w(\d_a\a,\a\d_a,\d_a,\d_a)\d_a
\otimes \d_b w(\d_b,\d_b,\d_b\b,\b\d_b)\d_b \in \cA \otimes \cA.
\end{align*}
Thus
\begin{align*}
&(\tau \otimes \tau)\Big[(\e_1 \otimes \sU' \otimes \sV')^*
[\eta^{-1}\tilde\d\tilde\q^m\tilde\d](\e_1 \otimes \sU \otimes \sV)\Big]\\
&=\sum_{w \in \cW_m} \tau\Big(\sU'\d_a
w(\d_a\a,\a\d_a,\d_a,\d_a)\d_a\sU\Big) \cdot \tau\Big(\sV'\d_b
w(\d_b,\d_b,\d_b\b,\b\d_b)\d_b \sV\Big)\\
&=\sum_{w \in \cW_m} \tau\Big(\d_a\sU\sU'\d_a
w(\d_a\a,\a\d_a,\d_a,\d_a)\Big) \cdot \tau\Big(\d_b \sV\sV'\d_b
w(\d_b,\d_b,\d_b\b,\b\d_b)\Big).
\end{align*}

For any $\d_1,\ldots,\d_{k+1} \in \cD$, by the free moment-cumulant relations
\cite[Proposition 11.4]{NicaSpeicher2006}, we have
\[\tau(\d_1\a\d_2\a\ldots \d_k\a\d_{k+1})
=\sum_{\pi \in \NC(2k+1)} \kappa_\pi(\d_1,\a,\d_2,\a,\ldots,\d_k,\a,\d_{k+1})\]
where $\NC(2k+1)$ is the set of non-crossing partitions of $(1,2,\ldots,2k+1)$
and $\kappa_\pi$ is the free cumulant associated to each $\pi \in \NC(2k+1)$.
Since $\a$ is free of $\cD$ and has $2^\text{nd}$ free cumulant equal to 1
and remaining free cumulants 0 \cite[Example 11.21]{NicaSpeicher2006}, we have
$\kappa_\pi(\d_1,\a,\d_2,\a,\ldots,\d_k,\a,\d_{k+1})=0$ unless the blocks of
$\pi$ containing $\a$ constitute a non-crossing pairing and are
disjoint from those containing
$(\d_1,\ldots,\d_{k+1})$. Let $\NC_2(k)$ denote the set of all
non-crossing pairings of $\{2,4,6,\ldots,2k\}$ corresponding to the locations of
the $\a$'s, and for each $\rho \in \NC_2(k)$, let $K(\rho) \in \NC(k+1)$ be
its complement in $\{1,2,\ldots,2k+1\}$, i.e.\ the
coarsest non-crossing partition of the remaining elements
$\{1,3,5,\ldots,2k+1\}$ for which
$\rho \cup K(\rho)$ forms a non-crossing partition of $\{1,2,\ldots,2k+1\}$.
Then each $\pi \in \NC(2k+1)$ for which
$\kappa_\pi(\d_1,\a,\d_2,\a,\ldots,\d_k,\a,\d_{k+1}) \neq 0$
is the union of some $\rho \in \NC_2(k)$ and some
$\bar\rho \leq K(\rho)$ that refines $K(\rho)$, so we have
\begin{align*}
\tau(\d_1\a\d_2\a\ldots \d_k\a\d_{k+1})
&=\sum_{\rho \in \NC_2(k)} \sum_{\bar\rho \in \NC(k+1):\bar\rho\leq K(\rho)}
\kappa_{\rho}(\a,\ldots,\a)\kappa_{\bar\rho} (\d_1,\d_2,\ldots,\d_{k+1})\\
&=\sum_{\rho \in \NC_2(k)} \sum_{\bar\rho \leq K(\rho)}
\kappa_{\bar\rho} (\d_1,\d_2,\ldots,\d_{k+1})
=\sum_{\rho \in \NC_2(k)} \prod_{S \in K(\rho)} \tau\left(\prod_{i \in S}
\d_i\right),
\end{align*}
the last equality applying that $\{\bar\rho:\bar\rho \leq K(\rho)\}$ is a
product of non-crossing partitions of the blocks of $K(\rho)$ and applying the
free moment-cumulant relation over each block of $K(\rho)$.

Applying this above,
corresponding to each word $w \in \cW_m$, let $\NC_{w,2}(\A)$ be the set of all
non-crossing pairings of the letters $\{A,A^*\}$ of $w$ (not necessarily
pairing $A$ with $A^*$), and let $\NC_{w,2}(\B)$ be those of the letters
$\{B,B^*\}$ of $w$. Then
\[\tau\Big(\d_a\sU\sU'\d_a w(\d_a\a,\a\d_a,\d_a,\d_a)\Big)
\cdot \tau\Big(\d_b\sV\sV'\d_b w(\d_b,\d_b,\d_b\b,\b\d_b)\Big)\\
=\sum_{\rho_a \in \NC_{w,2}(\A)}
\sum_{\rho_b \in \NC_{w,2}(\B)} \val(\rho_a,\rho_b)\]
where $\val(\cdot)$ is precisely the quantity defined in
Section \ref{sec:computation}.
Finally, we observe that summing first over
words $w \in \cW_m$ of $(A,A^*,B,B^*)$ and then over pairings $\rho_a \in
\NC_{w,2}(\A)$ and $\rho_b \in \NC_{w,2}(\B)$ is equivalent to
summing over all disjoint non-crossing pairings $(\rho_a,\rho_b) \in
\NC_{2,2}(m)$, and then identifying $w$ as the word with letters
$\{\A,\A^*\}$ for elements of $\rho_a$ and $\{\B,\B^*\}$ for elements of
$\rho_b$ and that alternates between $\{A,B\}$ and $\{A^*,B^*\}$. Thus
\[(\tau \otimes \tau)\Big[(\e_1 \otimes \sU' \otimes \sV')^*
[\eta^{-1}\tilde\d\tilde\q^m\tilde\d](\e_1 \otimes \sU \otimes \sV)\Big]
=\sum_{(\rho_a,\rho_b) \in \NC_{2,2}(m)} \val(\rho_a,\rho_b).\]
A direct calculation shows that
this identity holds also for $m=0$, upon defining
$\NC_{2,2}(0)$ to have the single pair $(\emptyset,\emptyset)$
with its value defined as in (\ref{eq:emptyval}).
Applying this back to (\ref{eq:computationexpansion}) 
and using $i^m=(-1)^{m/2}$ in $C_m(z)$ concludes the proof.
\end{proof}

\appendix

\section{Contour integral representation in the case $\Theta=\Xi=\eta I$}\label{appendix:contour}

The optimization problem (\ref{eq:QPintro}) with $\Theta=\Xi=\eta I$ 
and $\eta \sim 1/(\log n)^C$ was studied previously in \cite{fan2023spectral}.
In this setting, its solution is given explicitly by
\[\widehat{X}=\frac{1}{n^{-2}(\u \otimes \v)^*P(\u \otimes \v)} \cdot n^{-1}
P(\u \otimes \v)
\quad \text{ where } \quad
P=[(A \otimes I+I \otimes B)^2+\eta^2 I \otimes I]^{-1}.\]
Applying the linearization
\[[(A \otimes I+I \otimes B)^2+\eta^2 I \otimes I]^{-1}
=\frac{1}{\eta}\Im [A \otimes I+I \otimes B-i\eta I \otimes I]^{-1},\]
the analyses of \cite{fan2023spectral}
rested on a deterministic approximation for the resolvent
\[R(z)=(Q-z\,I \otimes I)^{-1}, \qquad Q=A \otimes I+I \otimes B\]
at spectral scales $\Im z \sim 1/(\log n)^C$ decaying slowly with $n$.

This setting is special because the matrices $A \otimes I$ and $I \otimes B$
constituting $Q$
commute. In this setting, writing the spectral decompositions $A=\sum_{j=1}^n
\lambda_j\u_j\u_j^*$ and $B=\sum_{k=1}^n \mu_k\v_k\v_k^*$, the spectral
decomposition of $Q$ is explicit and given by
\[Q=\sum_{j,k=1}^n (\lambda_j+\mu_k)(\u_j \otimes \v_k)(\u_j \otimes \v_k)^*.\]
In particular, the limit eigenvalue distribution of $Q$ is the (classical)
convolution of the semicircle law with itself. Furthermore, defining a contour
$\Gamma$ enclosing $\{\lambda_1,\ldots,\lambda_n\}$ and such that
$|\Im w| \leq (\Im z)/2$ for all $w \in \Gamma$,
by the Cauchy integral formula applied to
$f_k(w)=(w+\mu_k-z)^{-1}$ (which is analytic inside $\Gamma$)
the resolvent of $Q$ has the explicit contour integral representation
\begin{align*}
R(z)&=\sum_{j,k=1}^n \frac{1}{\lambda_j+\mu_k-z}\u_j\u_j^* \otimes \v_k\v_k^*\\
&=\frac{1}{2\pi i}\oint_\Gamma \frac{1}{w-\lambda_j}\frac{1}{w+\mu_k-z}
\u_j\u_j^* \otimes \v_k\v_k^*\,dw\\
&={-}\frac{1}{2\pi i} \oint_\Gamma R_A(w) \otimes R_B(z-w)\,dw
\end{align*}
where $R_A(z)=(A-zI)^{-1}$ and $R_B(z)=(B-zI)^{-1}$. From this representation,
resolvent estimates of the form in Theorem \ref{rmt:main theorem} may be 
deduced from known local laws for the resolvents of the Wigner matrices $A$ and
$B$, see e.g.\ \cite{ErdosYauYin2012Rigidity,ErdosEtAl2013Spectral},
and analysis of this commutative case suggests that the estimates in
Theorem \ref{rmt:main theorem} are also optimal (for fixed $z \in \C^+$).

We emphasize that this type of analysis and contour integral representation
does not extend to models of the form
$Q=A \otimes I+I \otimes B+\Theta \otimes \Xi$ when $\Theta \otimes \Xi$ does
not commute with either $A \otimes I$ or $I \otimes B$, which is the focus of our
current work.

\section{Operator concentration inequalities}\label{appendix:concentration}

In this section, we prove Lemma \ref{lemma:concentration} using the following
version of the non-commutative Rosenthal inequality of
\cite{JungeXu2008}. We recall that $\cX$ is a von Neumann algebra with
faithful, normal, tracial state $\phi$, and $L^p(\cX)$ is its associated
non-commutative $L^p$ space (c.f.\ Appendix \ref{appendix:background}) with norm
$\|\x\|_p=\phi((\x^*\x)^{p/2})^{1/p}$.

\begin{lem}[\cite{JungeXu2008}, Theorem 2.1]\label{lemma:rosenthal}
Let $\cY \subset \cX$ be a von Neumann subalgebra with $\phi$-invariant
conditional expectation (c.f.\ Lemma \ref{conditional expectation})
$\phi^\cY:\cX \to \cY$. Suppose $\x_1,\ldots,\x_n
\in L^p(\cX)$ satisfy $\phi^\cY(\x_i)=0$, and are independent over $\cY$
in the sense that for each $i$, each $\x$ in the von Neumann
subalgebra generated by $\x_i$, and each $\x'$ in the von
Neumann subalgebra generated by $\{\x_j:j \neq i\}$, we have
$\phi^\cY(\x\x')=\phi^\cY(\x)\phi^\cY(\x')$.

Then for any $p \in [2,\infty)$ and a universal constant $C>0$,
\[\bigg\|\sum_{i=1}^n \x_i\bigg\|_p
\leq Cp\max\bigg\{\bigg\|\bigg(\sum_{i=1}^n
\phi^\cY(\x_i\x_i^*)\bigg)^{1/2}\bigg\|_p,\;
\bigg\|\bigg(\sum_{i=1}^n
\phi^\cY(\x_i^*\x_i)\bigg)^{1/2}\bigg\|_p,\;
\bigg(\sum_{i=1}^n \|\x_i\|_p^p\bigg)^{1/p}\bigg\}.\]
\end{lem}

\begin{lem}[Decoupling]\label{lemma:decoupling}
Let $(\alpha_i)_{i=1}^n$ be a sequence of independent $\C$-valued random
variables, let $(\alpha_i^{\prime})_{i=1}^n$ be an independent
copy of $(\alpha_i)_{i=1}^n$, and let
$(\y_{ij}:i,j=1,\ldots,n)$ be elements of a Banach space with norm
$\|\cdot\|$. Then for a universal constant $C>0$,
        \begin{align*}
            \E\bigg[\bigg\|\sum_{i \neq j}\alpha_i\alpha_j\y_{ij}\bigg\|^p\bigg]
\leq C^{p+1} \E\bigg[\bigg\|\sum_{i \neq
j}\alpha_i\alpha_i^{\prime}\y_{ij}\bigg\|^p\bigg]
        \end{align*}
    \end{lem}
    \begin{proof}
        By \cite[Theorem 1]{DelapenaMontgomerySmith1994}, for a universal
constant $C>0$,
        \begin{align*}
            \frac{1}{p}\,\E\bigg[\bigg\|\sum_{i \neq
j}\alpha_i\alpha_j\y_{ij}\bigg\|^p\bigg] &= \int_0^\infty
t^{p-1}\P\bigg[\bigg\|\sum_{i \neq j}\alpha_i\alpha_j\y_{ij}\bigg\|\geq t\bigg]\;dt\\
            &\leq \int_0^\infty Ct^{p-1}\P\bigg[C\bigg\|\sum_{i \neq
j}\alpha_i\alpha_i^{\prime}\y_{ij}\bigg\|\geq t\bigg]\;dt =
\frac{C^{p+1}}{p}\E\bigg[\bigg\|\sum_{i \neq
j}\alpha_i\alpha_i^{\prime}\y_{ij}\bigg\|^p\bigg].
        \end{align*}
    \end{proof}

\begin{proof}[Proof of Lemma \ref{lemma:concentration}]
Throughout, $C_p,C_p',C_p''$ denote $p$-dependent constants
that may change from instance to instance.

For (a), fix any $p \geq 2$. We apply Lemma \ref{lemma:rosenthal} in the
setting of \cite[Example 1.3]{JungeXu2008}:
Let $L^\infty(\Omega)$ be the von Neumann algebra of bounded scalar
random variables over the underlying probability space
$(\Omega,\mathscr{F},\P)$, and consider $\cM=L^\infty(\Omega) \otimes \cX$
equipped with the state
$\E \circ \phi$. Then $\alpha_i\x_i \in L^p(\cM,\E \circ \phi)$,
$\E:\cM \to \cX$ coincides with the conditional
expectation onto the subalgebra $\cX \subset \cM$, and
$\{\alpha_i\x_i\}_{i=1}^n$ are independent over $\cX$ in the sense of
Lemma \ref{lemma:rosenthal}, so Lemma \ref{lemma:rosenthal} shows
\[\E\bigg[\bigg\|\sum_{i=1}^n \alpha_i\x_i\bigg\|_p^p\bigg] \leq C_p
\max\bigg\{\bigg\|\bigg(\sum_{i=1}^n \E[|\alpha_i|^2]\x_i\x_i^*\bigg)^{1/2}\bigg\|_p^p,\;
\bigg\|\bigg(\sum_{i=1}^n  \E[|\alpha_i|^2]\x_i^*\x_i\bigg)^{1/2}\bigg\|_p^p,\;
\sum_{i=1}^n  \E[|\alpha_i|^p]\|\x_i\|_p^p\bigg\}.\]
Applying the bounds $\E[|\alpha_i|^2],
\E[|\alpha_i|^p] \leq C_p$, operator monotonicity of the square-root
$0 \leq \x \leq \y \Rightarrow \x^{1/2} \leq \y^{1/2}$, and monotonicity of the
$L^p$-norm on the positive cone (Lemma \ref{Holder's inequality}), this implies
\begin{equation}\label{eq:rosenthalbound}
\E\bigg[\bigg\|\sum_{i=1}^n \alpha_i\x_i\bigg\|_p^p\bigg] \leq C_p'
\max\bigg\{\bigg\|\bigg(\sum_{i=1}^n \x_i\x_i^*\bigg)^{1/2}\bigg\|_p^p,\;
\bigg\|\bigg(\sum_{i=1}^n \x_i^*\x_i\bigg)^{1/2}\bigg\|_p^p,\;
\sum_{i=1}^n \|\x_i\|_p^p\bigg\}.
\end{equation}
Here, for $p \geq 2$,
the third term is bounded by the second by the following argument (see
also \cite[Eq.\ (2.4)]{JungeXu2008}): Consider
$\hat{\x}=\sum_{i=1}^n E_{i1} \otimes \x_i \in \C^{n \times n} \otimes \cX$
equipped with the trace $n^{-1}\Tr \otimes \phi$,
and the linear map $T:\C^{n \times n} \otimes \cX \to \C^{n \times n} \otimes
\cX$ defined by
\[
    T(E_{ij} \otimes \x)=\begin{cases}
        E_{i,j+i-1} \otimes \x & \text{if $j=1$}\\
        0 & \text{otherwise}
      \end{cases}
\]
where $j+i-1$ is
interpreted modulo $n$. Thus $\hat\x$ has $\cX$-valued entries
$\x_1,\ldots,\x_n$ along the first column, and $T(\hat\x)$ has these entries
instead along the main diagonal. We have
\begin{align*}
\|\hat\x\|_p^p&=\|(\hat\x^*\hat\x)^{1/2}\|_p^p=\frac{1}{n}
\bigg\|\bigg(\sum_{i=1}^n \x_i^*\x_i\bigg)^{1/2}\bigg\|_p^p,\\
\|T(\hat\x)\|_p^p&=\|(T(\hat\x)^*T(\hat\x))^{1/2}\|_p^p
=\frac{1}{n}\sum_{i=1}^n \|(\x_i^*\x_i)^{1/2}\|_p^p
=\frac{1}{n}\sum_{i=1}^n \|\x_i\|_p^p.
\end{align*}
For $p=2$, this shows $\|T(\hat\x)\|_2=\|\hat\x\|_2$. For $p=\infty$, we have
\[\|T(\hat\x)\|_\op=\max_i \|\x_i\|_\op
=\max_i \|(\x_i^*\x_i)^{1/2}\|_\op \leq \bigg\|\bigg(\sum_{i=1}^n
\x_i^*\x_i\bigg)^{1/2}\bigg\|_\op
=\|\hat\x\|_\op\]
by operator monotonicity of the square-root and monotonicity of the operator
norm on the positive cone. Then
$\|T(\hat\x)\|_p \leq \|\hat\x\|_p$ for all $p \in [2,\infty]$
by the Riesz-Thorin interpolation (Lemma \ref{lemma:rieszthorin}). Thus, the
third term of (\ref{eq:rosenthalbound}) is at most the first, yielding the
claim of part (a).

Part (b) follows from a two-fold application of part (a): Let $\x_i = \sum_j
\beta_j\y_{ij}$, so $\sum_i \alpha_i\x_i=\sum_{i,j}\alpha_i\beta_j \y_{ij}$.
Then, applying part (a) conditional on $(\beta_i)_{i=1}^n$, we have
    \begin{align}\label{khintchine bilinear proof}
        \E\bigg[\bigg\|\sum_{i,j=1}^n \alpha_i\beta_j\y_{ij}\bigg\|_{p}^p\bigg] \leq
C_p \max\bigg\{\E\bigg[\bigg\|\bigg(\sum_{i=1}^n
\x_i\x_i^*\bigg)^{1/2}\bigg\|_p^p\bigg], \E\bigg[\bigg\|\bigg(\sum_{i=1}^n
\x_i^*\x_i\bigg)^{1/2}\bigg\|_p^p\bigg]\bigg\}
    \end{align}
    To apply part (a) again on these errors, define
        $\hat{\y}_{ij}=E_{1i} \otimes \y_{ij} \in \C^{n \times n} \otimes \cX$
equipped with the trace $n^{-1}\Tr \otimes \phi$.
    It follows that $\hat{\y}_{ij}\hat{\y}_{kl}^*=\1_{i=k}
E_{11}\otimes \y_{ij}\y_{il}^*$, so
\[E_{11} \otimes \sum_{i=1}^n \x_i\x_i^*=\sum_{j,l=1}^n \beta_j\bar\beta_l
\sum_{i=1}^n
E_{11} \otimes \y_{ij}\y_{il}^*=
\bigg(\sum_{j=1}^n \beta_j \sum_{i=1}^n \hat \y_{ij}\bigg)
\bigg(\sum_{l=1}^n \beta_l \sum_{k=1}^n \hat \y_{kl}\bigg)^*.\]
Then
\[\bigg\|\bigg(\sum_{i=1}^n\x_i\x_i^*\bigg)^{1/2}\bigg\|_p^p
=n\bigg\|\bigg(E_{11} \otimes \sum_{i=1}^n
\x_i\x_i^*\bigg)^{1/2}\bigg\|_p^p
=n\bigg\|\sum_{j=1}^n \beta_j \sum_{i=1}^n \hat\y_{ij}\bigg\|_p^p.\]
Applying part (a) with $\hat \x_j=\sum_i \hat{\y}_{ij} \in \C^{n \times n}
\otimes \cX$ in place of $\x_j \in \cX$, this is bounded as
\begin{align*}
\E\bigg[\bigg\|\bigg(\sum_{i=1}^n \x_i\x_i^*\bigg)^{1/2}\bigg\|_p^p\bigg]
&\leq C_pn\max\bigg\{\bigg\|\bigg(\sum_{j=1}^n
\hat\x_j\hat\x_j^*\bigg)^{1/2}\bigg\|_p^p,
\bigg\|\bigg(\sum_{j=1}^n \hat\x_j^*\hat\x_j\bigg)^{1/2}\bigg\|_p^p
\bigg\}\\
&=C_pn\max\bigg\{\bigg\|\bigg(E_{11} \otimes \sum_{i,j}
\y_{ij}\y_{ij}^*\bigg)^{1/2}\bigg\|_p^p,
\bigg\|\bigg(\sum_{i,j,k} E_{ik} \otimes
\y_{ij}^*\y_{kj}\bigg)^{1/2}\bigg\|_p^p\bigg\}.
\end{align*}
The first term is $n^{-1}\|(\sum_{i,j}\y_{ij}\y_{ij}^*)^{1/2}\|_p^p$.
For the second term,
we identify $\sum_{i,j,k} E_{ik} \otimes \y_{ij}^*\y_{kj}={\Y^\st}^*{\Y^\st}$ where
$\Y^\st=\sum_{i,j} E_{ji} \otimes \y_{ij} \in \C^{n \times n} \otimes \cX$, and
we apply $\|({\Y^\st}^*{\Y^\st})^{1/2}\|_p^p=\|\Y^\st\|_p^p$.
We have analogously $\E\|(\sum_i \x_i^*\x_i)^{1/2}\|_p^p
\leq C_pn\max(n^{-1}\|(\sum_{i,j}\y_{ij}^*\y_{ij})^{1/2}\|_p^p,\|\Y\|_p^p)$,
and combining these gives part (b).
   
Finally, part (c) follows from part (b)
and the decoupling result of Lemma \ref{lemma:decoupling}:
Since $\y_{ii}=0$ for all $i=1,\ldots,n$,
    \begin{align*}
        &\E\bigg[\bigg\|\sum_{i \neq j}\alpha_i\alpha_j\y_{ij}\bigg\|_p^p\bigg]
\leq C_p \E\bigg[\bigg\|\sum_{i \neq
j}\alpha_i\alpha_i^{\prime}\y_{ij}\bigg\|_p^p\bigg]
=C_p \E\bigg[\bigg\|\sum_{i,j=1}^n
\alpha_i\alpha_i^{\prime}\y_{ij}\bigg\|_p^p\bigg] 
\\
        &\leq C_p'\max\bigg\{\bigg\|\bigg(\sum_{i\neq j}
\y_{ij}\y_{ij}^*\bigg)^{1/2}\bigg\|_p^p, 
        \bigg\|\bigg(\sum_{i\neq j}\y_{ij}^*\y_{ij}\bigg)^{1/2}\bigg\|_p^p, 
        n\|\Y\|_p^p,n\|\Y^\st\|_p^p\bigg\}.
    \end{align*}
\end{proof}

\section{Fluctuation averaging}\label{Fluctuation averaging}

In this section, we prove Lemma \ref{lemma:fluctuationavg}. The proof is
analogous to the argument in the scalar setting of
\cite[Lemma A.2]{FanJohnstoneSun2018Spiked}. For part (c), we
will apply the following
combinatorial lemma from \cite{FanJohnstoneSun2018Spiked}.

    \begin{lem}[\cite{FanJohnstoneSun2018Spiked}, Lemma
A.3]\label{lemma:FAcombinatorial}
        Fix $l\geq 1$. For each $a=1,\ldots,l$, let $B^a=(B_{ij}^a)_{i,j=1}^n
\in \R^{n\times n}$ satisfy
        \begin{align*}
            B_{ij}^a \geq 0, \quad\quad B_{ii}^a=0, \quad\quad \|B^a\|_F \leq 1
\qquad \text{ for all } i,j=1,\ldots,n.
        \end{align*}
       For $(\i,\j)=(i_1,...,i_l,j_1,...,j_l)
\in\{1,\ldots,n\}^{2l}$, let $T(\i,\j)$ be the set of elements of
$\{1,\ldots,n\}$ that appear exactly once in $(\i,\j)$. Then for some constant
$C_l>0$ and all $t\in\{0,\ldots,2l\}$,
    \begin{align*}
    \sum_{(\i,\j) \in \{1,\ldots,n\}^{2l}:|T(\i,\j)|=t}\;
\prod_{a=1}^lB_{i_aj_a}^a \leq C_l n^{t/2}
    \end{align*}
    \end{lem}

\begin{proof}[Proof of Lemma \ref{lemma:fluctuationavg}]
For part (a), fix any $\epsilon,D>0$ and $p \in [1,\infty)$.
Pick an even integer $l>p$ such that $\epsilon(l-1)>D+1$. Then by
monotonicity of $\|\cdot\|_p$ in $p$ (Lemma \ref{Holder's inequality}),
\begin{equation}\label{eq:FAmomentexpansion}
\E\left\|\sum_{i=1}^n u_i \x_i\right\|_p^l \leq
\E\left\|\sum_{i=1}^n u_i \x_i\right\|_l^l =
\E\phi\left(\left[\left(\sum_{i=1}^n u_i \x_i\right)\left(\sum_{i=1}^n \bar{u}_i
\x_i^*\right)\right]^{l/2}\right) =
\sum_{\i}u_\i\E\phi\left(\x_\i\right)
\end{equation}
where we denote
    \begin{align*}
        \i = (i_1,i_2,...,i_l), \quad \sum_{\i}=\sum_{i_1,...,i_l=1}^n, \quad
u_\i = u_{i_1}\bar{u}_{i_2}\cdots u_{i_{l-1}}\bar{u}_{i_l}, \quad \x_\i =
\x_{i_1}\x_{i_2}^*\cdots\x_{i_{l-1}}\x_{i_l}^* =
\prod_{a=1}^l\widetilde{\x}_{i_a}
    \end{align*}
and write $\widetilde{\x}_{i_a} = \x_{i_a}$ if $a$ is odd and
$\widetilde{\x}_{i_a}=\x_{i_a}^*$ if $a$ is even.
Here and below, the product over $a=1,\ldots,l$ is non-commutative, and should be
understood in the ordered sense.

For fixed $\i$, let
$T \equiv T(\i) \subseteq \{1,\ldots,n\}$ be the indices that appear exactly
once in $\i=(i_1,\ldots,i_l)$. Using that $\{\E_j,\cQ_j\}$ commute, we have the identity
    \begin{align*}
        \x=\left(\prod_{j\in T}(\E_j+\cQ_j)\right)[\x] =
\sum_{S\subseteq T}\E_{T\setminus S}\cQ_S[\x].
    \end{align*}
(In the case $T=\emptyset$, this is the trivial identity $\x=\x$.)
Applying this to each $\widetilde{\x}_{i_a}$,
    \begin{align*}
        \x_\i=\sum_{S_1,\ldots,S_l \subseteq
T} \x(S_1,\ldots,S_l),
\quad\quad \x(S_1,\ldots,S_l) = \prod_{a=1}^l\E_{T\setminus
S_a}\cQ_{S_a}[\widetilde{\x}_{i_a}].
    \end{align*}
    By H\"older's inequality (Lemma \ref{Holder's inequality}) and
(\ref{eq:jensencondexp}),
    \begin{align*}
        |\phi(\x(S_1,\ldots,S_l))| \leq
\prod_{a=1}^l \|\E_{T\setminus S_a}\cQ_{S_a}[\widetilde{\x}_{i_a}]\|_l
\leq \prod_{a=1}^l\E_{T\setminus S_a}\|\cQ_{S_a}[\widetilde{\x}_{i_a}]\|_l.
    \end{align*}
Let us write $S \setminus i$ for $S$ removing $i$ if $i \in S$, or
for $S$ itself if $i \notin S$. Since $\E_{i_a}[\widetilde \x_{i_a}]=0$ by
assumption, we have $\cQ_{i_a}[\widetilde \x_{i_a}]=\widetilde \x_{i_a}$, so
$\cQ_{S_a}[\widetilde{\x}_{i_a}]=\cQ_{S_a\setminus i_a}[\widetilde{\x}_{i_a}]$.
Then the given condition (\ref{eq:QSassump3}) and
Lemma \ref{lemma:domination}(c) (in the setting of scalar random variables) imply 
$\E_{T\setminus S_a}\|\cQ_{S_a}[\widetilde{\x}_{i_a}]\|_l
 \prec n^{-\alpha - \beta|S_a \setminus i_a|}$, for each fixed $l$
uniformly over $i_a \in \{1,\ldots,n\}$ and over $S_a \subseteq T \subseteq
\{1,\ldots,n\}$ with $|T| \leq l$. Thus, multiplying across $a=1,\ldots,l$
and taking the full expectation, for all $n \geq n_0(l,\epsilon)$ we have
\begin{equation}\label{eq:FAbound}
|\E \phi(\x(S_1,\ldots,S_l))| \leq 
\E |\phi(\x(S_1,\ldots,S_l))| \leq
n^{-\alpha l-\beta\sum_{a=1}^l |S_a \setminus i_a|+\epsilon}.
\end{equation}
Now consider any $a \in \{1,\ldots,l\}$ such that $i_a \in T$. Observe that
\begin{itemize}
\item $\x(S_1,\ldots,S_l)=0$ unless $i_a \in S_a$. Indeed, if instead
$i_a \in T\setminus S_a$, then the assumption $\E_{i_a}[\widetilde
\x_{i_a}]=0$ implies $\E_{T\setminus S_a}\cQ_{S_a}[\widetilde{\x}_{i_a}]=0$.
\item $\E\phi(\x(S_1,\ldots,S_l))=0$ unless also $i_a \in S_b$ for some
$b \neq a$. Indeed, if instead $i_a \in T \setminus S_b$ for every $b \neq a$,
then $\E_{T\setminus S_b}\cQ_{S_b}[\widetilde{\x}_{i_b}]
=\E_{i_a}\E_{T\setminus (S_b\cup\{i_a\})}
\cQ_{S_b}[\widetilde{\x}_{i_b}]$ is $\mathscr{G}_{i_a}$-measurable, hence
    \begin{align*}
        \E_{i_a} \x(S_1,\ldots,S_l)
        = \prod_{b<a} \E_{T\setminus S_b}
\cQ_{S_b}[\widetilde{\x}_{i_b}] \cdot
\underbrace{\E_{i_a}\E_{T\setminus S_a}\cQ_{S_a}[\widetilde{\x}_{i_a}]}_{=0}
\cdot \prod_{b>a} \E_{T\setminus S_b}\cQ_{S_b}[\widetilde{\x}_{i_b}]=0
    \end{align*}
where the middle term is 0 because $\E_{i_a}\cQ_{S_a}=0$ for $i_a \in S_a$.
Then by linearity of $\phi$,
$\E\phi(\x(S_1,\ldots,S_l))=\phi(\E \x(S_1,\ldots,S_l))=0$.
\end{itemize}
Thus, if $\E \phi(\x(S_1,\ldots,S_l)) \neq 0$, then each $i_a \in T$
must appear in both $S_a$ and
some set $S_b$ for $b \neq a$, where $i_a \neq i_b$ because $i_a$
appears only once in $\i=(i_1,\ldots,i_l)$ by definition of $T$. So 
$\sum_{a=1}^l |S_a \setminus i_a| \geq |T|$. Applying this to
(\ref{eq:FAbound}) and then to (\ref{eq:FAmomentexpansion}),
since for each fixed $\i$ and $T \equiv T(\i)$
the number of choices of subsets
$S_1,\ldots,S_l \subseteq T$ is at most a constant $C_l$,
\begin{equation}\label{eq:FAbound2}
        \E\left\|\sum_{i=1}^n u_i\x_i\right\|_p^l \leq
C_ln^{-\alpha l+\epsilon} \sum_\i |u_\i| n^{-\beta|\mathcal{T(\i)}|}
\leq C_ln^{-\alpha l+\epsilon} \|\u\|_\infty^l \sum_{t=1}^l n^{-\beta t}
\cdot |\{\i:T(\i)=t\}|.
\end{equation}
By scale invariance of the statement of the lemma,
let us assume without loss of generality that $\|\u\|_\infty=n^{-1}$.
Applying this and the bound $|\{\i:T(\i)=t\}| \leq C_l'n^{t+(l-t)/2}$ for a
constant $C_l'>0$, we get (for different constants $C_l,C_l'>0$)
\[\E\left\|\sum_{i=1}^n u_i\x_i\right\|_p^l \leq C_ln^{-\alpha l+\epsilon} 
\sum_{t=1}^l(n^{-\beta})^t(n^{-1/2})^{l-t}\leq C_ln^{-\alpha l+\epsilon}\sum_{t=1}^l (n^{-\beta}+n^{-1/2})^l
\leq C_l'n^{-(\alpha+\beta') l+\epsilon}\]
where $\beta'=\min\{1/2,\beta\}$. Then by Markov's inequality under our choices $l>p$ and
$\epsilon(l-1)>D+1$,
    \begin{align*}
        \mathbb{P}\left(\left\|\sum_{i=1}^n u_i\x_i\right\|_p \geq
n^{-\alpha-\beta'+\epsilon}\right) \leq n^{(\alpha+\beta'-\epsilon)l}\,
\E\left\|\sum_{i=1}^n u_i\x_i\right\|_p^l \leq C_l'n^{-\epsilon(l-1)}<n^{-D}
    \end{align*}
for all $n \geq n_0(l,\epsilon,D)$. Here $l$ depends only on $(p,\epsilon,D)$,
showing $\sum_i u_i\x_i \prec n^{-\alpha-\beta'}$ as claimed in part (a).

The argument for part (b) is the same, until the analysis of
(\ref{eq:FAbound2}) where we apply a different
counting argument: By scale invariance, we may consider 
without loss of generality $\u\in\C^n$ with $\|\u\|_2=1$. Under the given
condition (\ref{eq:QSassump1}), specializing to $\beta=1/2$,
the first inequality of (\ref{eq:FAbound2}) becomes
\[\E\left\|\sum_{i=1}^n u_i\x_i\right\|_p^l \leq
C_ln^{-\alpha l+\epsilon} \sum_\i |u_\i| n^{-|\mathcal{T(\i)}|/2}.\]
Now let $\pi(\i)$ be the partition of $\{1,\ldots,l\}$ induced by coincidence of
indices in $\i$, i.e.\ $a,b$ belong to the same block of $\pi(\i)$ if and only
if $i_a=i_b$. Then $|T(\i)| \equiv |T(\pi(\i))|$ is the number of singleton
blocks of $\pi(\i)$, depending on $\i$ only via $\pi(\i)$, so
we may write the above as
\begin{equation}\label{eq:FAtmp}
\E\left\|\sum_{i=1}^n u_i\x_i\right\|_p^l \leq
C_ln^{-\alpha l+\epsilon} \sum_\pi n^{-|T(\pi)|/2}
\sum_{\i:\pi(\i)=\pi} |u_\i|.
\end{equation}
Here $\sum_{\i:\pi(\i)=\pi} |u_\i|$ is a sum over one
(distinct) index $i \in \{1,\ldots,n\}$ for each block $P$ of $\pi$, for which
we have $\sum_{\i:\pi(\i)=\pi} |u_\i| \leq \prod_{P \in \pi} \sum_{i=1}^n
|u_i|^{|P|}$. Under our assumed normalization $\|\u\|_2=1$, we have
    \begin{align*}
        \sum_{i=1}^n |u_i| \leq \sqrt{n}, \quad\quad \sum_{i=1}^n |u_i|^k
\leq \|\u\|_\infty^{k-2} \sum_{i=1}^n |u_i|^2 \leq 1 \quad \text{ for any }
k \geq 2.
    \end{align*}
Thus $\sum_{\i:\pi(\i)=\pi} |u_\i| \leq n^{|T(\pi)|/2}$. Applying this to
(\ref{eq:FAtmp}), we have
$\E\|\sum_i u_i\x_i\|_p^l \leq C_l'n^{-\alpha l+\epsilon}$. The proof of (b) now follows 
by the same Markov inequality argument as in part (a).

Part (c) is also similar: By scale invariance, we may consider $U \in \C^{n \times
n}$ such that
$\sum_{i \neq j} |u_{ij}|^2=1$. Fixing $p \in [1,\infty)$ and picking an even
integer $l>p$,
    \begin{align*}
        \E\Bigg\|\sum_{i\neq j} u_{ij} \x_{ij}\Bigg\|_p^l \leq
\E\phi\Bigg(\Bigg[\Bigg(\sum_{i\neq j} u_{ij} \x_{ij}\Bigg)\Bigg(\sum_{i\neq j}
\bar{u}_{ij} \x_{ij}^*\Bigg)\Bigg]^{l/2}\Bigg) = \sum_{\i,\j}u_{\i,\j}\E\phi(\x_{\i,\j})
    \end{align*}
    where
    \begin{align*}
        &(\i,\j) = (i_1,\ldots,i_l,j_1,\ldots,j_l),\quad\quad\sum_{\i,\j}=\sum_{i_1\neq j_1}\cdots\sum_{i_l\neq j_l}\\
        &u_{\i,\j} = u_{i_1j_1}\bar{u}_{i_2j_2}\cdots
u_{i_{l-1}j_{l-1}}\bar{u}_{i_lj_l},\quad\quad\x_{\i,\j}=\x_{i_1j_1}\x_{i_2j_2}^*\cdots\x_{i_{l-1}j_{l-1}}\x_{i_lj_l}^*=\prod_{a=1}^l
\widetilde \x_{i_aj_a}.
    \end{align*}
    Define $T \equiv T(\i,\j)$ as the indices that appear exactly once in
the combined index list $(\i,\j)=(i_1,\ldots,i_l,j_1,\ldots,j_l)$. Then,
expanding
\[\x_{\i,\j}=\sum_{S_1,\ldots,S_l \subseteq T} \x(S_1,\ldots,S_l)
=\sum_{S_1,\ldots,S_l \subseteq T} 
\prod_{a=1}^l \E_{T \setminus S_a}\cQ_{S_a}[\widetilde \x_{i_a,j_a}],\]
the same arguments as above using the conditions $\E_{i_a}[\widetilde
\x_{i_a,j_a}]=\E_{j_a}[\widetilde \x_{i_a,j_a}]=0$ and (\ref{eq:QSassump2}) show
\begin{itemize}
\item $|\E \phi(\x(S_1,\ldots,S_l))| \leq n^{-\alpha l-\sum_{a=1}^l |S_a
\setminus \{i_a,j_a\}|/2+\epsilon}$.
\item If $i_a \in T$ (or $j_a \in T$), then $\x(S_1,\ldots,S_l)=0$ unless
$i_a \in S_a$ (resp.\ $j_a \in S_a$).
\item If $i_a \in T$ (or $j_a \in T$), then $\E \phi(\x(S_1,\ldots,S_l))=0$
unless furthermore $i_a \in S_b$ (resp.\ $j_a \in S_b$) for some $b \neq a$.
\end{itemize}
Thus $\sum_{a=1}^l |S_a \setminus \{i_a,j_a\}| \geq |T|$, so we obtain
similarly as part (b)
\[\E\left\|\sum_{i\neq j} u_{ij} \x_{ij}\right\|_p^l
\leq C_l n^{-\alpha l+\epsilon}\sum_\pi n^{-|T(\pi)|/2}
\sum_{\i,\j:\pi(\i,\j)=\pi} |u_{\i,\j}|\]
where $\pi$ is the partition of $\{1,\ldots,2l\}$ induced by coincident indices
of the combined list $(\i,\j)$, and $|T(\pi)|$ is the number of singleton
blocks of $\pi$. Lemma \ref{lemma:FAcombinatorial} applied with
\[B_{ij}^a=\begin{cases} |u_{ij}| & \text{ if } i \neq j \\
0 & \text{ if } i=j \end{cases} \qquad \text{ for all } a=1,\ldots,l\]
shows $\sum_{\pi:|T(\pi)|=t} \sum_{\i,\j:\pi(\i,\j)=\pi} |u_{\i,\j}| \leq C_l
n^{t/2}$. Then the proof of part (c) is concluded by the same Markov
inequality argument as in part (a-b).
\end{proof}

\section{von Neumann algebras and non-commutative $L^p$ spaces}\label{appendix:background}

We collect here several pieces of background on von Neumann algebras and
non-commutative $L^p$ spaces that are needed in our main arguments. We refer to
\cite{nelson1974notes}, \cite[Section 1]{davies1992noncommutative},
and \cite[Chapter 14]{pisier2016martingales} for additional discussion.
Throughout, $\cX$ is a (finite) von Neumann algebra with faithful, normal,
tracial state $\phi$ as in Section \ref{sec:preliminaries}. 

\begin{lem}[Conditional expectation.
\cite{brown2008c}, Lemma 1.5.11]\label{conditional expectation}
Let $\cB \subseteq \cX$ be a von Neumann subalgebra.
Then there exists a unique linear map $\phi^\cB:\cX \to \cB$
(the $\phi$-invariant conditional expectation) that satisfies the following:
    \begin{itemize}
        \item $\phi^\cB$ is normal, contractive in the operator norm,
and completely positive.
        \item For any $\y_1,\y_2 \in \cB$ and $\x \in \cX$, we have 
        $\phi^\cB[\y_1\x\y_2]=\y_1\phi^\cB[\x]\y_2$.
        \item For any $\y \in \cB$, $\phi(\y)=\phi(\phi^\cB[\y])$.
    \end{itemize}
\end{lem}

Defining $\|\x\|_p=\phi(|x|^p)^{1/p}$, the space $L^p(\cX)$ is the
Banach space completion of $\cX$ under $\|\cdot\|_p$. We set
$\|\x\|_\infty \equiv \|\x\|_\op$ and $L^\infty(\cX) \equiv \cX$. These spaces
$L^p(\cX)$ may be continuously embedded into a common space of (unbounded)
densely-defined operators affiliated to $\cX$ --- we refer to
\cite{nelson1974notes} or \cite[Section 1]{davies1992noncommutative} for this
construction.

\begin{lem}[Non-commutative $L^p$-spaces.
\cite{pisier2016martingales} Theorem 14.1,
\cite{davies1992noncommutative} Proposition 1.1]\label{Holder's inequality}
For each $p \in [1,\infty)$, $\|\x\|_p=\phi(|x|^p)^{1/p}$ defines a complete
norm on $L^p(\cX)$, satisfying
\[\|\x\|_p \leq \|\y\|_p \text{ for all } \x,\y \in L^p(\cX) \text{ with }
0 \leq \x \leq \y, \qquad |\phi(x)| \leq \|\x\|_1 \text{ for all } \x \in
L^1(\cX).\]
\begin{enumerate}[(a)]
\item (H\"older's inequality)
For any $1 \leq p,q,r \leq \infty$ with $1/p+1/q=1/r$, these norms satisfy
\[\|\x\y\|_r \leq \|\x\|_p\|\y\|_q \text{ for all } \x \in L^p(\cX),
\;\y \in L^q(\cX).\]
In particular, $\|\x\|_p \leq \|\x\|_q$ for any $1 \leq p \leq q \leq \infty$
so $L^q(\cX) \subseteq L^p(\cX)$, and $\|\x\y\|_p \leq \|\x\|_\op\|\y\|_p$.
\item (Duality) For each $p \in [1,\infty)$, let $q$ be such that
$\frac{1}{p}+\frac{1}{q}=1$. Then
the map $\y \in L^q \mapsto \ell_\y \in (L^p)^*$ given by
$\ell_\y(\x)=\phi(\x\y)$ is a Banach space isomorphism between
$L^q$ and the dual $(L^p)^*$ of $L^p$.
\end{enumerate}
\end{lem}

\begin{lem}[$L^p$-contractivity of conditional expectation]\label{lemma:Lpcontraction}
Let $\cB \subseteq \cX$ be a von Neumann subalgebra, and let $\phi^\cB:\cX \to
\cB$ be the unique $\phi$-invariant conditional expectation. Then for any $p \in
[1,\infty)$ and $\x \in \cX$,
\[\|\phi^\cB(\x)\|_p \leq \|\x\|_p.\]
\end{lem}
\begin{proof}
For $p=1$, let $\y \in \cB$ be the projection
operator for which $\y\phi^\cB(\x)=|\phi^\cB(\x)|$. Then
\[\|\phi^\cB(\x)\|_1=\phi(\y\phi^\cB(\x))
=\phi(\y\x) \leq \|\y\|_\op \|\x\|_1=\|\x\|_1.\]
Similarly for $p \in (1,\infty)$, by the density of $\cX$ in $L^q(\cX)$ and
the above $L^p$-$L^q$ duality on $\cX$ as well as
on its subalgebra $\cB$,
\[\|\phi^\cB(\x)\|_p
=\sup_{\y \in \cB:\|\y\|_q=1} \phi(\y\phi^\cB(\x))
=\sup_{\y \in \cB:\|\y\|_q=1} \phi(\y\x)
\leq \sup_{\y \in \cX:\|\y\|_q=1} \phi(\y\x)=\|\x\|_p.\]
\end{proof}

\begin{lem}[Riesz-Thorin interpolation. \cite{davies1992noncommutative},
Proposition 1.6]\label{lemma:rieszthorin}
Suppose, for some $p_0,q_0,p_1,q_1 \in [1,\infty]$, that
$T:\cX \to L^{q_0}(\cX,\phi) \cap L^{q_1}(\cX,\phi)$ is a linear map
satisfying
\[\|T\x\|_{q_0} \leq M_0\|\x\|_{p_0},
\qquad \|T\x\|_{q_1} \leq M_1\|\x\|_{p_1}\]
for all $\x \in \cX$ and some $M_0,M_1>0$. If
$\frac{1}{p_\theta}=\frac{1-\theta}{p_1}+\frac{\theta}{p_2}$
and $\frac{1}{q_\theta}=\frac{1-\theta}{q_1}+\frac{\theta}{q_2}$,
then for all $\x \in \cX$,
\[\|T\x\|_{q_\theta} \leq M_0^{1-\theta}M_1^\theta\|\x\|_{p_\theta}.\]
\end{lem}
\vspace{\baselineskip}

\subsection*{Acknowledgments}

Z.F.\ would like to thank Cheng Mao, Jiaming Xu, and Yihong Wu for the
collaboration \cite{fan2023spectral} that inspired this work. A preliminary
version of these results were presented at the Random Matrices and
Applications workshop at ICERM, and we would like to thank the workshop
participants for their helpful feedback. This research is supported in part by
NSF DMS-2142476 and a Sloan Research Fellowship.

\bibliographystyle{plain}
\bibliography{citations}

\end{document}